\documentclass[a4paper,reqno]{amsart}

\usepackage[T1]{fontenc}
\usepackage[utf8]{inputenc}
\usepackage[english]{babel}

\usepackage{amsmath,amssymb,amsthm} %AMS packages
\usepackage{mathtools}   %more tools for mathematical typesetting, based on amsmath
\usepackage{tensor}

\usepackage{bm}   %bold math expressions

\usepackage{tikz}
\usetikzlibrary{matrix,arrows,decorations.pathmorphing}
\usepackage{tikz-cd}
\usepackage{eucal}
\usepackage{mathrsfs}
\usepackage{dsfont}

\usepackage{enumitem}   %provides customization options for "enumerate"

\usepackage{stmaryrd} %for the commands \llbracket and \rrbracket

\theoremstyle{plain} %defaul style: bold name, math font text
\newtheorem{thm}{Theorem}[section]
\newtheorem{prop}[thm]{Proposition}
\newtheorem{lmm}[thm]{Lemma}
\newtheorem{cor}[thm]{Corollary}

\theoremstyle{definition} %bold name, normal font text
\newtheorem{dfn}[thm]{Definition}

\theoremstyle{remark} %italic name, normal font text
\newtheorem{rmk}[thm]{Remark}

\newtheorem{notat}{Notation}

\newcommand{\E}{\mathbb{E}}
\newcommand{\N}{\mathbb{N}}
\newcommand{\Z}{\mathbb{Z}}
\newcommand{\Q}{\mathbb{Q}}

\newcommand{\C}{\mathbb{C}}

\newcommand{\pr}{\mathbb{P}}   %projective space
\newcommand{\cO}{\mathcal{O}} %structure sheaf
\newcommand{\fo}{\mathbf{1}}
\newcommand{\z}{\zeta}
\newcommand{\cD}{{\mathcal D}}
\newcommand{\ve}{{\varepsilon}}
\newcommand{\cR}{{\mathcal R}}
\newcommand{\ka}{{\kappa}}

\title[Equivariant CCRC for ADE-orbifolds]{Equivariant Cohomological Crepant Resolution Conjecture for ADE-orbifolds}
\author{Qike Li}

\address[Qike Li]{Dipartimento di Matematica, Informatica e Geoscienze, Università degli Studi di Trieste, 
via Valerio 12/1, 34127 Trieste, Italy, QIKE.LI@phd.units.it}

\author{Fabio Perroni}

\address[Fabio Perroni]{Dipartimento di Matematica, Informatica e Geoscienze, Università degli Studi di Trieste, 
via Valerio 12/1, 34127 Trieste, Italy, fperroni@units.it}

\begin{document}

%-------------------------------------------------
\begin{abstract}
Let $G\subset \mathrm{SL}_2(\mathbb C)$ be a non-trivial finite subgroup and let
\[
\pi \colon X\longrightarrow \mathbb C^2/G
\]
be the minimal resolution. We prove the $\mathbb C^*$-equivariant
Cohomological Crepant Resolution Conjecture for the ADE orbifold
$[\mathbb C^2/G]$. More precisely, after specializing the quantum
parameters at suitable  roots of unity, we
prove that the Bryan--Gholampour transformation induces an
isomorphism between the $\mathbb C^*$-equivariant quantum cohomology
of $X$ and the $\mathbb C^*$-equivariant Chen--Ruan cohomology of
$[\mathbb C^2/G]$.

Our proof is direct and is based on the McKay correspondence, ADE root
systems, and character theory. 
Following the character-theoretic form of the Bryan–Gholampour change of variables, 
we use the McKay correspondence to diagonalize the transformation and to reduce the compatibility 
of the products to a uniform root-system identity. 
In type \(A\), this reduction admits a discrete Fourier realization. 
In type \(D\), the cyclic subgroup of the binary dihedral group leads instead to finite sine transforms, 
while the exceptional cases \(E_6,E_7,E_8\) are treated by exact matrix computations over the corresponding cyclotomic fields.

As a preliminary result, for a symplectic complex vector space $V$
and a finite subgroup $G\subset \mathrm{Sp}(V)$, we give an explicit
presentation of the $\mathbb C^*$-equivariant Chen--Ruan product of
$[V/G]$ and identify the resulting algebra with the Rees algebra of
the center of the group algebra with respect to the age filtration.
In particular, the equivariant Chen--Ruan algebra interpolates between
the center of the group algebra and its associated graded algebra,
the latter recovering the ordinary Chen--Ruan cohomology.
\end{abstract}

\date{\today}
\maketitle

%-------------------------------------------------
\section{Introduction}\label{Section_introduction}

Let $G$ be a finite group acting on a smooth variety $Y$. The quotient
$Y/G$ is typically singular, whereas the quotient stack $[Y/G]$
retains the information on the stabilizers and is a smooth
Deligne--Mumford stack. A natural problem is to compare invariants of
$[Y/G]$ with those of a resolution
\[
\pi \colon X\longrightarrow Y/G
\]
of the coarse quotient. When $Y/G$ is Gorenstein, the resolutions
which are most closely related to the orbifold geometry are the
crepant ones, namely those for which the canonical class is preserved.

A fundamental invariant on the orbifold side is the Chen--Ruan
cohomology ring
\[
\left( H^*_{\mathrm{CR}}([Y/G]) \, , \cup_{{\rm CR}} \right) \, ,
\]
introduced in \cite{CR04}; see also \cite{FG03,ALR07}. As a vector
space, it is the cohomology of the inertia stack, with a degree shift
determined by the age. Besides the untwisted sector, it contains
twisted sectors associated with non-trivial stabilizers, while the
Chen--Ruan cup product incorporates the geometry of their
intersections.

Ideas originating in theoretical physics suggest that the quantum
cohomology of an orbifold should be closely related to the quantum
cohomology of a crepant resolution of its coarse moduli space. This
led to several versions of the Crepant Resolution Conjecture, starting
with Ruan \cite{RuanCCRC} and subsequently refined and generalized in
\cite{BG09,CIT09,CR13}. In its cohomological form, the conjecture
predicts, roughly speaking, that the Chen--Ruan cohomology ring of an
orbifold can be recovered from the small quantum cohomology of a
crepant resolution after analytic continuation in the quantum
parameters and specialization at suitable values.

In this paper we study this correspondence for the surface quotient
singularities
\[
\mathbb C^2/G,
\qquad
G\subset \mathrm{SL}_2(\mathbb C)
\]
finite and $\not= \{ e \}$. These are the Kleinian, or ADE, surface singularities. They
admit a unique minimal resolution
\[
\pi:X\longrightarrow\mathbb C^2/G \, ,
\]
which is crepant, and the configuration of the irreducible components of the exceptional divisor is
described by a finite Dynkin diagram of ADE type.
Through the McKay correspondence, the irreducible components of the
exceptional divisor are related to the non-trivial irreducible
representations of $G$; see, for example,
\cite{McKay,GSVMcKay, CS98,ReidMcKay, DolgachevMcKay}. Thus the geometry of $X$, the representation
theory of $G$, and the corresponding ADE root system are naturally
linked.

We work equivariantly with respect to the scalar action of
$\mathbb C^*$ on $\mathbb C^2$, which lifts to the minimal resolution in such a way that $\pi$ is equivariant.
The $\mathbb C^*$-equivariant quantum product on $X$ was computed by
Bryan and Gholampour in \cite{bryan2008root} in terms of the corresponding ADE
root system. 
The precise form of the specialization of the quantum parameters has
emerged through several developments of the Crepant Resolution
Conjecture. Ruan's original formulation suggested specializing the
quantum parameters associated with exceptional curves to $-1$.
Subsequent examples showed that, in general, this prescription has to
be replaced by specialization at suitable roots of unity; see
\cite{Perroni07,BGP08}. This phenomenon was later incorporated into more
general formulations of the conjecture and received a conceptual
explanation through the work of Coates--Iritani--Tseng, Coates--Ruan and Iritani
\cite{CIT09, CR13,Iritani10}.

Our main result is a direct verification of this prediction for every
finite subgroup $G\subset \mathrm{SL}_2(\mathbb C)$.
Let ${\rm Irr} (G)$ be the set  of characters of the irreducible representations of $G$.
The McKay correspondence yields classes \(\alpha_\chi \in H^*_{\C^*}(X) \), 
for $\chi \in {\rm Irr} (G)$ non-trivial, which, together with the unit, form a basis
 of $H^*_{\C^*}(X)$ \cite{GSVMcKay} (see also \cite{McKay, ReidMcKay, DolgachevMcKay}).
On the other hand,  $H^*_{\mathbb C^*,\mathrm{CR}} \bigl([\mathbb C^2/G]\bigr)$
has a basis $\{ \fo_{[g]} \, | \, g\in \overline{G}\}$ indexed by conjugacy classes $[g]$ of $G$, where $\overline{G}$
is a set of representatives of such classes.
Following \cite{bryan2008root, BG09}, we specialize the quantum parameters corresponding to the irreducible exceptional curves
as follows:
\begin{equation}\label{BGq}
q_\chi = \exp \left( \frac{2\pi i \chi (1)}{|G|} \right) \, , \qquad \chi \in {\rm Irr} (G) \quad \mbox{non-trivial} \, .
\end{equation}
Let us write
\begin{equation}\label{qccr}
QH^*_{\mathbb C^*}(X)_{\pi}
\end{equation}
for the corresponding specialized equivariant quantum cohomology ring, 
which is known as the (equivariant) \textit{quantum corrected cohomology ring} of $X$ \cite{RuanCCRC}.
The Bryan--Gholampour transformation is the linear map
\[
L \colon H^*_{\C^*}(X) \longrightarrow H^*_{\mathbb C^*,\mathrm{CR}} \bigl([\mathbb C^2/G]\bigr)
\]
defined by
\begin{equation}\label{BGLintro}
L (\alpha_\chi) = \sum_{g\in \overline{G}} \sqrt{\chi_V(g) -2} \chi (g) \fo_{[g]} \, , \qquad \forall \chi \in {\rm Irr} (G)
\quad \mbox{non-trivial} \, ,
\end{equation}
$L(1)=1$, and extended $\Q[t]$-linearly \cite{bryan2008root}, where $V$ is the  representation of $G$ on $\C^2$
given by the inclusion $G\subset {\rm SL}_2(\C)$ and $\chi_V$ is  its character. 
Here and throughout
the paper, the square root is chosen according to the
Bryan--Gholampour convention, namely as the positive multiple of $i$.

\begin{thm}\label{mainthm}
Let $G\subset \mathrm{SL}_2(\mathbb C)$ be a non-trivial finite subgroup and let
$\pi \colon X\to\mathbb C^2/G$ be the minimal resolution. The
Bryan--Gholampour transformation induces an isomorphism of
$\mathbb C[t]$-algebras
\[
L:
QH^*_{\mathbb C^*}(X)_{\pi}
\longrightarrow
\left( H^*_{\mathbb C^*,\mathrm{CR}}
\bigl([\mathbb C^2/G]\bigr) \, , \, \cup_{\C^*, {\rm CR}} \right) \, .
\]
\end{thm}

For singularities of type $A$ and $D$, the statement of the theorem also follows from the more general results of 
\cite{CCIT09} and \cite{Hu13}, respectively. Our proofs are independent of those approaches and are based on 
explicit algebraic calculations, the McKay correspondence, and finite Fourier analysis. 
To the best of our knowledge, the equivariant Cohomological Crepant Resolution Conjecture for the surface orbifolds of types 
\(E_6\), \(E_7\), and \(E_8\) had not previously been proved. The proof in type \(A\) given in \cite{CCIT09} 
relies on mirror symmetry. In this connection we also mention the recent work of Brini, Ma and Strachan 
\cite{BMS25}, who obtained a mirror-symmetric description of the equivariant quantum cohomology of all ADE resolutions.

The character-theoretic form of the Bryan--Gholampour transformation is
closely tied to the classical McKay correspondence. As observed in
\cite[Sec.~5]{bryan2008root}, the Cartan pairing can be expressed in terms of
the irreducible characters of $G$ and the function $\chi_V-2$. Equivalently,
this is the classical observation of McKay \cite{McKay} that the columns of
the character table diagonalize the affine Cartan matrix, with eigenvalues
$2-\chi_V(g)$. We use this point of view systematically.
For every conjugacy class \([g]\), we introduce a vector \(v_g\) in the finite root space, 
defined from the corresponding column of the character table. 
For non-trivial \([g]\), these vectors form the 
 \textit{McKay coordinate basis}. If
$z_g=|C_G(g)|$ is the order of the centralizer of $g$ in $G$, and $\Delta_g=\sqrt{\chi_V(g)-2}$, then
\[
L(v_g)=z_g\Delta_g\,\fo_{[g]}.
\]
Thus the Bryan--Gholampour change of variables becomes diagonal in McKay
coordinates. The same coordinates also diagonalize the Cartan pairing, and
this makes the compatibility of the classical term of the quantum product
with the identity-sector contribution of the Chen--Ruan product uniform for
all finite subgroups of $\mathrm{SL}_2(\mathbb C)$.

What remains is the comparison of the quantum correction with multiplication
of non-trivial conjugacy-class sectors. We formulate this comparison as a
single identity in McKay coordinates. The subsequent ADE analysis verifies
this identity in different concrete realizations. In type $A_{n-1}$, the
McKay coordinate vectors are precisely the discrete Fourier basis vectors of the cyclic
group. In these coordinates
\[
L(v_a)=n\Delta_a\,\fo_{[a]},
\qquad
\Delta_a=2i\sin\left(\frac{\pi a}{n}\right), \qquad 1\le a\le n-1,
\]
and, away from the zero Fourier mode, the specialized quantum product respects
addition of Fourier indices modulo $n$. The required identity is then reduced
to character orthogonality and a finite Fourier--cotangent identity.

In type \(D\), the character-theoretic reduction is combined with the cyclic subgroup of the binary dihedral group. 
The root-system calculation is most naturally performed in an orthonormal basis, 
while the corresponding Chen–Ruan classes are expressed through finite sine transforms, 
together with two additional components associated with the non-cyclic conjugacy classes.

For the exceptional groups, the character table gives the same uniform change
of coordinates, but no comparably simple cyclic realization is available. We
therefore verify the remaining product-compatibility identities exactly in the
relevant cyclotomic fields. This gives the cases $E_6$, $E_7$, and $E_8$
without relying on numerical approximations.

There is a second ingredient of the paper which is independent of the
ADE classification. Let $V$ be a symplectic complex vector space,
let $G\subset \mathrm{Sp}(V)$ be a finite subgroup, and let
$T=\mathbb C^*$ act on $V$ by scalar multiplication. We compute
explicitly the $T$-equivariant Chen--Ruan product of $[V/G]$. If
$\fo_{[g]}$ denotes the class associated with the conjugacy class of
$g\in G$, we obtain
\[
\fo_{[g]}\cup_{T,\mathrm{CR}}\fo_{[h]}
=
\sum_{[k]}
N(g,h,k)\,
t^{a(g)+a(h)-a(k)}\fo_{[k]},
\]
where $N(g,h,k)$ are the structure constants for multiplication of
conjugacy-class sums in the center of the group algebra and $a(x)$ is the age of $x \in G$.

This formula admits a natural algebraic interpretation. Let
$F_\bullet Z(\mathbb Q[G])$ be the filtration of the center of the
group algebra induced by the age function. We prove that
\[
\left( H^*_{T,\mathrm{CR}}([V/G]) \, , \cup_{T, {\rm CR}} \right) 
\cong
\operatorname{Rees}_{F}
Z(\mathbb Q[G]).
\]
Consequently, the equivariant Chen--Ruan algebra provides a
deformation interpolating between the center of the group algebra and
its associated graded algebra: specialization at $t=1$ recovers
$Z(\mathbb Q[G])$, while specialization at $t=0$ gives
$\operatorname{gr}_F Z(\mathbb Q[G])$, which identifies with the
ordinary Chen--Ruan cohomology. This description is closely related,
up to a reindexing of the filtration, to the construction of
Ginzburg--Kaledin \cite{GK}.

These two ingredients reflect complementary aspects of the McKay
correspondence. On the orbifold side, the product is governed by conjugacy
classes and multiplication in the center of the group algebra. On the
resolution side, the quantum product is governed by the associated ADE root
system. The same character table which enters the Bryan--Gholampour change of
variables relates these two descriptions; the McKay coordinates introduced
below make this relation explicit.

We conclude by describing the organization of the paper.
In Section~2 we recall the equivariant Chen--Ruan product for global quotients
in the form needed in the sequel. In Section~3 we compute this product for
symplectic vector spaces and identify the resulting algebra with the Rees
algebra associated with the age filtration of the center of the group algebra.
In Section~4 we introduce McKay coordinates for an arbitrary finite subgroup
of $\mathrm{SL}_2(\mathbb C)$, rewrite the Bryan--Gholampour transformation in
these coordinates, and reduce the comparison of products to a uniform
root-system identity. Section~5 treats cyclic groups and proves this identity
in type $A$ using discrete Fourier coordinates. 
Section 6 is devoted to binary dihedral groups and type \(D\), 
where the comparison is carried out using root coordinates and finite sine transforms.
 Section 7 treats the three exceptional cases \(E_6\), \(E_7\), and \(E_8\) 
within a common computational framework. 
The finite Fourier–cotangent identities used in the classical cases, together with their proofs, are collected in the appendices.

\bigskip

\subsection*{AI usage}
The authors used generative AI tools during the preparation of this manuscript to assist with language editing, 
mathematical computations and their verification, and the development and verification of some mathematical arguments. 
All results, computations, and proofs presented in the paper were subsequently independently reviewed and verified by the authors, 
who take full responsibility for the contents of the article.

\subsection*{Acknowledgments} 
We are grateful to Barbara Fantechi for encouraging us to work on this project and for constructive advice.
We would also like to thank Vanja Zuliani for helpful discussions,
and to Matteo Gallet for bringing to our attention the article \cite{LiuXin}.

The  authors were partially supported by
the research group GNSAGA of INDAM, the FRA  of the University of Trieste and the
National Project PRIN 2022 ``Geometry of Algebraic Structures: Moduli, Invariants, Deformations".

\bigskip

\subsection{Conventions}\label{conventions}
We work over the field of complex numbers $\C$.
In particular, by a variety we mean a complex algebraic variety and vector bundles will be complex. 
A vector bundle on a disjoint union 
of varieties is given by a vector bundle on each connected component, possibly having different ranks on each 
connected component; its top Chern class is the cohomology class restricting to the top Chern class on each 
connected component. 
All group actions are left group actions.

For a topological space $Y$, 
$H^*(Y)=\oplus_i H^i(Y, \Q)$ denotes the graded singular cohomology group of $Y$ with rational coefficients. 
If we don't take into account the graded structure, we write simply $H(Y)$. When $Y$ is a variety, $H(Y)$
is the cohomology of the complex analytic space associated to $Y$. 

If $Y$ is a differentiable manifold and $T$ is a Lie group acting on it by diffeomorphisms, 
we denote by $H_T(Y)$ the $T$-equivariant cohomology of $Y$,
either defined via the Borel construction as the cohomology of $Y\times_T \E$, 
or as the de Rham cohomology of the stack $[Y/T]$ (\cite{Beh04}, \cite[Ch. 2]{BGNX}). 
In all the applications we deal with in this article, $Y$ is a complex algebraic variety and $T$ is a complex linear algebraic group
acting algebraically on $Y$. Then $H^i_T(Y)$ coincides with the equivariant cohomology group of $Y$, as defined in 
\cite{AF}, for $i\geq 0$. The reader may freely refer to this latter reference regarding equivariant cohomology.
For the sake of clarity, we should mention that, in the Borel construction $\E$ is a contractible topological space
with a free $T$-action. While, in the algebraic framework of \cite{AF} (compare also \cite{EG}) $\E$ is a nonsingular 
connected algebraic variety (hence finite dimensional) such that $H^i(\E)=0$ for every $0<i<N$ (for $N$ a positive integer or infinity),
and with a free $T$-action such that $\E \to \E/T$ is a principal $T$-bundle locally trivial in the complex topology.
Then $H_T^i(Y) \cong H^i(Y\times_T \E)$, for every $i<N$.

\bigskip

\section{Equivariant Chen-Ruan cohomology}\label{ecrc}
Let $Y$ be a smooth variety, let $G$ be a finite group with the identity $e\in G$ and let $T$ be a complex linear algebraic group 
both acting on $Y$ by regular maps, 
such that these actions commute, i.e. they correspond to a $(G\times T)$-action on $Y$.
Under these hypotheses, we obtain  a strict action by $T$ on the orbifold $[Y/G]$, in the sense of \cite{Romagny}.
In a different context (e.g. when $Y$ is a symplectic manifold and $T, G$ are tori),
the $T$-equivariant Chen-Ruan cohomology of $[Y/G]$, that we denote  $H^*_{T, \rm{CR}}([Y/G])$, 
has been defined and studied in \cite{HoMa}. Their approach follows the original definition of \cite{CR04}
and also uses some results from \cite{AGV08}. A similar definition can be found in \cite{Johnson}. 
Since the examples we are interested in don't fit in these frameworks, we review here the
definition of  $H^*_{T, \rm{CR}}([Y/G])$, adapting the presentation of the usual Chen-Ruan cohomology given in \cite{FG03}.

Let 
$$
I_G(Y):= \{ (g, y) \in G\times Y \, | \, gy=y\} \, 
$$
be the inertia variety. Note that $I_G(Y)= \sqcup_{g\in G}Y^g$, where $Y^g= \{ y\in Y \, | \, g y = y\}$
and that it carries a natural action by $T$.

Define the \textit{equivariant inertia cohomology} of $(Y, G)$ by
$$
H_T(Y, G) := H_T(I_G(Y))=  \oplus_{g\in G} H_T(Y^g) \, .
$$
For $\alpha \in H_T(Y^g)$, we denote by $\alpha_g$ the corresponding element 
in the $g$-th summand of $H_T(Y, G)$.

The grading on $H_T(Y, G)$ is defined using the notion of \textit{age}, which we recall now.
In the discussion that follows we consider $Y$ as a complex manifold.
Let $g\in G$ and let $y\in Y^g$. Let $\lambda_1, \ldots , \lambda_{\dim (Y)}$ be the eigenvalues of the 
differential $T_yg$ of the map $g\colon Y \to Y$ at $y$. Since $g$ has finite order,  $\lambda_j = e^{2\pi i r_j}$, 
for unique  $r_j \in \Q \cap [0, 1[$, $j=1, \ldots , \dim (Y)$. 
The age of $g$ in $y$ is defined as $a(g, y):= \sum_{j=1}^{\dim (Y)} r_j$.
By a well known result of Cartan the action of $g$ in a neighborhood of $y$  can be linearized, hence
$a(g, y)$ is constant on the connected component $Z$ that contains $y$. We denote by $a(g, Z)$ the age of $g$ in any point 
$y\in Z$.
Let $\iota \colon Z \hookrightarrow Y$ be the inclusion. Then, for   $\alpha \in H^k_T(Z)$, we assign 
to $(\iota_* \alpha)_g \in H_T(Y, G)$ the degree $k + 2a(g, Z)$. In this way we obtain a rationally graded 
vector space $H_T^*(Y, G)$.

\begin{rmk}
It follows from the definition of age that
$$
a(g, Z) + a(g^{-1}, Z) = \dim (Y) - \dim (Z) = {\rm codim} (Z\subseteq Y) \, .
$$
\end{rmk}
\begin{rmk}
$a(g, Z) \in \Z$ if and only if $\det (T_y g)=1$, for some $y\in Z$.
In particular, $H_T^*(Y, G)$ is integrally graded if the canonical line bundle of $Y$ is $G$-invariant.
\end{rmk}

\begin{notat}
When $Y^g$ is connected (for example if $Y$ is a vector space), we denote $a(g, Y^g)$ by $a(g)$.
\end{notat}

The group $G$ acts on $H_T^*(Y, G)$ as follows: for every $g, h \in G$, $\alpha \in H_T(Y^g)$, define 
$$
h \alpha_g := (h_* \alpha)_{hgh^{-1}} \, .
$$ 
Note that this action preserves the grading.

\begin{dfn}\label{TCR}
Under the previous notation and hypotheses, the $T$-equivariant Chen-Ruan cohomology of $[Y/G]$ is defined as follows:
$$
H^*_{T, {\rm CR}}([Y/G]) := H^*_T(Y, G)^G \, ,
$$
i.e. as the $G$-invariant subspace of $H^*_T(Y, G)$.
\end{dfn}

\begin{rmk}\label{rmkTCR}
Let $\overline{G} \subseteq G$ be a set of representatives of the conjugacy classes of $G$.
Then $H^\ast_{T, {\rm CR}}([Y/G])$ is isomorphic to
$$
\bigoplus_{g\in \overline{G}} H_T^\ast\left(Y^g/C_G(g)\right) \, .
$$ 
where $C_G(g)$ denotes the centralizer of $g$ in $G$.
\end{rmk}

We  now define a bilinear map 
$$
\mu \colon H^*_T(Y, G) \times H^*_T(Y, G) \to H^*_T(Y, G) \, .
$$

Let $g, h \in G$ and let $\langle g, h\rangle \subseteq G$ be the subgroup that they generate.
Let $Y^{g,h} \subseteq Y$ be the complex sub-manifold of points fixed by $\langle g, h\rangle$, that is $Y^{g,h}=Y^g\cap Y^h$.
Note that the inclusions $Y^{g,h} \hookrightarrow Y^g$, $Y^{g,h} \hookrightarrow Y^h$ are $T$-equivariant.
For  $\alpha \in H_T(Y^g)$ and $\beta \in H_T(Y^h)$, we denote by 
$$
\alpha_{|Y^{g,h}} \, , \beta_{|Y^{g,h}} \in H_T(Y^{g,h})
$$
their pull-backs with respect to the above inclusions.

Let $C\to \pr^1$ be the Galois cover, with Galois group $\langle g, h\rangle$, defined in \cite[Appendix]{FG03}
($C$ is denoted $C(\pr^1, g, h, (gh)^{-1}, \langle g, h\rangle)$ in \textit{op. cit.}).
Let us recall that $C\to \pr^1$ is branched over $0, 1, \infty$, and its local monodromy over $0$
(resp. $1, \infty$) is conjugated to $g$ (resp. $h, (gh)^{-1}$). The action of $\langle g, h\rangle$ on $C$ induces a natural action 
on $H^1(C, \cO_C)$. In the  representation ring of $\langle g, h\rangle$ (with complex coefficients) 
we have the following equality \cite[Lemma 8.5]{JKK07}:
\begin{eqnarray}\label{l85JKK07}
H^1(C, \cO_C) = \mathbb C \ominus \C[\langle g, h\rangle] &\oplus& 
\bigoplus_{k=0}^{|g|-1}\frac{k}{|g|} {\rm Ind}^{\langle g, h\rangle}_{\langle g \rangle} \mathbb{V}_{\langle g \rangle, k} \\
&& \bigoplus_{k'=0}^{|h|-1}\frac{k'}{|h|} {\rm Ind}^{\langle g, h\rangle}_{\langle h \rangle} \mathbb{V}_{\langle h \rangle, k'} \nonumber \\
&&  \bigoplus_{k''=0}^{|gh|-1}\frac{k''}{|gh|} {\rm Ind}^{\langle g, h\rangle}_{\langle (gh)^{-1} \rangle} \mathbb{V}_{\langle (gh)^{-1} \rangle, k''} \nonumber \, ,
\end{eqnarray}
where, for every element $x \in G$, $\langle x \rangle$ denotes the subgroup generated by $x$, 
$|x|$ is its order,  $\mathbb{V}_{\langle x \rangle, k}$ is the irreducible representation of $\langle x \rangle$
such that $x$ acts with character $\exp \left( -\frac{2\pi i k}{|x|}\right)$ (for $k=0, \ldots , |x|-1$), and 
${\rm Ind}^{\langle g, h\rangle}_{\langle x \rangle} \mathbb{V}_{\langle x \rangle, k} $ is the induced 
$\langle g, h\rangle$-representation 
$\C[\langle g, h\rangle]\otimes_{\C[\langle x \rangle]} \mathbb{V}_{\langle x \rangle, k}$.

Let us consider the vector bundle $\left( TY_{|Y^{g,h}}\right) \otimes H^1(C, \cO_C)$ on $Y^{g,h}$ with the natural 
$\langle g, h\rangle$-action and define 
$$
F(g,h) := \left[ \left( TY_{|Y^{g,h}}\right) \otimes H^1(C, \cO_C)\right]^{\langle g, h\rangle} \, ,
$$
its $\langle g, h\rangle$-invariant sub-bundle, which is well defined since $\langle g, h\rangle$ is finite. 

Note that $F(g, h)$ is a $T$-equivariant vector bundle on $Y^{g,h}$ (see e.g.   
\cite[Ch. 2.3]{AF} or \cite[Prop. 3.2]{BGNX} for a definition), in fact
$T$ does not act on $H^1(C, \cO_C)$ and $TY_{|Y^{g,h}}$ is a $T$-equivariant vector bundle by standard reasons.
Hence, $F(g, h)$ induces a vector bundle on $Y^{g, h} \times_T \E$, which we denote $F(g, h)_T$. 
By \cite[Lemma 1.12]{FG03} (or \cite[Ch. 8]{JKK07}),
for every connected component $U$ of $Y^{g,h}$, $F(g,h)_{|U}$ and ${F(g,h)_T}_{|U}$ have  rank 
\begin{equation}\label{rkFgh}
{\rm rk}(F(g, h)_{|U}) =a(g, U) + a(h,U) - a(gh, U) - {\rm codim}(U\subseteq Y^{gh}) \, .
\end{equation}
In the following we denote by 
$$
c^T(g, h) \in H_T(Y^{g,h}) 
$$ 
the top equivariant Chern class of $F(g,h)$, which is the top Chern class of $F(g,h)_T$.

\begin{dfn}
The bilinear map 
$$
\mu \colon H^*_T(Y, G) \times H^*_T(Y, G) \to H^*_T(Y, G) \, 
$$
is defined by
$$
\mu (\alpha_g, \beta_h) := \iota_* \left( \alpha_{|Y^{g,h}} \cdot \beta_{|Y^{g,h}} \cdot c^T(g,h) \right) \, ,
$$
for every $g, h \in G$, $\alpha \in H_T(Y^g)$, $\beta \in H_T(Y^h)$; where
$\iota \colon Y^{g, h} \hookrightarrow Y^{gh}$ is the natural  inclusion, $\iota_* \colon H^i_T(Y^{g,h}) \to H^{i+2d}_T(Y^{gh})$ 
is the $T$-equivariant Gysin map ( \cite[Ch. 3.6]{AF}, \cite{BGNX}) and $d= {\rm codim} (Y^{g,h} \subseteq Y^{gh})$.
\end{dfn}

\begin{rmk}\label{mu1}
Note that, if $g=1$, then $C\cong \pr^1$ by the  Riemann-Hurwitz formula, therefore 
$F(1,h)$ has rank $0$ and 
$$
\mu (\alpha_1, \beta_h) =  \left( \alpha_{|Y^{h}} \cdot \beta  \right)_h \, . 
$$
A similar formula holds true if $h=1$.
\end{rmk}

\begin{thm}
The bilinear map $\mu$ defines a graded, $G$-equivariant, associative multiplication on $H^*_T(Y, G)$. 
\end{thm}
\begin{proof}
The theorem follows from the same arguments used in the proof of \cite[Theorem 1.18]{FG03}, suitably adapted
to the $T$-equivariant case.
In particular, the fact that $\mu$ is a graded multiplication is a consequence of the formula \eqref{rkFgh} for the rank of $F(g,h)$.
The $G$-equivariance of $\mu$ follows from the fact that, for $v\in G$, the action map induced by 
$v\colon Y^{g, h} \to Y^{g', h'}$
(where $g'= vgv^{-1}$, $h'= vhv^{-1}$) is $T$-equivariant and $v^* F(g', h') \cong F(g, h)$ by construction. 

To prove the associativity of $\mu$, observe that the excess intersection formula used in \cite{FG03}
holds true for $T$-equivariant cohomology, under our assumptions. In fact, if $S$ is a smooth variety with a $T$-action, 
$S_1$ and $S_2$ are closed smooth $T$-invariant subvarieties, with smooth intersection $U:=S_1 \cap S_2$,   
the excess bundle $E(S, S_1, S_2)$ of $U$ is $T$-equivariant and the bundle induced on $U\times_T \E$
is isomorphic to the excess bundle $E(S\times_T \E, S_1\times_T \E, S_2 \times_T \E)$ of $U\times_T \E$,
where $\E$ is a smooth variety on which $T$ acts freely, satisfying the properties recalled in Section \ref{conventions}
and of sufficiently high dimension. Therefore, for $j_i \colon S_i \hookrightarrow S$
and $\iota_i \colon U \hookrightarrow S_i$ be the natural inclusions, we have that
$$
j_2^*{j_1}_* (\alpha) = {\iota_2}_* \left( e^T(S, S_1, S_2) \cdot \iota_1^*(\alpha) \right) \, 
$$
in the $T$-equivariant cohomology of $S_2$, for every  $\alpha \in H^*_T(S_1)$, where $e^T(S, S_1, S_2)$ is the top equivariant Chern class of $E(S, S_1, S_2)$. 

As a consequence, we have that  Lemma 1.17 of \cite{FG03} holds true in our framework.
Moreover, the vector bundles $F_L$ and $F_R$ in the proof of \cite[Thm. 1.18]{FG03} are isomorphic 
as $T$-equivariant bundles, because $T$ acts trivially on the cohomology of the curve $C$ in \cite[Construction 1.21]{FG03}
(and hence on the cohomology of $\bar{C}$ in \cite[Rem. 1.24]{FG03}). Hence the associativity of $\mu$ follows.
\end{proof}

\begin{dfn}
The $T$-equivariant Chen-Ruan cup product $\cup_{T, {\rm CR}}$ is defined as the restriction of $\mu$
on $H^*_{T, {\rm CR}}([Y/G]) := H^*_T(Y, G)^G $. 
\end{dfn}

\begin{rmk}
Clearly, when $T$ is the trivial group,  the previous constructions and definitions coincide with those 
of \cite{FG03} and hence $H^*_{T, {\rm CR}}([Y/G])$ coincides with the Chen-Ruan cohomology ring of $[Y/G]$
defined in \cite{CR04}.
\end{rmk}

\section{Equivariant Chen-Ruan cohomology of symplectic vector spaces}\label{Section_ECRSympl}
In this section we give an explicit presentation of the Chen-Ruan cohomology 
algebra $\left( H^*_{T, {\rm CR}}([V/G]), \cup_{T, {\rm CR}} \right)$,
for a symplectic complex vector space $V$, $G \subset \mathrm{Sp}(V)$ a finite subgroup, 
and $T=\C^*$ acting on $V$ by scalar multiplication.
We will denote by $2n$ the dimension of $V$. 

We first recall the following equality for the age of any $g\in G$:
$$
a(g)= \frac{1}{2}{\rm codim}(V^g) \, .
$$
This is well known (see e.g. \cite{GK}) and it can be proved by an explicit computation using Darboux coordinates.

As graded vector space, $H_T^*(V, G) = \oplus_{g\in G} H_T^{*-2a(g)}(V^g)$. 
Moreover, for every $g\in G$, $V^g\times_T \E$ is a vector bundle over $\E/T$
of rank  $\dim (V^g)$. Therefore $H_T^*(V^g)$ is naturally identified with
the equivariant cohomology of a point,  $H^*_T({\rm pt})$, which is isomorphic to the polynomial ring $\Q[t]$,
with the usual polynomial grading, but where $\deg (t)=2$. In other words, as graded vector space, 
\begin{equation}\label{HTVG}
H_T^*(V, G) \cong \Q[t] \otimes_\Q \Q[G] \, ,
\end{equation}
where $\Q[G]$ is the vector space underlying the group algebra of $G$.
In particular,  the standard basis of $\Q[G]$ is  $\{ \fo_g \, | \, g \in G \}$,
where $\fo_g$ is the unity $\fo \in \Q[t] =H^*_T(V^g)$ in the $g$-th summand of  $H_T^*(V, G)$.
The graded structure is induced by setting $\deg (t)=2$ and $\deg (\fo_g) = 2a(g)$.

Let us recall here, for later use, that $t =c_1^T(\C_1)$, that is the equivariant first Chern class of 
the weight-one representation $\C_1$ of $T=\C^*$.

Under the above identifications, the cohomology of the untwisted sector, $H^*_T(V^1)=H^*_T(V) \cong \Q[t]$ acts centrally on 
$H^*_T(V,G)$ (Remark \ref{mu1}). Hence $\mu$ is determined by its $\Q[t]$-bilinearity, associativity, and by the products 
$\mu (\mathbf{1}_g, \mathbf{1}_h)$, for all $g, h \in G$.

\begin{prop}\label{musympl}
For $g, h \in G$, the following equality holds true:
$$
\mu (\fo_g, \fo_h) = t^{a(g)+a(h)-a(gh)} \fo_{gh} \, .
$$
\end{prop}
\begin{proof} 
By definition, $\mu (\fo_g, \fo_h) = \iota_* \left( c^T(g,h) \right)$, where 
$\iota \colon V^{g,h}\times_T \E \hookrightarrow V^{gh} \times_T \E$ is the inclusion and 
$c^T(g,h)$ is the top equivariant  Chern class of the vector bundle $F(g, h)$ over $V^{g,h}$.
Note that, as a $T$-equivariant vector bundle, $F(g, h)$ is the pull-back from a point of
the weight-one representation  $\C_1^{\oplus {\rm rk} (F(g,h))}$. From this we conclude that 
$$
c^T(g,h) = t^{{\rm rk} (F(g,h))} \in H_T(V^{g,h}) \, .
$$
To compute $\iota_* \left( c^T(g,h) \right)$ we observe that the normal bundle $N_{V^{g,h}/V^{gh}}$ is
the pull-back, from a point, of a direct sum of copies of the weight-one representation. 
Therefore, its top equivariant Chern class
is $t^{{\rm codim} (V^{g,h} \subseteq V^{gh})} \in H_T(V^{g,h})$. The claim follows from equation \eqref{rkFgh}
and the self-intersection formula (see e.g. \cite[Ch. 3]{AF}).
\end{proof}

For the $T$-equivariant Chen-Ruan cohomology ring of $[V/G]$, by Definition \ref{TCR} we have that
$$
H^*_{T, {\rm CR}}([V/G]) \cong \left( \Q[t] \otimes_\Q \Q[G] \right)^G \, .
$$ 
Since $G$ acts trivially on $\Q[t]$, we conclude that
\begin{equation}\label{TCRV/G}
H^*_{T, {\rm CR}}([V/G]) \cong \Q[t] \otimes Z\left( \Q[G]\right)  \, ,
\end{equation}
where $Z\left( \Q[G]\right)  \subseteq \Q[G]$ is the vector space underlying the center of the group algebra of $G$.

As in Remark \ref{rmkTCR}, let us choose a set $\overline{G} \subseteq G$ of representatives of the conjugacy classes of $G$.
For any $g\in \overline{G}$, let us denote the conjugacy class of $g$ with $[g]$ and set
$$
\fo_{[g]} := \sum_{g' \in [g]} \fo_{g'} \, .
$$
Then, the  set $\left\{ \fo_{[g]}   \, | \, g\in \overline{G} \right\}$ is a basis for $Z\left( \Q[G]\right)$.

Let $N(g, h, k)$, for $g, h, k \in \overline{G}$, be the structure constants of $Z\left( \Q[G]\right)$, that is they are rational numbers
defined by the relations
$$
\fo_{[g]} \cdot \fo_{[h]} = \sum_{k\in \overline{G}} N(g, h, k)\fo_{[k]} \, , \quad \forall \, g, h \in \overline{G} \, ,
$$
where the product is the one of $Z\left( \Q[G]\right)$.
\begin{rmk}\label{rmk_Burnside}
Note that the $N(g, h, k)$'s are actually natural numbers, in fact  
$$
N(g, h, k) = | \{ (g', h') \in [g] \times [h] \, | \, g'h'=k \} | = 
|[g] \cap k [h^{-1}]| \, .
$$
Moreover, they can be expressed (via the so-called Burnside's formula) 
in terms of the characters of $G$ as follows (see e.g. \cite[p. 45]{Isaacs}):
for any $g, h, k \in \overline{G}$,
\begin{equation}\label{Burnside}
N(g, h, k) = \frac{|[g]|\cdot |[h]|}{|G|} \sum_{\chi \in {\rm Irr}(G)}\frac{\chi (g) \chi (h) \overline{\chi (k)}}{\chi (e)} \, , 
\end{equation}
where ${\rm Irr}(G)$ is the set of the characters of
the irreducible complex representations of $G$ and $e\in G$ is the identity of the group. This formula will be used later in order to explicitly compute  the coefficients $N(g, h, k)$ for finite subgroups $G \subset \mathrm{SL}_2(\C)$.
\end{rmk}

As a direct consequence of Proposition \ref{musympl}, we obtain the following presentation for the $T$-equivariant Chen-Ruan 
product.

\begin{cor}\label{UTCRsympl}
For $g, h \in \overline{G}$,
$$
\fo_{[g]} \cup_{T, {\rm CR}} \fo_{[h]} = 
\sum_{k\in \overline{G}} N(g, h, k) t^{a(g)+a(h)-a(k)} \fo_{[k]} \, .
$$
\end{cor}

\subsection{Relation with Rees algebras}
Note that the  expression for $\cup_{T, {\rm CR}}$ in Corollary \ref{UTCRsympl}
shows that $\left( H^*_{T, {\rm CR}}([V/G]), \cup_{T, {\rm CR}} \right)$
is a graded deformation of the center of the group algebra of $G$. 
Proposition \ref{TCR=Rees} shows that this deformation is naturally described by the Rees algebra 
associated to a particular filtration of $Z\left( \Q[G]\right)$.

\begin{dfn}
Let $A$ be an associative algebra over a commutative ring $R$.  
Let $F:= \{ F_pA \, | \, p\in \N\}$ be a positive increasing filtration of $A$, 
which is compatible with the product of $A$.
That is, for every $p\in \N$, $F_pA \subseteq A$ is an $R$-submodule, $F_pA \subseteq F_{p+1}A$ for every $p$,
and 
$$
F_pA \cdot F_qA \subseteq F_{p+q}A \, , \quad \forall \, p, q \in \N \, .
$$
We define the Rees algebra of $(A, F)$ as the following sub-algebra of $A[t]$:
$$
{\rm Rees}_F(A) := \oplus_{p} (F_pA)t^p \subseteq A[t] \, .
$$
\end{dfn}

In our case, for $A=Z\left( \Q[G]\right)$, let
$$
F_pA:= {\rm Span}\{ \fo_{[g]} \, | \, a(g)\leq p \} \, \quad \forall \, p \in \N \, .
$$
Note that $F_pA \cdot F_qA \subseteq F_{p+q}A$, in fact from Corollary \ref{UTCRsympl} it follows that 
$$
N(g,h,k) \not= 0 \quad \Rightarrow \quad a(k) \leq a(g)+a(h) \, .
$$

\begin{prop}\label{TCR=Rees}
Let $V$ be a finite dimensional symplectic vector space over $\C$, let $G \subset {\rm Sp}(V)$ be a finite subgroup,
let $T=\C^*$ acting on $V$ by scalar multiplication. Then we have the following isomorphism of $\Q[t]$-algebras:
\begin{eqnarray*}
\Phi \colon \left( H_{T, {\rm CR}}^*([V/G]), \cup_{T, {\rm CR}} \right) &\to& {\rm Rees}_F \left( Z\left( \Q[G] \right) \right) \\
\fo_{[g]} &\mapsto & t^{a(g)} \fo_{[g]}
\end{eqnarray*}
with $\Phi(t)=t$, and \(\Phi\) extending \(\mathbb Q[t]\)-linearly,
where ${\rm Rees}_F \left( Z\left( \Q[G] \right) \right)$ is the Rees algebra of $Z\left( \Q[G] \right)$ with respect to the filtration
defined above.
$\Phi$ is not graded with respect to the Chen–Ruan grading (see Remark \ref{interpolation}).
\end{prop}
\begin{proof}
First of all we observe that ${\rm Rees}_F \left( Z\left( \Q[G] \right) \right)$
is a free $\Q[t]$-module with basis $\{ t^{a(g)}\fo_{[g]} \, | \, g\in \overline{G}\}$.
To see this, set $\overline{G}_p:= \{ g\in \overline{G} \, | \, a(g)\leq p \}$.
Then every element of the Rees algebra is of the  form 
$$
\sum_{p=0}^d c_pt^p \, , \quad \mbox{for some} \, c_p\in F_pZ\left( \Q[G] \right) \, .
$$
Writing $c_p=\sum_{g\in \overline{G}_p} \gamma_{g, p} \fo_{[g]}$, for some $\gamma_{g, p}\in \Q$, we have that
$$
\sum_{p=0}^d c_pt^p = \sum_{p=0}^d \sum_{g\in \overline{G}_p} \gamma_{g, p}t^{p-a(g)} \left( t^{a(g)} \fo_{[g]} \right) \, .
$$
Finally, since $\{ \fo_{[g]} \, | \, g\in \overline{G}\}$ is linearly independent in $Z\left( \Q[G] \right)$,
$\{ t^{a(g)}\fo_{[g]} \, | \, g\in \overline{G}\}$ is linearly independent in $\left( Z\left( \Q[G] \right)\right) [t]$.

This observation, together with \eqref{TCRV/G}, implies that $\Phi$ is an isomorphism of $\Q[t]$-modules. 

By Proposition \ref{musympl}, $\Phi$ is a homomorphism of algebras, therefore it is an isomorphism of $\Q[t]$-algebras. 
\end{proof}

\begin{rmk}\label{interpolation}
Note that, under the hypotheses of the previous proposition, 
$$
{\rm Rees}_F \left( Z\left( \Q[G] \right) \right)/(t) \cong {\rm gr}_F \left( Z\left( \Q[G] \right) \right) \, , \qquad
{\rm Rees}_F \left( Z\left( \Q[G] \right) \right)/(t-1) \cong Z\left( \Q[G] \right) \, , 
$$
therefore the Rees algebra interpolates between the center $Z\left(\Q[G]\right)$ and its associated graded 
algebra with respect to the age filtration. By Corollary \ref{UTCRsympl}, the latter is naturally 
isomorphic to the non-equivariant Chen–Ruan cohomology algebra $H^*_{\rm CR}([V/G])$.

Note also that the age filtration is related to the filtration of Ginzburg–Kaledin in \cite{GK} 
by reindexing the degrees by a factor of two,
therefore we recover their presentation of the orbifold Chen–Ruan product.
\end{rmk}

\section{The McKay spectral decomposition and the Bryan–Gholampour transformation}\label{sec:mckay-coordinates}

Let $G\subset \mathrm{SL}_2(\mathbb C)$ be a finite subgroup. From now on we extend scalars from \(\mathbb Q\) to \(\mathbb C\). In this
section we isolate a character-theoretic feature of the Bryan--Gholampour
transformation which is common to all ADE types. The starting point is the
classical observation of McKay that the columns of the character table are
eigenvectors of the affine Cartan matrix \cite{McKay}. The same observation is
implicit in the derivation of the change of variables in
\cite[Sec.~5]{bryan2008root}; our purpose here is to use it systematically in
the comparison with the Chen--Ruan product.  For an abelian group, the resulting
character transform is the Fourier transform on the finite group.  In particular,
for $G=C_n$ the McKay coordinates introduced below become the usual discrete
Fourier coordinates.  We keep the explicit Fourier conventions and inversion
formula in Section~\ref{sec:cyclic}, where they are used in the calculation.

Let $\pi \colon X \to \C^2/G$ be the minimal resolution. The exceptional fiber $\pi^{-1}(0)$ is of the form
\[
\pi^{-1}(0)=E_1 \cup \ldots \cup E_r \, ,
\]
where $E_i$ are smooth rational curves with self-intersection $E_i \cdot E_i = -2$ and pairwise transversal intersection
such that the corresponding dual graph is a Dynkin diagram of ADE type. 
Equivalently, this configuration can be described in terms of the
intersection form on $H_2(X,\mathbb Z)$.
The group $H_2(X, \Z)$ is freely generated by the classes 
$[E_i]$ of $E_i$ and it is equipped with the intersection product $( \, , \, )$. Let 
\[
\langle \, , \, \rangle := - ( \, , \, ) \, .
\]
Then $\left( H_2(X, \Z), \langle \, , \, \rangle \right)$ identifies with the root lattice associated with the previous Dynkin diagram
in such a way that $[E_1], \ldots , [E_r]$ correspond to a system of simple roots and $\langle [E_i] , [E_j] \rangle = C_{ij}$
is the corresponding Cartan matrix.
(cf. \cite{CS98} and the references therein, \cite{Humphreys} for what concerns Lie theory).

Let now $\rho_0, \ldots , \rho_r$ be the irreducible representations of $G$, with $\rho_0$ being the trivial one,
let $\chi_i$ be the character of $\rho_i$ and denote
\[
{\rm Irr}(G):=\{\chi_0,\chi_1,\ldots,\chi_r\},
\qquad \chi_0=1 \, .
\]
Let $d_i:=\chi_i(1)$ be the dimension of the representation $\rho_i$. 
The McKay correspondence establishes a bijection  $\chi_i \longleftrightarrow E_i$, $i=1, \ldots , r$,
such that the McKay adjacency matrix $A=(a_{ij})_{0\leq i,j\leq r}$,
defined by  
\[
V\otimes \rho_i\cong \bigoplus_{j=0}^r a_{ij}\rho_j,
\]
where  $V$ is the representation of $G$ on $\C^2$ given by the inclusion $G\subset {\rm SL}_2(\C)$, satisfies the following identity:
\[
\widetilde C=2I-A \, ,
\]
where $\widetilde C$ is the affine Cartan matrix,  
see, for example,  \cite{GSVMcKay},  \cite{McKay}, \cite{ReidMcKay}, \cite{DolgachevMcKay}. 

Using the natural identification
\[
H^2(X,\mathbb R)\simeq H_2(X,\mathbb R)^*
\]
and the non-degenerate pairing $\langle\, ,\,\rangle$, we identify
$H^2(X,\mathbb R)$ with $H_2(X,\mathbb R)$.
We denote by $\alpha_i\in H^2(X,\mathbb R)$ the element corresponding
to $[E_i]$ under this identification.
Furthermore, we consider the complexified space $H^2(X, \C)$ and we denote by $\langle \, , \, \rangle$
the  complex-bilinear extension of the pairing. Note that
\[
\langle\alpha_i,\alpha_j\rangle=C_{ij},
\]
where $C$ is the Cartan matrix of the finite ADE root system.

Let us define
\[
\beta_i :=\alpha_i,\quad 1\leq i\leq r,
\qquad
\beta_0 :=-\sum_{i=1}^r d_i\alpha_i.
\]
Hence we have
\begin{equation}\label{affine-dimension-relation}
\sum_{i=0}^r d_i \beta_i =0 \, .
\end{equation}
Note that, by the McKay correspondence, the dimensions \(d_i\) are the coefficients of the highest root.

\begin{prop}[McKay spectral decomposition {\cite{McKay}}]\label{prop:mckay-spectral}
\label{prop:mckay-spectral}
With the notation above, the following statements hold true.
\begin{enumerate}
    \item[(i)] The Gram matrix of the family
    $\beta_0,\ldots,\beta_r$ is the affine Cartan matrix:
    \[
        \langle\beta_i,\beta_j\rangle=\widetilde C_{ij},
        \qquad 0\leq i,j\leq r.
    \]

    \item[(ii)] For $g\in G$, set
    \[
        u_g:=\bigl(\chi_0(g),\ldots,\chi_r(g)\bigr)^\mathsf{T}.
    \]
    Then
\begin{equation}\label{mckay-eigenvector}
\widetilde C\,u_g=(2-\chi_V(g))u_g.
\end{equation}
    In particular, the columns of the character table form a system of
    eigenvectors for the affine Cartan matrix.
\end{enumerate}
\end{prop}

\begin{proof}
Taking dimensions in the McKay decomposition
\[
    V\otimes\rho_i\cong\bigoplus_{j=0}^r a_{ij}\rho_j
\]
gives
\[
    2d_i=\sum_{j=0}^r a_{ij}d_j,
    \qquad 0\leq i\leq r.
\]
Hence, if $d:=(d_0,\ldots,d_r)^\mathsf T$, then $\widetilde C\,d=0$. 
For $1\leq i,j\leq r$, we have
\[
    \langle\beta_i,\beta_j\rangle
    =\langle\alpha_i,\alpha_j\rangle
    =C_{ij}
    =\widetilde C_{ij}.
\]
Moreover, since $d_0=1$ and $\widetilde C$ is symmetric,
$\widetilde C d=0$ gives, for $1\leq j\leq r$,
\[
    \widetilde C_{0j}
    =-\sum_{i=1}^r d_i\widetilde C_{ij}
    =-\sum_{i=1}^r d_i C_{ij}.
\]
Using the definition of $\beta_0$, we therefore obtain
\[
    \langle\beta_0,\beta_j\rangle
    =
    -\sum_{i=1}^r d_i\langle\alpha_i,\alpha_j\rangle
    =
    -\sum_{i=1}^r d_i C_{ij}
    =
    \widetilde C_{0j}.
\]
By symmetry the same equality holds with the two indices interchanged.
Finally, $\beta_0$ is the negative of the highest root, and hence
\[
    \langle\beta_0,\beta_0\rangle=2=\widetilde C_{00}.
\]
This proves (i).

For (ii), taking characters in the same McKay decomposition gives
\[
    \chi_V(g)\chi_i(g)
    =
    \sum_{j=0}^r a_{ij}\chi_j(g),
    \qquad 0\leq i\leq r.
\]
Equivalently,
\[
    A u_g=\chi_V(g)u_g,
\]
and therefore
\[
    \widetilde C\,u_g
    =(2I-A)u_g
    =(2-\chi_V(g))u_g,
\]
which proves (ii).
\end{proof}

For every  conjugacy class $[g]$, let 
\[
\Delta_g :=\sqrt{\chi_V(g)-2} \, ,
\]
with the same choice of square root as in the Bryan--Gholampour
transformation, and let
\[
z_g :=|C_G(g)| \, ,
\] 
where $C_G(g)$ is the centralizer of $g$ in $G$. 
Since characters are class functions, the following vector depends only on the
conjugacy class $[g]$ of $g$. Define 
\begin{equation}\label{def:mckay-coordinate}
v_g:=\sum_{i=0}^r \chi_i(g^{-1})\beta_i
=\sum_{i=1}^r\bigl(\chi_i(g^{-1})-d_i\bigr)\alpha_i.
\end{equation}
The second equality follows from \eqref{affine-dimension-relation}. 

Let $e\in G$ denote the identity of the group $G$. Note that $\Delta_e=0$ and that $v_e=0$. 
For $g\not= e$, we have  $\chi_V(g) \not= 2$, hence $\Delta_g \not= 0$. 

We refer to the vectors $v_g$, for $[g]\ne[e]$, as the \textit{McKay coordinate
vectors}.

\begin{prop}\label{prop:mckay-coordinates}
With the notation above, the following statements hold.
\begin{enumerate}
\item[(i)] The Bryan--Gholampour transformation extends to the affine family $\{ \beta_i \, | \, 0\leq i\leq r \}$ as
\begin{equation}\label{BG-affine-family}
L(\beta_i)=\sum_{[h]\neq[e]}\Delta_h\chi_i(h)\fo_{[h]},
\qquad 0\leq i\leq r.
\end{equation}
\item[(ii)] For every non-trivial conjugacy class $[g]$, in McKay coordinates the transformation is diagonal:
\begin{equation}\label{BG-diagonal-general}
L(v_g)=z_g\Delta_g\fo_{[g]}.
\end{equation}
Consequently, $L$ is an isomorphism of $\C[t]$-modules and the vectors $v_g$, indexed by the non-trivial conjugacy classes,
form a basis of the complexified finite root space.
\item[(iii)] For every pair of non-trivial conjugacy classes $[g],[h]$, the Cartan pairing is diagonal up to inversion of conjugacy
classes:
\begin{equation}\label{mckay-pairing-general}
\langle v_g,v_h\rangle=
\begin{cases}
z_g\bigl(2-\chi_V(g)\bigr),& [h]=[g^{-1}],\\
0,& [h]\neq[g^{-1}].
\end{cases}
\end{equation}
\end{enumerate}
\end{prop}

\begin{proof}
For $i\geq1$, \eqref{BG-affine-family} is the Bryan--Gholampour formula
\eqref{BGLintro}. For $i=0$, using \eqref{affine-dimension-relation} and the
 character $\chi_{\rm reg}$ of the regular representation, we have
\[
L(\beta_0) = - \sum_{[h] \not= [e]} \left( \chi_{\rm reg} (h) - \chi_0 (h) \right) \Delta_h \fo_{[h]} = \sum_{[h] \not= [e]}\Delta_h \fo_{[h]} \, ,
\]
which proves the same formula for $\beta_0$.

Applying $L$ to \eqref{def:mckay-coordinate} and using  
\cite[Ch. 2, Prop. 7]{Serre77} (column orthogonality relation for characters),
\[
\sum_{i=0}^r \chi_i(g^{-1})\chi_i(h)
=
\begin{cases}
|C_G(g)|,& [h]=[g],\\
0,& [h]\neq[g],
\end{cases}
\]
gives
\[
L(v_g)
=
\sum_{[h]\neq[e]}\Delta_h
\left(\sum_{i=0}^r\chi_i(g^{-1})\chi_i(h)\right)\fo_{[h]}
=z_g\Delta_g\fo_{[g]},
\]
which proves (ii).
Finally, by (13), the images $L(v_g)=z_g\Delta_g \fo_{[g]}$, indexed by non-trivial conjugacy classes,
are linearly independent,
since $z_g\Delta_g\neq0$. Together with
\(L(1)=1\) and \(\mathbb C[t]\)-linearity, this shows that $L$ is surjective.
Since both sides are free \(\mathbb C[t]\)-modules of the same finite rank, \(L\) is an isomorphism.
 It follows from this that the vectors $v_g$ are linearly
independent. Their number is the rank of the finite root system, so
they form a basis.

For (iii), Proposition~\ref{prop:mckay-spectral}  implies
\[
\langle v_g,v_h\rangle
=
\sum_{i,j=0}^r
\chi_i(g^{-1})\chi_j(h^{-1})\widetilde C_{ij}
=
(2-\chi_V(h^{-1}))
\sum_{i=0}^r
\chi_i(g^{-1})\chi_i(h^{-1}) \, .
\]
By column orthogonality, the last expression vanishes unless
$[h^{-1}]=[g]$. In that case,
\[
\sum_{i=0}^r
\chi_i(g^{-1})\chi_i(h^{-1})=z_g
\]
and $\chi_V(h^{-1})=\chi_V(g)$, hence
\[
\langle v_g,v_h\rangle
=
z_g(2-\chi_V(g)).
\]
This proves (iii).
\end{proof}

\begin{rmk}
The invertibility of \(L\) is already implicit in Bryan–Gholampour's construction: 
their change of variables identifies the non-degenerate Cartan quadratic form with the non-degenerate 
quadratic form on the non-trivial orbifold sectors.
\end{rmk}

The preceding proposition separates the part of the comparison which follows
formally from character theory from the genuinely quantum identity. Write the
specialized Bryan--Gholampour product on the finite root space as
\begin{equation}\label{BG-product-Q}
x*y=-|G|\langle x,y\rangle t^2+t\,\mathcal Q(x,y),
\end{equation}
where
\begin{equation}\label{def:QBG}
\mathcal Q(x,y)=
\sum_{\beta\in R^+}
\langle x,\beta\rangle\langle y,\beta\rangle
\frac{1+q^\beta}{1-q^\beta}\,\beta
\end{equation}
is evaluated at the specialization \eqref{BGq} and $R^+$ denotes the set of positive roots. 
More explicitly, $q^\beta:=\prod_{i=1}^r q_i^{b_i}= \exp \left( {\frac{2\pi \sqrt{-1}}{|G|}\sum_{i=1}^rd_i b_i}\right)$ 
for $\beta=\sum_{i=1}^r b_i\alpha_i$.
Note that, in the previous expression, $q^\beta \not=1$. Indeed, 
since \(\beta\) is a  positive root, \(b_i\ge 0\) for all \(i\) and at least one \(b_i\) is positive. 
Moreover, since the highest root is \(\theta= -\beta_0 = \sum_i d_i\alpha_i\), one has \(b_i\le d_i\). Hence
$$
1\le \sum_{i=1}^r d_i b_i\le \sum_{i=1}^r d_i^2=|G|-1, 
$$
where the last equality follows from \(\sum_{i=0}^r d_i^2=|G|\). Therefore \(q^\beta\ne1\).

\begin{prop}[Uniform reduction of the CCRC]\label{prop:uniform-reduction}
With the notation above, 
\[
L:
QH^*_{\mathbb C^*}(X)_{\pi}
\longrightarrow
\left( H^*_{\mathbb C^*,\mathrm{CR}}
\bigl([\mathbb C^2/G]\bigr) \, , \, \cup_{\C^*, {\rm CR}} \right) \, 
\]
is an isomorphism of $\mathbb C[t]$-algebras if and only if
\begin{equation}\label{master-identity}
\mathcal Q(v_g,v_h)
=
z_gz_h\Delta_g\Delta_h
\sum_{[k]\neq[e]}
\frac{N(g,h,k)}{z_k\Delta_k}\,v_k \, , 
\end{equation}
for every $[g], [h] \not= [e]$, where $N(g, h, k)$ are the structure constants of $Z(\Q [G])$ (see Remark \ref{rmk_Burnside}).
\end{prop}
\begin{proof}
Since $L$ is an isomorphism of $\C[t]$-modules (Proposition \ref{prop:mckay-coordinates} (ii)),
 \(L(1)=1\) and both products are \(\mathbb C[t]\)-bilinear,
 it is an algebra isomorphism 
if and only if 
\[
L(v_g * v_h) = L(v_g) \cup_{\C^*, {\rm CR}} L(v_h) \, ,
\] 
for every $[g], [h] \not= [e]$. 
Using \eqref{BG-product-Q} and \eqref{mckay-pairing-general}, the $t^2$-term on the left-hand side
vanishes unless $[h]=[g^{-1}]$. In that case it is
\[
-|G|z_g(2-\chi_V(g))t^2 \fo_{[e]}.
\]
On the other hand, by \eqref{BG-diagonal-general} and Corollary ~\ref{UTCRsympl},
\[
L(v_g)\cup_{\mathbb C^*,CR}L(v_h)
=
z_gz_h\Delta_g\Delta_h
\sum_{[k]}
N(g,h,k)t^{a(g)+a(h)-a(k)}\fo_{[k]}.
\]
For every non-trivial $g\in G$, one has $a(g)=1$, since $g$
has no nonzero fixed vector in the natural two-dimensional
representation. Hence a term of degree $t^2$ can occur only in the
identity sector $[k]=[e]$. Moreover,
\[
N(g,h,e)\neq0
\quad\Longleftrightarrow\quad
[h]=[g^{-1}].
\]
Thus the $t^2$-term vanishes unless $[h]=[g^{-1}]$. In that case it is
\[ 
z_gz_{g^{-1}} \Delta_g \Delta_{g^{-1}} N(g, g^{-1}, e ) \, .
\] 
Since $z_{g^{-1}} =z_g$,  $\Delta_{g^{-1}} =\Delta_{g}$ because $V$ is self-dual, and $N(g, g^{-1}, e )= |[g]| = \frac{|G|}{z_g}$,
the previous expression equals
\[
z_g^2 \Delta_g^2 \frac{|G|}{z_g} = -|G|z_g(2-\chi_V(g)) \, .
\]

This shows that the identity-sector contribution is automatically compatible for all finite
subgroups $G\subset\mathrm{SL}_2(\mathbb C)$.

It remains to compare the non-trivial sectors. Using
\eqref{BG-diagonal-general}, the Chen--Ruan side is
\[
z_gz_h\Delta_g\Delta_h\,t
\sum_{[k]\neq[e]}N(g,h,k)\fo_{[k]}.
\]
On the quantum side it is $tL(\mathcal Q(v_g,v_h))$. Since the $\{ v_k \, | \, [k] \not= [e] \}$ is a
basis, we can write uniquely
\[
\mathcal Q(v_g,v_h) = \sum_{[k] \not= [e]} {\mathcal Q(v_g,v_h)}_{[k]} v_k \, , \qquad {\mathcal Q(v_g,v_h)}_{[k]} \in \C \, .
\]
Therefore
\[
L(\mathcal Q(v_g,v_h)) = \sum_{[k] \not= [e]} {\mathcal Q(v_g,v_h)}_{[k]} L(v_k) = 
\sum_{[k] \not= [e]} {\mathcal Q(v_g,v_h)}_{[k]} z_k\Delta_k\fo_{[k]} \, .
\]
So, $L$ is an algebra homomorphism if and only if
\[
{Q(v_g,v_h)}_{[k]} = \frac{z_gz_h\Delta_g \Delta_h N(g,h,k)}{z_k\Delta_k} \, , \qquad \forall [k] \not= [e] \, ,
\]
which  is equivalent to \eqref{master-identity}.
\end{proof}

\begin{rmk}
Using Burnside's formula, the identity \eqref{master-identity}
admits an equivalent symmetric trilinear formulation. Indeed, pairing
\eqref{master-identity} with $v_k$ and using Proposition~\ref{prop:mckay-coordinates}(iii), we obtain
\[
 \bigl\langle \mathcal Q(v_g,v_h),v_k\bigr\rangle
 =
 -z_gz_h\Delta_g\Delta_h\Delta_k\,
 N(g,h,k^{-1}).
\]
On the other hand, Burnside's formula \eqref{Burnside} gives
\[
 z_gz_h N(g,h,k^{-1})
 =
 |G|\sum_{\chi\in{\rm Irr}(G)}
 \frac{\chi(g)\chi(h)\chi(k)}{\chi(e)}.
\]
Hence the master identity \eqref{master-identity} is equivalently
\[
 \bigl\langle \mathcal Q(v_g,v_h),v_k\bigr\rangle
 =
 -|G|\Delta_g\Delta_h\Delta_k
 \sum_{\chi\in{\rm Irr}(G)}
 \frac{\chi(g)\chi(h)\chi(k)}{\chi(e)}
 .
\]
In particular, this formulation makes explicit the symmetry in the three
conjugacy classes and expresses the remaining comparison as an identity
between the ADE root system, through the quantum correction $\mathcal Q$,
and the character theory of $G$.

The left-hand side may also be viewed as the root-of-unity specialization
of the trigonometric cubic tensor governing the equivariant quantum
cohomology of the ADE resolution. Related trigonometric tensors arise
naturally in Frobenius-manifold descriptions of ADE resolutions; see, for
instance, \cite{BMS25}.
\end{rmk}

\begin{rmk}\label{rmk:BG-origin}
The spectral content of Propositions~\ref{prop:mckay-spectral} and
\ref{prop:mckay-coordinates} is not meant as a new construction of the
Bryan--Gholampour transformation. In \cite[Sec.~5]{bryan2008root}, the change
of variables is obtained by rewriting the Cartan pairing in character-theoretic
form and matching the quadratic terms of the resolution and orbifold
potentials. The purpose of the McKay coordinates is instead to use that same
character-theoretic structure as a uniform framework for proving product
compatibility. Proposition~\ref{prop:uniform-reduction} isolates the remaining
identity which will be verified below in the ADE cases.
\end{rmk}

\section{The cyclic case}\label{sec:cyclic}

In this section \(n\geq 2\) denotes the order of the cyclic group.
We prove the equivariant cohomological crepant resolution conjecture
for cyclic subgroups of \(\mathrm{SL}_2(\mathbb{C})\). We shall see that the
McKay coordinates of Section~\ref{sec:mckay-coordinates} are exactly the usual
discrete Fourier coordinates in this case.

Set $\z :=e^{2\pi i/n}$, $\theta=\frac{\pi}{n}$. For $r\in\mathbb{Z}/n\mathbb{Z}$, set
\[
g_r=
\begin{pmatrix}
\z^r & 0\\
0 & \z^{-r}
\end{pmatrix},
\]
and let $G=C_n = \{g_r\mid r\in\mathbb{Z}/n\mathbb{Z}\} \subset \mathrm{SL}_2(\mathbb{C})$.
Since \(C_n\) is abelian, every conjugacy class is a singleton.
We write $[r]:=\{g_r\}$, and denote the corresponding Chen--Ruan class by 
$\mathbf{1}_{[r]} \in H^*_{\C^*, {\rm CR}}([\C^2/G])$.

Under the McKay correspondence of Section~\ref{sec:mckay-coordinates}, the finite root system associated with \(G=C_n\) is of type 
$A_{n-1}$.

In this case, the specialization \eqref{BGq} is
\[
q_1=\cdots=q_{n-1}=\z = e^{2\pi i/n} \, .
\]
We now describe the McKay coordinates explicitly in terms of the discrete Fourier transform.
The main result of the section is the following.

\begin{thm}\label{thm:cyclic-CCRC}
The Bryan--Gholampour map  \(L\) in \eqref{BGLintro} defines an isomorphism of
\(\mathbb{C}[t]\)-algebras
\[
L: QH_{\mathbb{C}^*}^*(X)_{\pi} \longrightarrow \left( H_{\mathbb{C}^*,\mathrm{CR}}^* \left( [\mathbb{C}^2/C_n] \right),
\cup_{\mathbb{C}^*,\mathrm{CR}}
\right) \, ,
\]
where $QH_{\mathbb{C}^*}^*(X)_{\pi}$ is defined in \eqref{qccr}.
\end{thm}
By Proposition \ref{prop:uniform-reduction}, it is enough to verify the master identity \eqref{master-identity}. 
We do this using the standard Fourier realization of the \(A_{n-1}\) root system.

\subsection{The root system of type $A_{n-1}$ and Fourier coordinates}\label{RSAF}

We use the standard realization of the root system of type $A_{n-1}$;
see, for instance, \cite{Humphreys}.  As explained in
Section~\ref{sec:mckay-coordinates}, for a cyclic group the character transform
is the discrete Fourier transform.  We record the explicit Fourier conventions
here, since they will be used repeatedly in the root-sum calculation below.

Let
\[
\mathfrak h_{\mathbb C}
=
\left\{
\sum_{j=0}^{n-1} z_j \ve_j \in \mathbb C^n
\;\middle|\;
\sum_{j=0}^{n-1}z_j=0
\right\},
\]
where $\ve_0,\ldots, \ve_{n-1}$ denote the standard basis of
$\mathbb C^n$, with indices understood modulo $n$.  We equip
$\mathbb C^n$ with the standard $\mathbb C$-bilinear form
\[
\langle \ve_j, \ve_k\rangle=\delta_{jk},
\]
and use the same notation for its restriction to
$\mathfrak h_{\mathbb C}$.

With these conventions, the simple roots are
\[
\alpha_j= \ve_j - \ve_{j+1},
\qquad 1\leq j\leq n-1,
\]
and the positive roots are
\[
\ve_s - \ve_r
=
\alpha_s+\cdots+\alpha_{r-1},
\qquad 1\leq s<r\leq n.
\]

For the Fourier-theoretic argument it is convenient to enlarge the
notation cyclically.  Define
\[
\beta_j = \ve_j - \ve_{j+1},
\qquad j\in\mathbb Z/n\mathbb Z.
\]
Thus $\beta_j=\alpha_j$, for $1\leq j\leq n-1$, whereas $\beta_0=-\sum_{j=1}^{n-1}\alpha_j$.
In particular, $\sum_{j=0}^{n-1}\beta_j=0$. 
Note that this is precisely the affine family \(\{\beta_j\}\) introduced in Section~\ref{sec:mckay-coordinates}.

We shall use the standard discrete Fourier transform on the cyclic
group $\mathbb Z/n\mathbb Z$.  Set $\z=e^{2\pi i/n}$, as before. 
For a function $f:\mathbb Z/n\mathbb Z\to\mathbb C$, we use the
Fourier convention
\[
\widehat f(a)
=
\sum_{j=0}^{n-1}\z^{-aj}f(j) \, .
\]
\begin{rmk}\label{FTRT}
For every $a\in \mathbb Z/n\mathbb Z$, let $\rho_a \colon \mathbb Z/n\mathbb Z \to \C^\ast$ be the irreducible representation
$j\mapsto \z^{aj}$. Let $\chi_a$ be the character of $\rho_a$ (which in this case coincides with $\rho_a$).
Let $(\, | \, )$ be the standard Hermitian product on the space of complex-valued functions on $\mathbb Z/n\mathbb Z$
(\cite[Sec. 2.3]{Serre77}), that is
\[
(f|g) = \frac{1}{n}\sum_{j=0}^{n-1}f(j)\overline{g(j)} \, , \qquad f, g \colon \mathbb Z/n\mathbb Z \to \C \, .
\]
Then, 
\[
\widehat f(a) = n(f|\chi_a) \, , \qquad \forall a\in \Z/n\Z \, .
\] 
The orthogonality relations for the characters of $\mathbb Z/n\mathbb Z$ (\cite[Thm. 3, Sec. 2.3]{Serre77}) give
\begin{equation}\label{ORC}
n(\chi_a | \chi_0)=\sum_{j=0}^{n-1}\z^{aj}
=
n\,\delta_{a,0},
\qquad a\in\mathbb Z/n\mathbb Z \, .
\end{equation}
Moreover, by \cite[Thm. 6, Sec. 2.5]{Serre77}, $\chi_0, \ldots, \chi_{n-1}$ form an orthonormal basis 
for the space of complex-valued functions, therefore the following \textit{Fourier inversion} formula holds true:
\begin{equation}\label{FI}
f(j)
= \sum_{a=0}^{n-1}(f|\chi_a)\chi_a(j)=
\frac1n\sum_{a=0}^{n-1}\z^{aj}\widehat f(a) \, ,
\end{equation}
for every $f\colon \Z/n\Z \to \C$ and $j\in \Z/n\Z$.
\end{rmk}
Applying the discrete Fourier transform to the cyclic family
$(\beta_j)_{j\in\mathbb Z/n\mathbb Z}$, we define
\[
v_a
:=
\sum_{j=0}^{n-1}\z^{-aj}\beta_j,
\qquad a\in\mathbb Z/n\mathbb Z.
\]
Since
\[
v_0=\sum_{j=0}^{n-1}\beta_j=0,
\]
Fourier inversion becomes
\begin{equation}\label{FIb}
\beta_j
=
\frac1n\sum_{a=1}^{n-1}\z^{aj}v_a.
\end{equation}
In particular, $v_1,\ldots,v_{n-1}$ form a basis of
$\mathfrak h_{\mathbb C}$. With the McKay labelling of the affine
$A_{n-1}$-diagram, these are precisely the vectors $v_g$ of
\eqref{def:mckay-coordinate}: for $g=g_a$ one has
$\chi_j(g_a^{-1})=\z^{-aj}$, and hence $v_{g_a}=v_a$.

\begin{lmm}\label{lem:fourier-root}
 Let $1\leq a,b\leq n-1$, $s\in\mathbb Z/n\mathbb Z$,
and $1\leq m\leq n-1$.  Set
\[
\alpha_{s,m}:= \ve_s - \ve_{s+m}.
\]
Then:
\begin{enumerate}
\item[(i)]
\[
\alpha_{s,m}
=
\frac1n\sum_{c=1}^{n-1}
\z^{cs}\frac{1-\z^{cm}}{1-\z^c}\,v_c \, ;
\]

\item[(ii)]
\[
\langle v_a,\alpha_{s,m}\rangle
=
\z^{-as}(1-\z^a)(1-\z^{-am}) \, ;
\]

\item[(iii)]
\[
\langle v_a,v_b\rangle
=
\begin{cases}
n(2-\z^a-\z^{-a}),
   & a+b\equiv0\pmod n,\\[2mm]
0,&a+b\not\equiv0\pmod n.
\end{cases}
\]
In particular,
\[
\langle v_a,v_{n-a}\rangle
=
4n\sin^2(a\theta),
\qquad
\theta=\frac{\pi}{n}.
\]
\end{enumerate}
\end{lmm}
\begin{proof}
Since $\alpha_{s,m} = \beta_s+\beta_{s+1}+\cdots+\beta_{s+m-1}$, 
equation \eqref{FIb} (Fourier inversion) and the fact that $v_0=0$ gives
\[
\alpha_{s,m}
=
\frac1n
\sum_{c=1}^{n-1}
\z^{cs}
(1+\z^c+\cdots+\z^{c(m-1)})v_c =
\frac{1}{n}
\sum_{c=1}^{n-1}
\z^{cs}
\frac{1-\z^{cm}}{1-\z^c}\,v_c \, ,
\]
which proves (i).

For (ii), using $\beta_j = \ve_j - \ve_{j+1}$, we obtain
\[
\langle\beta_j, \ve_s - \ve_{s+m}\rangle
=
\delta_{j,s}-\delta_{j,s-1}
-\delta_{j,s+m}+\delta_{j,s+m-1}.
\]
Multiplying by $\z^{-aj}$ and summing over $j$ gives
\[
\langle v_a,\alpha_{s,m}\rangle
=
\z^{-as}(1-\z^a)(1-\z^{-am}).
\]

Part (iii) is also the cyclic specialization of
Proposition~\ref{prop:mckay-coordinates}(iii), since $z_{g_a}=n$ and
$\chi_V(g_a)=\z^a+\z^{-a}$.
\end{proof}

The diagonal form of the Bryan--Gholampour transformation now follows directly
from Proposition~\ref{prop:mckay-coordinates}(ii).  Indeed, for $g_a\neq1$ one has
$z_{g_a}=n$ and
\[
\Delta_{g_a}=\sqrt{\chi_V(g_a)-2}
=2i\sin(a\theta)=:\Delta_a,
\]
with the Bryan--Gholampour choice of square root.  Hence
\begin{equation}\label{BGL'}
L(v_a)=n\Delta_a\fo_{[a]},
\qquad 1\leq a\leq n-1.
\end{equation}
Thus the general McKay coordinates of Section~\ref{sec:mckay-coordinates}
are precisely the discrete Fourier coordinates in type $A$.

%----------------------------------------------------------
% The specialized quantum correction
%----------------------------------------------------------

For the root
\[
\alpha_{s,m}= \ve_s - \ve_{s+m},
\]
the Bryan--Gholampour specialization gives
\[
q^{\alpha_{s,m}}=\z^m,
\qquad
\frac{1+q^{\alpha_{s,m}}}{1-q^{\alpha_{s,m}}}
=
\frac{1+\z^m}{1-\z^m}
=
i\cot(m\theta).
\]
We shall use these identities directly in the
verification of the master identity below.  Notice that the pairs
\((s,m)\), with \(s\in\mathbb Z/n\mathbb Z\) and
\(1\leq m\leq n-1\), parametrize the oriented roots.  Thus, when the
sum over positive roots in the definition of \(\mathcal Q\) is written
in these cyclic coordinates, a factor \(1/2\) is introduced.

%----------------------------------------------------------
% Proof of the main theorem
%----------------------------------------------------------

\subsection{Proof of Theorem~\ref{thm:cyclic-CCRC}}

By Proposition~\ref{prop:uniform-reduction}, it remains only to verify the
master identity \eqref{master-identity}.  Since $C_n$ is abelian, $z_{g_a}=n$
for every $a$, and its class multiplication coefficients satisfy
\[
N(g_a,g_b,g_c)=
\begin{cases}
1,&c\equiv a+b\pmod n,\\
0,&\text{otherwise}.
\end{cases}
\]
Thus, if $c\equiv a+b\pmod n$ and $c\neq0$, the required identity is
\begin{equation}\label{eq:cyclic-master}
\mathcal Q(v_a,v_b)
=
n\frac{\Delta_a\Delta_b}{\Delta_c}\,v_c.
\end{equation}
If $a+b\equiv0\pmod n$, the right-hand side of
\eqref{master-identity} is zero, so it remains to show that
$\mathcal Q(v_a,v_{n-a})=0$.

Assume first that $a+b\not\equiv0\pmod n$, and let
$c\in\{1,\ldots,n-1\}$ be determined by $c\equiv a+b\pmod n$.
Using Lemma~\ref{lem:fourier-root}(ii) and the specialization above, the quantum correction is
\[
\mathcal Q(v_a,v_b)
=
\frac{i}{2}
\sum_{s=0}^{n-1}\sum_{m=1}^{n-1}
\z^{-(a+b)s}
(1-\z^a)(1-\z^b)
(1-\z^{-am})(1-\z^{-bm})
\cot(m\theta)\,\alpha_{s,m}.
\]
Substituting the Fourier expansion of $\alpha_{s,m}$ from
Lemma~\ref{lem:fourier-root}(i) and using character orthogonality
\eqref{ORC}, only the Fourier mode $v_c$ survives.  Hence
\[
\mathcal Q(v_a,v_b)=C_{a,b}v_c,
\qquad
C_{a,b}=\frac{i}{2}
\frac{(1-\z^a)(1-\z^b)}{1-\z^c}T_{a,b,c},
\]
where $T_{a,b,c}$ is the finite Fourier--cotangent sum of
Corollary~\ref{cor:T-abc}.

Since $1-\z^r=-e^{ir\theta}\Delta_r$, we have
\[
\frac{(1-\z^a)(1-\z^b)}{1-\z^c}
=-e^{i(a+b-c)\theta}\frac{\Delta_a\Delta_b}{\Delta_c}.
\]
If $a+b<n$, then $c=a+b$ and Corollary~\ref{cor:T-abc} gives
$T_{a,b,c}=2ni$; if $a+b>n$, then $c=a+b-n$ and it gives
$T_{a,b,c}=-2ni$.  In both cases
\[
C_{a,b}=n\frac{\Delta_a\Delta_b}{\Delta_c},
\]
which proves \eqref{eq:cyclic-master}.

Finally, suppose $a+b\equiv0\pmod n$.  After substituting the Fourier
expansion of $\alpha_{s,m}$ in the quantum correction, the coefficient of
each $v_r$, $1\leq r\leq n-1$, contains
\[
\sum_{s=0}^{n-1}\z^{(r-a-b)s}
=
\sum_{s=0}^{n-1}\z^{rs}=0
\]
by \eqref{ORC}.  Therefore $\mathcal Q(v_a,v_{n-a})=0$, as required. This proves the theorem.
\qed

\section{The binary dihedral case}\label{sec:binary-dihedral}

In this section we prove the equivariant cohomological crepant
resolution conjecture for the binary dihedral group of order \(4n\),
where \(n\ge2\). The general McKay-coordinate formalism already
accounts for the identity-sector contribution and for the invertibility of the
Bryan--Gholampour transformation.  We therefore retain below only the explicit
binary-dihedral calculations needed to compare the non-trivial sectors.

Set
\[
\eta=e^{\pi i/n},
\qquad
\theta'=\frac{\pi}{2n},
\]
and let
\[
G=\cD_n
=
\left\langle
\sigma,\tau
\ \middle|\
\sigma^{2n}=1,\quad
\tau^2=\sigma^n,\quad
\tau\sigma\tau^{-1}=\sigma^{-1}
\right\rangle .
\]
We use the standard realization
\[
\sigma=
\begin{pmatrix}
\eta&0\\
0&\eta^{-1}
\end{pmatrix},
\qquad
\tau=
\begin{pmatrix}
0&-1\\
1&0
\end{pmatrix}
\]
as a subgroup of \(\mathrm{SL}_2(\mathbb C)\). Recall that $|G|=4n$. 

Let
\[
\pi \colon X\longrightarrow\mathbb C^2/G
\]
be the minimal, equivalently crepant, resolution. By the McKay
correspondence, its exceptional configuration is of type
\(D_{n+2}\).

%====================================================================
\subsection{Conjugacy classes and McKay data}\label{conjugacy_classes_Dn}
%====================================================================

The conjugacy classes of \(G\) are
\[
[0]:=\{e\},
\]
\[
[r]:=\{\sigma^r,\sigma^{-r}\},
\qquad
1\le r\le n-1,
\]
\[
[n]:=\{\sigma^n\},
\]
and
\[
[n+1]
:=
\{\tau\sigma^{2j}\mid 0\le j\le n-1\},
\]
\[
[n+2]
:=
\{\tau\sigma^{2j+1}\mid 0\le j\le n-1\}.
\]
Indeed,
\[
\tau\sigma^r\tau^{-1}=\sigma^{-r}  \qquad \mbox{and} \qquad \sigma^k(\tau\sigma^j)\sigma^{-k}= \tau\sigma^{j-2k} \, ,
\]
so conjugation preserves the parity of \(j\) in the non-cyclic
coset and acts transitively on the exponents of each fixed
parity. The class sizes are therefore (following the  order above)
\[
1,\underbrace{2,\ldots,2}_{n-1\ {\rm times}},1,n,n,
\]
whose sum is \(4n\).

For \(1\le r\le n+2\), let $\mathbf1_{[r]}$ denote the corresponding twisted-sector class, and put $\mathbf1_{[0]}$ as the trivial-sector class.

As in Proposition \ref{prop:uniform-reduction}, every non-trivial element of \(G\) has age one.

We recall the irreducible representations of the binary
dihedral group.
In this section, we use notation slightly different from that of Sections \ref{Section_introduction}--\ref{sec:mckay-coordinates}. 
We then explicitly relate it to the notation used in the McKay correspondence.

Set
\begin{equation}\label{kn}
\ka_n
=
\begin{cases}
1,& n\ \text{even},\\
i,& n\ \text{odd} \, ,
\end{cases}
\qquad \mbox{thus} \qquad \ka_n^2 =(-1)^{n}\, .
\end{equation}
There are four one-dimensional representations
\[
Q_0, Q_1, Q_2, Q_3
\]
given by
\[
\begin{array}{c|cc}
 &\sigma&\tau\\
\hline
Q_0&1&1\\
Q_1&1&-1\\
Q_2&-1&\ka_n\\
Q_3&-1&-\ka_n
\end{array}
\]
and \(n-1\) two-dimensional representations
\[
R_k,
\qquad
1\le k\le n-1,
\]
defined by
\[
R_k(\sigma)
=
\begin{pmatrix}
\eta^{k}&0\\
0&\eta^{-k}
\end{pmatrix},
\qquad
R_k(\tau)
=
\begin{pmatrix}
0&(-1)^{k}\\
1&0
\end{pmatrix}.
\]
Their characters satisfy
\[
\chi_{R_k}(\sigma^r)
=
2\cos\left(\frac{\pi kr}{n}\right),
\qquad
\chi_{R_k}(\tau\sigma^r)=0.
\]
For \(1\le k\le n-1\), the two eigenvalues \(\eta^k\) and \(\eta^{-k}\) of \(R_k(\sigma)\) are distinct, 
while \(R_k(\tau)\) exchanges the corresponding eigenspaces. 
Hence \(R_k\) is irreducible. Moreover, the \(R_k\)'s are pairwise non-isomorphic, 
since \(\chi_{R_k}(\sigma)=2\cos(\pi k/n)\), and the four one-dimensional representations \(Q_0,\ldots,Q_3\) are clearly distinct. 
Since $ 4\cdot1^2+(n-1)\cdot2^2=4n=|G|$,
these representations form the complete list of irreducible representations of \(G\).

The natural two-dimensional representation $V$ is \(R_1\).
To determine the precise McKay labelling, set
\[
R_0:=Q_0 \oplus Q_1,
\qquad
R_n:= Q_2\oplus Q_3.
\]
For \(1\le k\le n-1\), the standard product-to-sum identity
gives, on the elements \(\sigma^r\),
\[
\chi_{R_1}(\sigma^r)
\chi_{R_k}(\sigma^r)
=
\chi_{R_{k-1}}(\sigma^r)
+
\chi_{R_{k+1}}(\sigma^r).
\]
On the elements \(\tau\sigma^r\), both sides vanish. Hence
\[
R_1\otimes R_k
\cong
R_{k-1}\oplus R_{k+1},
\qquad
1\le k\le n-1.
\]
The cases \(k=1,n-1\) are interpreted using
\( R_0, R_n\) above. Similarly,
\[
R_1\otimes Q_0\cong R_1,
\qquad
R_1\otimes Q_1\cong R_1,
\]
and
\[
R_1\otimes Q_2\cong R_{n-1},
\qquad
R_1\otimes Q_3\cong R_{n-1}.
\]

Thus, after removing the trivial representation \(Q_0\),
the finite McKay graph is the Dynkin diagram \(D_{n+2}\).
The McKay correspondence yields the following identification:
\[
Q_1\longleftrightarrow\alpha_1,
\]
\[
R_{a-1}\longleftrightarrow\alpha_a,
\qquad
2\le a\le n,
\]
and
\[
Q_2\longleftrightarrow\alpha_{n+1},
\qquad
Q_3\longleftrightarrow\alpha_{n+2}.
\]
Hence
\[
d_1=d_{n+1}=d_{n+2}=1,
\qquad
d_a=2,
\quad
2\le a\le n.
\]
This description also applies when \(n=2\), in which case the
finite diagram is \(D_4\).

We use the standard realization of the root system of type
\(D_{n+2}\); see, for instance, \cite{Humphreys}. Let
\[
\varepsilon_1,\ldots,\varepsilon_{n+2}
\]
be the standard basis of $\C^{n+2}$, equipped with the standard bilinear form $\langle \ve_i , \ve_j \rangle = \delta_{ij}$.
Put
\[
\alpha_a=\varepsilon_a-\varepsilon_{a+1},
\qquad
1\le a\le n,
\]
\[
\alpha_{n+1}
=
\varepsilon_{n+1}-\varepsilon_{n+2},
\qquad
\alpha_{n+2}
=
\varepsilon_{n+1}+\varepsilon_{n+2}.
\]
The positive roots are
\[
R^+
=
\left\{
\varepsilon_a-\varepsilon_b,\,
\varepsilon_a+\varepsilon_b
\ \middle|\
1\le a<b\le n+2
\right\}.
\]

In this case, the specialization \eqref{BGq} is
\[
q_j
=
\exp\left(\frac{2\pi i d_j}{4n}\right),
\qquad
1\le j\le n+2.
\]

\begin{thm}\label{thm:D-CCRC}
Let
\[
\pi \colon X\longrightarrow\mathbb C^2/\cD_n
\]
be the minimal resolution. After the Bryan--Gholampour
specialization
\[
q_j
=
\exp\left(\frac{2\pi i d_j}{4n}\right),
\qquad
1\le j\le n+2,
\]
where
\[
d_1=d_{n+1}=d_{n+2}=1,
\qquad
d_j=2
\quad(2\le j\le n),
\]
the Bryan--Gholampour change of variables induces an
isomorphism of \(\mathbb C[t]\)-algebras
\[
L:
QH_{\mathbb C^*}^*(X)_{\pi}
\longrightarrow
\left(
H_{\mathbb C^*,\mathrm{CR}}^*
([\mathbb C^2/\cD_n]),
\cup_{\mathbb C^*,\mathrm{CR}}
\right).
\]
\end{thm}

%====================================================================
\subsection{Weighted heights and the specialized quantum correction}
%====================================================================

Define the weighted height to be the linear map
\[
h:
\operatorname{Span}_{\mathbb C}
\{\alpha_1,\ldots,\alpha_{n+2}\}
\longrightarrow\mathbb C
\]
such that
\[
h(\alpha_1)
=
h(\alpha_{n+1})
=
h(\alpha_{n+2})
=
1\, , \qquad \mbox{and} \qquad h(\alpha_a)=2 \, , \qquad 2\le a\le n \, .
\]
Thus, if $\beta=\sum_{j=1}^{n+2}b_j\alpha_j$, then $h(\beta) = \sum_{j=1}^{n+2}d_jb_j$ and, 
at the specialization of Theorem~\ref{thm:D-CCRC},
\[
q^\beta
=
\exp\left(\frac{\pi i\,h(\beta)}{2n}\right).
\]

Let
\begin{equation}\label{DefF}
F(m)
:=
\frac{1+\exp(\pi i m/(2n))}
     {1-\exp(\pi i m/(2n))}
=
i\cot\left(\frac{m\theta'}{2}\right),
\qquad
m\not\equiv0\pmod{4n} \, ,
\end{equation}
where, as defined at the beginning of the section, $\theta'=\frac{\pi}{2n}$.
We shall use the following identities:
\[
F(-m)=-F(m),
\qquad
F(4n-m)=-F(m), \qquad
F(2n)=0,
\]
and
\[
F(2n-m)
=
i\tan\left(\frac{m\theta'}{2}\right),
\qquad
1\le m\le2n-1.
\]

\begin{lmm}\label{lem:D-height}
Set
\[
\lambda_1=2n,
\qquad
\lambda_a=2n-2a+3
\quad(2\le a\le n+1),
\qquad
\lambda_{n+2}=0.
\]
Then
\[
h(\varepsilon_a-\varepsilon_b)
=
\lambda_a-\lambda_b \qquad \mbox{and} \qquad h(\varepsilon_a+\varepsilon_b) = \lambda_a+\lambda_b \, , 
\qquad \mbox{for} \quad 1\le a<b\le n+2 \, .
\]
\end{lmm}

\begin{proof}
Let \(\lambda\) be the linear map determined by $\lambda(\varepsilon_a)=\lambda_a$, for $a=1, \ldots , n+2$.
Then
\[
\lambda(\alpha_1)=1 \, ,
\qquad
\lambda(\alpha_a)=2
\quad(2\le a\le n) \, , \qquad \mbox{and} \qquad \lambda(\alpha_{n+1}) = \lambda(\alpha_{n+2}) = 1 \, .
\]
Hence \(\lambda=h\), since the simple roots form a basis.
The asserted formulas follow.
\end{proof}

Note that, in the notation of \eqref{def:QBG}, the specialized quantum correction is
\begin{equation}\label{eq:D-Q}
\mathcal Q(x,y)
=
\sum_{\beta\in R^+}
\langle x,\beta\rangle
\langle y,\beta\rangle
F(h(\beta))\,\beta .
\end{equation}

As shown by the general argument in the proof of Proposition~\ref{prop:uniform-reduction}, 
the identity-sector contributions are automatically compatible. 
Thus only the non-trivial-sector contribution of \eqref{eq:D-Q}  has to be checked below.

%====================================================================
\subsection{The orbifold product in compact form}
%====================================================================

We now use the presentation of the equivariant Chen--Ruan product
established in Section~\ref{Section_ECRSympl} to give a convenient description of the
orbifold product for $\cD_n$. The cyclic subgroup $\langle\sigma\rangle$
provides the main algebraic model, while the two remaining conjugacy
classes are treated separately. Since every non-trivial element of
$\cD_n$ has age $1$, Corollary~\ref{UTCRsympl} shows that the Chen--Ruan product
is determined by the corresponding class-sum products in the center
of the group algebra, together with the appropriate powers of $t$.

Put
\begin{equation}\label{B+-}
B_+
:=
\mathbf1_{[n+1]}+\mathbf1_{[n+2]},
\qquad
B_-
:=
\mathbf1_{[n+1]}-\mathbf1_{[n+2]} \in H_{\mathbb C^*,\mathrm{CR}}^* ([\mathbb C^2/\cD_n]) \, .
\end{equation}

Let
\[
\mathcal R
=
\mathbb C[z,z^{-1}]/(z^{2n}-1) \cong \C[\langle \sigma \rangle] \, , \qquad z \longleftrightarrow \sigma \, ,
\]
and let
\[
\iota:\mathcal R\longrightarrow \mathcal R,
\qquad
\iota(z)=z^{-1}.
\]
Write
\[
\mathcal R^\iota
=
\{f\in \mathcal R\mid \iota(f)=f\}.
\]
For \(1\le r\le n-1\), set
\[
E_r=z^r+z^{-r},
\qquad
E_n=z^n \in \cR \, .
\]
Then $1,E_1,\ldots,E_n$ is a basis of \(\mathcal R^\iota\).

Consider the subspaces 
\[
\operatorname{Span}_{\mathbb C}\{t,\mathbf1_{[1]},\ldots,\mathbf1_{[n]}\} \, , 
\operatorname{Span}_{\mathbb C} \{t^2,t\mathbf1_{[1]},\ldots,t\mathbf1_{[n]}\} \subseteq  
H_{\mathbb C^*,\mathrm{CR}}^* ([\mathbb C^2/\cD_n]) \, .
\] 
Define linear maps
\[
\Psi:
\operatorname{Span}_{\mathbb C}
\{t,\mathbf1_{[1]},\ldots,\mathbf1_{[n]}\}
\longrightarrow
\mathcal R^\iota
\]
by
\[
\Psi(t)=1,
\qquad
\Psi(\mathbf1_{[r]})=E_r, \qquad \mbox{for} \qquad r=1, \ldots , n \, ,
\]
and
\[
\widehat\Psi:
\operatorname{Span}_{\mathbb C}
\{t^2,t\mathbf1_{[1]},\ldots,t\mathbf1_{[n]}\}
\longrightarrow
\mathcal R^\iota
\]
by
\[
\widehat\Psi(t^2)=1,
\qquad
\widehat\Psi(t\mathbf1_{[r]})=E_r \qquad \mbox{for} \qquad r=1, \ldots , n \, .
\]
Since \(1,E_1,\ldots,E_n\) is a basis of \(\mathcal R^\iota\), both
\(\Psi\) and \(\widehat\Psi\) are linear isomorphisms.

\begin{lmm}\label{lem:D-CR-compact}
For
\[
x,y\in
\operatorname{Span}_{\mathbb C}
\{t,\mathbf1_{[1]},\ldots,\mathbf1_{[n]}\},
\]
one has
\begin{equation}\label{eq:D-cyclic-product}
\widehat\Psi
\left(
x\cup_{\mathbb C^*,\mathrm{CR}}y
\right)
=
\Psi(x)\Psi(y).
\end{equation}
Moreover,
\begin{equation}\label{eq:D-B-action}
B_+\cup_{\mathbb C^*,\mathrm{CR}}x
=
\Psi(x)(1)\,tB_+,
\end{equation}
\begin{equation}\label{eq:D-X-action}
B_-\cup_{\mathbb C^*,\mathrm{CR}}x
=
\Psi(x)(-1)\,tB_-,
\end{equation}
where $\Psi(x)(\pm1) \in \C$ is the evaluation of $\Psi(x)$ at $\pm 1$,
\begin{equation}\label{eq:D-BX}
B_+\cup_{\mathbb C^*,\mathrm{CR}}B_-=0,
\end{equation}
and
\begin{equation}\label{eq:D-B-square}
B_+\cup_{\mathbb C^*,\mathrm{CR}}B_+
=
2n
\left(
t^2+
t\sum_{r=1}^n\mathbf1_{[r]}
\right),
\end{equation}
\begin{equation}\label{eq:D-X-square}
B_-\cup_{\mathbb C^*,\mathrm{CR}}B_-
=
2n(-1)^n
\left(
t^2+
t\sum_{r=1}^n(-1)^r\mathbf1_{[r]}
\right).
\end{equation}
\end{lmm}

\begin{proof}
By Corollary~\ref{UTCRsympl}, it suffices to compute the
corresponding class-sum products in $Z(\C[\cD_n])$, keeping track of
the powers of $t$ determined by the ages.

Identify the group algebra of $\langle\sigma\rangle$ with $\mathcal R$
via $\sigma\leftrightarrow z$. The class sums supported on
$\langle\sigma\rangle$ are
\[
 \sigma^r+\sigma^{-r},\qquad 1\leq r<n,
 \qquad\text{and}\qquad
 \sigma^n,
\]
and they correspond respectively to $E_r$ and $E_n$.
Hence their products correspond to multiplication in
$\mathcal R^\iota$. Since every non-trivial element of $\cD_n$ has
age $1$, Corollary~\ref{UTCRsympl} shows that, in the
corresponding Chen--Ruan product, the identity component acquires
a factor $t^2$, whereas every non-identity component acquires a
factor $t$. Since
\[
 \widehat\Psi(t^2)=1,\qquad
 \widehat\Psi(t\fo_{[r]})=E_r,
\]
and $\Psi(t)=1$, it follows by bilinearity that
\[
 \widehat\Psi(x\cup_{\C^*,CR}y)=\Psi(x)\Psi(y),
\]
which proves~\eqref{eq:D-cyclic-product}.

We now consider the two non-cyclic conjugacy classes.
At the level of group-algebra class sums, $B_+$ and $B_-$
correspond respectively to
\[
 \sum_{j=0}^{2n-1}\tau\sigma^j,
 \qquad
 \sum_{j=0}^{2n-1}(-1)^j\tau\sigma^j.
\]
Right multiplication by $\sigma^r$ gives
\[
 \left(\sum_{j=0}^{2n-1}\tau\sigma^j\right)\sigma^r
 =
 \sum_{j=0}^{2n-1}\tau\sigma^j
\]
and
\[
 \left(\sum_{j=0}^{2n-1}(-1)^j\tau\sigma^j\right)\sigma^r
 =
 (-1)^r\sum_{j=0}^{2n-1}(-1)^j\tau\sigma^j.
\]
Therefore, for $1\leq r<n$, multiplication by the cyclic class sum
corresponding to $E_r$ acts on these two class sums by the scalars
$E_r(1)=2$, $E_r(-1)=2(-1)^r$, respectively. For $r=n$, the same statement holds with
$E_n(1)=1$, $E_n(-1)=(-1)^n$, since $E_n=z^n$.  Applying Corollary~\ref{UTCRsympl}, and using
$\Psi(t)=1$, gives
\[
 B_+\cup_{\C^*,CR}x=\Psi(x)(1)tB_+,
 \qquad
 B_-\cup_{\C^*,CR}x=\Psi(x)(-1)tB_-,
\]
proving~\eqref{eq:D-B-action} and
\eqref{eq:D-X-action}.

Finally, the defining relations of $\cD_n$ give $(\tau\sigma^j)(\tau\sigma^k)=\sigma^{n-j+k}$.
Consequently,
\[
 \left(\sum_{j=0}^{2n-1}\tau\sigma^j\right)
 \left(\sum_{k=0}^{2n-1}(-1)^k\tau\sigma^k\right)=0,
\]
whereas
\[
 \left(\sum_{j=0}^{2n-1}\tau\sigma^j\right)^2
 =
 2n\sum_{r=0}^{2n-1}\sigma^r
 =
 2n\left(
  1+\sum_{r=1}^{n-1}(\sigma^r+\sigma^{-r})+\sigma^n
 \right),
\]
and
\[
 \left(\sum_{j=0}^{2n-1}(-1)^j\tau\sigma^j\right)^2
 =
 2n(-1)^n\sum_{r=0}^{2n-1}(-1)^r\sigma^r
\]
\[
 =
 2n(-1)^n\left(
  1+\sum_{r=1}^{n-1}(-1)^r(\sigma^r+\sigma^{-r})
  +(-1)^n\sigma^n
 \right).
\]
Applying Corollary~\ref{UTCRsympl} once more, the identity
component acquires a factor $t^2$ and each non-trivial component
a factor $t$. Therefore
\[
 B_+\cup_{\C^*,CR}B_-=0,
\]
\[
 B_+\cup_{\C^*,CR}B_+
 =
 2n\left(t^2+t\sum_{r=1}^n\fo_{[r]}\right),
\]
and
\[
 B_-\cup_{\C^*,CR}B_-
 =
 2n(-1)^n
 \left(t^2+t\sum_{r=1}^n(-1)^r\fo_{[r]}\right).
\]
This proves~\eqref{eq:D-BX}--\eqref{eq:D-X-square}.
\end{proof}

%====================================================================
\subsection{The Bryan--Gholampour change of variables}
%====================================================================

Rather than working directly in the McKay basis \(v_g\) of Section \ref{sec:mckay-coordinates}, 
we perform the root-system computation in the orthonormal basis \(\varepsilon_1,\ldots,\varepsilon_{n+2}\). 
This is advantageous because the root sum defining \(\mathcal Q\) is sparse in these coordinates. 
The McKay-coordinate transform is nevertheless present explicitly: under \(L\), 
the coordinates \(\varepsilon_2,\ldots,\varepsilon_{n+1}\) become the finite sine modes \(W_m\), 
while \(\varepsilon_1\) and \(\varepsilon_{n+2}\) give the two combinations \(B_\pm\) of the non-cyclic sectors. 
Thus the calculation below may be viewed as performing the quantum correction in root coordinates 
and the Chen–Ruan product in the corresponding sine coordinates.

Let us rename the representatives of conjugacy classes of $\cD_n$ as follows:
\[
g_r=\sigma^r,
\qquad
1\le r\le n,
\qquad
g_{n+1}=\tau,
\qquad
g_{n+2}=\tau\sigma.
\]
Note that, under the notation of subsection \ref{conjugacy_classes_Dn}, $[g_r]=[r]$, for every $r=1, \ldots , n+2$.
Recall that the natural representation $V=R_1$. Then
\[
\chi_V(\sigma^r)=2\cos(2r\theta'),
\qquad
\chi_V(\tau\sigma^j)=0.
\]
We choose
\[
\sqrt{\chi_V(\sigma^r)-2}
=
2i\sin(r\theta'),
\qquad
1\le r\le n,
\]
and
\[
\sqrt{\chi_V(\tau\sigma^j)-2}
=
\sqrt{-2}
=
\sqrt2\,i.
\]

The Bryan--Gholampour change of variables \eqref{BGLintro} is
\[
L(\alpha_j)
=
\sum_{r=1}^{n+2}
\sqrt{\chi_V(g_r)-2}\,
\chi_j(g_r)\,
\mathbf1_{[r]},
\]
where \(\chi_j\) is the irreducible character corresponding to
\(\alpha_j\) under the McKay labelling above.

Set
\[
I_n:=\{1,3,\ldots,2n-1\}
\]
and, for \(m\in I_n\), define
\[
W_m
:=
\sum_{r=1}^n
\sin(mr\theta')\mathbf1_{[r]}.
\]

\begin{lmm}\label{lem:D-L-orthonormal}
On the standard orthonormal basis $\{ \ve_1, \ldots , \ve_{n+2} \}$,
\[
L(\varepsilon_1)
=
-\sqrt2\,iB_+,
\]
\[
L(\varepsilon_a)
=
-2iW_{2a-3},
\qquad
2\le a\le n+1,
\]
and
\[
L(\varepsilon_{n+2})
=
-\sqrt2\, i\kappa_nB_- \, ,
\]
where $B_\pm$ are defined in \eqref{B+-} and $\ka_n$ in \eqref{kn}.
\end{lmm}
\begin{proof}
Let $\tilde{L} \colon \C^{n+2} \to H^*_{\C^*, {\rm CR}}([\C^2/\cD_n])$ be the linear map defined  on the basis
vectors $\ve_1, \ldots , \ve_{n+2}$ by the formulas in the statement. Then,   
it is enough to check that $\tilde{L} (\alpha_a) = L(\alpha_a)$, for all $a= 1, \ldots , n+2$.

For \(\alpha_1=\varepsilon_1-\varepsilon_2\), a direct computation using the character \(\chi_1\) gives
\[
\tilde{L}(\ve_1) - \tilde{L}(\ve_2) = 2iW_1-\sqrt2\,iB_+ = L(\alpha_1) \, .
\]

For \(2\le a\le n\), the standard identity
\[
\sin((2a-1)x)-\sin((2a-3)x)
=
2\sin x\cos(2(a-1)x) \,  
\]
and the formulas for $\chi_{R_{a-1}}$ show that
\begin{align}
\tilde{L}(\ve_a) - \tilde{L}(\ve_{a+1}) = & -2iW_{2a-3} +2iW_{2a-1} \nonumber \\
= & 2i\sum_{r=1}^n \sin((2a-1)r\theta')\mathbf1_{[r]} -2i \sum_{r=1}^n \sin((2a-3)r\theta')\mathbf1_{[r]} \nonumber \\
= & 4i\sum_{r=1}^n
\sin(r\theta')
\cos(2(a-1)r\theta')\mathbf1_{[r]} = L(\alpha_{a}) \, . \nonumber
\end{align}

Finally, using the identity
\[
\sin((2n-1)r\theta')
=
(-1)^{r+1}\sin(r\theta') \, 
\]
together with the values of \( Q_2, Q_3\) on
\(\tau,\tau\sigma\), we obtain
\[
\tilde{L}(\ve_{n+1}) -\tilde{L}(\ve_{n+2}) = L(\alpha_{n+1}) \, , \qquad 
\tilde{L}(\ve_{n+1}) +\tilde{L}(\ve_{n+2}) = L(\alpha_{n+2}) \, .
\]
\end{proof}

%====================================================================
\subsection{Fourier identities}
%====================================================================

The finite trigonometric identities used below are instances of standard discrete Fourier identities 
for cotangent functions; see, for example, \cite[Chapter 7]{BeckRobins} and \cite{Beck}. 
We include the required formulas and their proofs for completeness.

For \(m\in I_n =\{1,3,\ldots,2n-1\} \), define
\[
V_m
:=
\sum_{a=1}^{2n-1}
\sin(ma\theta')z^a
\in \mathcal R.
\]
Since \(m\) is odd, $\sin(m(2n-a)\theta')=\sin(ma\theta')$, 
and therefore
\begin{equation}\label{eq:D-Psi-W}
\Psi(W_m)=V_m,
\qquad
\widehat\Psi(tW_m)=V_m.
\end{equation}

We shall use two standard finite Fourier identities. 

First, the character orthogonality relations for
\(\mathbb Z/(2n)\mathbb Z\) (see \cite{Serre77}), together with the pairing of the
terms indexed by \(r\) and \(2n-r\), give
\begin{equation}\label{eq:D-sine-orthogonality}
\sum_{r=1}^{n-1}
\sin(mr\theta')\sin(\ell r\theta')
+
\frac12
\sin(mn\theta')\sin(\ell n\theta')
=
\frac n2\delta_{m\ell},
\qquad
m,\ell\in I_n.
\end{equation}
Indeed, using again $\sin(m(2n-r)\theta')=\sin(mr\theta')$ and the analogous identity for $\ell$ ($m, \ell \in I_n$),
twice the left hand side of \eqref{eq:D-sine-orthogonality} can be rewritten as
\begin{align*}
\sum_{r=1}^{2n-1} \sin(mr\theta')\sin(\ell r\theta') & = \sum_{r=0}^{2n-1} \sin(mr\theta')\sin(\ell r\theta') \\
& =\frac{1}{2} \sum_{r=0}^{2n-1} \left[ \cos ((m-\ell)r\theta') - \cos ((m+\ell)r\theta') \right] \, .
\end{align*}
For $a:= (m-\ell)/2$ and $b:= (m+\ell)/2$, the previous expression is the real part of
\[
\frac{1}{2}\sum_{r=0}^{2n-1} e^{2\pi i a r /(2n)} - \frac{1}{2}\sum_{r=0}^{2n-1} e^{2\pi i b r /(2n)} \, .
\]
By the character orthogonality relations, the first term equals  $n\delta_{m \ell}$,  
while the second one vanishes, since  $1\leq b=(m+\ell)/2 \leq 2n-1$.

Since \(|I_n|=n\), the fact that \(\Psi\) and \(\widehat\Psi\) are linear isomorphisms, together with 
\eqref{eq:D-Psi-W} and \eqref{eq:D-sine-orthogonality}, shows that
\[
\{V_m\mid m\in I_n\}
\]
is a basis of $\operatorname{Span}_{\mathbb C}\{E_1,\ldots,E_n\}$. 
Consequently,
\[
\{1\}\cup\{V_m\mid m\in I_n\}
\]
is another basis of \(R^\iota\).

Second,  differentiating the finite geometric series
\[ 
\sum_{r=0}^{N-1}x^r=\frac{1-x^N}{1-x} 
\]
and then setting \(x=e^{is\pi/N}\) and taking imaginary parts gives, for \(N\ge2\) and \(s\in\mathbb Z\),
\begin{equation}\label{eq:D-weighted-sine}
\sum_{r=1}^{N-1}
(N-r)\sin\left(\frac{s\pi r}{N}\right)
=
\begin{cases}
\displaystyle
\frac N2
\cot\left(\frac{s\pi}{2N}\right),
&
s\not\equiv0\pmod{2N},\\[3mm]
0,
&
s\equiv0\pmod{2N}.
\end{cases}
\end{equation}

We shall also need the following consequence of the finite
Fourier--cotangent identity proved in the appendix:
\begin{equation}\label{eq:D-cotangent-sum}
\sum_{s=1}^{N-1}
\cot\left(\frac{s\pi}{2N}\right)
\sin\left(\frac{sr\pi}{N}\right)
=
N-r,
\qquad
1\le r\le N.
\end{equation}
Indeed, apply Lemma~\ref{lem:cot-fourier} with \(2N\) in place
of \(n\), take imaginary parts, and pair the terms indexed by
\(s\) and \(2N-s\); the term $s=N$ vanishes.

The following identities contain the Fourier calculations
specific to the binary dihedral case.

\begin{lmm}\label{lem:D-Fourier-identities}
For \(m,\ell\in I_n\), the following identities hold.

\begin{enumerate}
\item[\rm(i)]
\[
V_m(1)
=
\cot\left(\frac{m\theta'}{2}\right),
\qquad
V_m(-1)
=
-\tan\left(\frac{m\theta'}{2}\right).
\]

\item[\rm(ii)]
If \(m\ne\ell\), then
\begin{align}
V_mV_\ell
={}&
\frac12
\left[
\cot\left(\frac{(m+\ell)\theta'}{2}\right)
-
\cot\left(\frac{(m-\ell)\theta'}{2}\right)
\right]V_m
\nonumber\\
&+
\frac12
\left[
\cot\left(\frac{(m-\ell)\theta'}{2}\right)
+
\cot\left(\frac{(m+\ell)\theta'}{2}\right)
\right]V_\ell .
\label{eq:D-VmVl}
\end{align}

\item[\rm(iii)]
\begin{align}
V_m^2
={}&
n
+
\frac32\cot(m\theta')V_m
\nonumber\\
&+
\frac12
\sum_{\substack{\ell\in I_n\\\ell\ne m}}
\left[
\cot\left(\frac{(\ell+m)\theta'}{2}\right)
+
\cot\left(\frac{(\ell-m)\theta'}{2}\right)
\right]V_\ell .
\label{eq:D-Vm-square}
\end{align}

\item[\rm(iv)]
\[
\sum_{m\in I_n}
\cot\left(\frac{m\theta'}{2}\right)V_m
=
n\sum_{r=1}^nE_r,
\]
and
\[
\sum_{m\in I_n}
\tan\left(\frac{m\theta'}{2}\right)V_m
=
-n\sum_{r=1}^n(-1)^rE_r.
\]
\end{enumerate}
\end{lmm}
The proof is given in Appendix \ref{Appendix B}.

%====================================================================
\subsection{Proof of Theorem~\ref{thm:D-CCRC}}
%====================================================================

\begin{proof}
Let
\[
\operatorname{pr}_{\mathrm{nt}}:
H^*_{\mathbb C^*,\mathrm{CR}}([\mathbb C^2/\cD_n])
\longrightarrow
\bigoplus_{r=1}^{n+2}\mathbb C[t]\mathbf1_{[r]}
\]
denote projection away from the identity sector. 
By Proposition \ref{prop:mckay-coordinates}(ii), \(L\) is an isomorphism of \(\mathbb C[t]\)-modules. 
The compatibility of the identity-sector contributions follows from the general argument in the proof of 
Proposition~\ref{prop:uniform-reduction}. 
Since \(\varepsilon_1,\ldots,\varepsilon_{n+2}\) form a basis of the root space, it remains to prove
\begin{equation}\label{eq:D-multiplicativity}
tL\bigl(\mathcal Q(\varepsilon_a,\varepsilon_b)\bigr)
=
\operatorname{pr}_{\mathrm{nt}}\!\left(
L(\varepsilon_a)\cup_{\mathbb C^*,\mathrm{CR}}L(\varepsilon_b)
\right)
\end{equation}
for \(1\le a,b\le n+2\), where $\mathcal{Q}$ is defined in \eqref{def:QBG}.  
By commutativity, assume \(a\le b\).

Throughout the proof we shall use the function $F$ defined in \eqref{DefF} and the notation $\lambda_a$
of Lemma \ref{lem:D-height}.

If \(a<b\), the only positive roots pairing non-trivially with both
\(\varepsilon_a,\varepsilon_b\) are
\(\varepsilon_a-\varepsilon_b\) and \(\varepsilon_a+\varepsilon_b\). Hence
\begin{equation}\label{eq:D-offdiag}
\mathcal Q(\varepsilon_a,\varepsilon_b)
=
-F(\lambda_a-\lambda_b)(\varepsilon_a-\varepsilon_b)
+F(\lambda_a+\lambda_b)(\varepsilon_a+\varepsilon_b).
\end{equation}
For a diagonal pair,
\begin{align}
\mathcal Q(\varepsilon_a,\varepsilon_a)
={}&
\sum_{c<a}
\left[
F(\lambda_c-\lambda_a)(\varepsilon_c-\varepsilon_a)
+F(\lambda_c+\lambda_a)(\varepsilon_c+\varepsilon_a)
\right]
\nonumber\\
&+
\sum_{c>a}
\left[
F(\lambda_a-\lambda_c)(\varepsilon_a-\varepsilon_c)
+F(\lambda_a+\lambda_c)(\varepsilon_a+\varepsilon_c)
\right].
\label{eq:D-diagonal}
\end{align}

\medskip
\noindent
\textit{Case 1: \(2\le a<b\le n+1\).}
Put $m=2a-3$, $\ell=2b-3$. Then (see Lemma \ref{lem:D-height}) $\lambda_a=2n-m$, $\lambda_b=2n-\ell$.
By definition of $F$, \eqref{DefF},
\[
F(\lambda_a-\lambda_b)
=
i\cot\left(\frac{(\ell-m)\theta'}{2}\right)
\]
and
\[
F(\lambda_a+\lambda_b)
=
-i\cot\left(\frac{(m+\ell)\theta'}{2}\right),
\]
equation~\eqref{eq:D-offdiag}, followed by
Lemma~\ref{lem:D-L-orthonormal}, gives
\begin{align*}
\widehat\Psi
\bigl(
tL(\mathcal Q(\varepsilon_a,\varepsilon_b))
\bigr)
={}&
-2
\left[
\cot\left(\frac{(m+\ell)\theta'}{2}\right)
-
\cot\left(\frac{(m-\ell)\theta'}{2}\right)
\right]V_m
\\
&-
2
\left[
\cot\left(\frac{(m-\ell)\theta'}{2}\right)
+
\cot\left(\frac{(m+\ell)\theta'}{2}\right)
\right]V_\ell.
\end{align*}
By Lemma~\ref{lem:D-Fourier-identities}(ii), this is
\[
-4V_mV_\ell.
\]
On the other hand,
\[
L(\varepsilon_a)=-2iW_m,
\qquad
L(\varepsilon_b)=-2iW_\ell,
\]
so Lemma~\ref{lem:D-CR-compact} gives
\[
\widehat\Psi
\left(
L(\varepsilon_a)
\cup_{\mathbb C^*,\mathrm{CR}}
L(\varepsilon_b)
\right)
=
(-2i)^2V_mV_\ell
=
-4V_mV_\ell.
\]
Since \(\widehat\Psi\) is injective,
\eqref{eq:D-multiplicativity} follows.

\medskip
\noindent
\textit{Case 2: \(a=1\) and \(2\le b\le n+1\).}
For $m=2b-3$, $\lambda_1-\lambda_b=m$ and $\lambda_1+\lambda_b=4n-m$.
Since
\[
F(4n-m)=-F(m),
\]
equation~\eqref{eq:D-offdiag} gives
\[
\mathcal Q(\varepsilon_1,\varepsilon_b)
=
-2F(m)\varepsilon_1.
\]
Hence
\[
tL(\mathcal Q(\varepsilon_1,\varepsilon_b))
=
-2\sqrt2
\cot\left(\frac{m\theta'}{2}\right)tB_+.
\]
On the other hand,
\[
L(\varepsilon_1)=-\sqrt2\,iB_+,
\qquad
L(\varepsilon_b)=-2iW_m.
\]
By \eqref{eq:D-B-action} and
Lemma~\ref{lem:D-Fourier-identities}(i),
\[
L(\varepsilon_1)
\cup_{\mathbb C^*,\mathrm{CR}}
L(\varepsilon_b)
=
-2\sqrt2\,
V_m(1)tB_+
=
-2\sqrt2
\cot\left(\frac{m\theta'}{2}\right)tB_+.
\]

\medskip
\noindent
\textit{Case 3: \(2\le a\le n+1\) and \(b=n+2\).}
For $m=2a-3$, $\lambda_a=2n-m$. Since $\lambda_{n+2}=0$, equation~\eqref{eq:D-offdiag} gives
\[
\mathcal Q(\varepsilon_a,\varepsilon_{n+2})
=
2F(2n-m)\varepsilon_{n+2}.
\]
Therefore
\[
tL(\mathcal Q(\varepsilon_a,\varepsilon_{n+2}))
=
2\sqrt2\, \kappa_n
\tan\left(\frac{m\theta'}{2}\right)tB_-.
\]
On the orbifold side, $L(\varepsilon_a)=-2iW_m$ and $L(\varepsilon_{n+2})=-\sqrt2\, i\kappa_nB_-$.
Equation~\eqref{eq:D-X-action} and
Lemma~\ref{lem:D-Fourier-identities}(i) give the same result.

\medskip
\noindent
\textit{Case 4: \(a=1\) and \(b=n+2\).}
Here $\lambda_1-\lambda_{n+2} = \lambda_1+\lambda_{n+2} = 2n$. Since \(F(2n)=0\),
$\mathcal Q(\varepsilon_1,\varepsilon_{n+2})=0$.

The orbifold product also vanishes because
\[
L(\varepsilon_1)
=
-\sqrt2\,iB_+,
\qquad
L(\varepsilon_{n+2})
=
-\sqrt2\, i\kappa_nB_-,
\]
and $B_+\cup_{\mathbb C^*,\mathrm{CR}}B_-=0$. 

\medskip
\noindent
\textit{Case 5: \(2\le a\le n+1\) and \(a=b\).}
Put $m=2a-3$. For every middle index $2\le c\le n+1$, $c\ne a$, write $\ell=2c-3$. 

We first determine the contribution of such an index \(c\).
Suppose \(c<a\), so that \(\ell<m\). Then
\[
\lambda_c-\lambda_a=m-\ell,
\qquad
\lambda_c+\lambda_a=4n-m-\ell.
\]
Hence the corresponding contribution in
\eqref{eq:D-diagonal} is
\[
F(m-\ell)(\varepsilon_c-\varepsilon_a)
-
F(m+\ell)(\varepsilon_c+\varepsilon_a).
\]
After applying \(L\) and \(\widehat\Psi\), its
\(V_\ell\)-component is
\[
-2
\left[
\cot\left(\frac{(\ell+m)\theta'}{2}\right)
+
\cot\left(\frac{(\ell-m)\theta'}{2}\right)
\right]V_\ell.
\]
If \(c>a\), so that \(\ell>m\), the same formula follows from $\lambda_a-\lambda_c=\ell-m$
and $\cot(-x)=-\cot x$. 
Thus, for every \(\ell\ne m\), the coefficient of \(V_\ell\)
in $\widehat\Psi \bigl( tL(\mathcal Q(\varepsilon_a,\varepsilon_a)) \bigr)$
is
\[
-2
\left[
\cot\left(\frac{(\ell+m)\theta'}{2}\right)
+
\cot\left(\frac{(\ell-m)\theta'}{2}\right)
\right].
\]

We next consider the two boundary indices.
The index \(c=1\) contributes
\[
-4\cot\left(\frac{m\theta'}{2}\right)V_m,
\]
whereas \(c=n+2\) contributes
\[
4\tan\left(\frac{m\theta'}{2}\right)V_m.
\]
Their total contribution is therefore
\[
4
\left[
\tan\left(\frac{m\theta'}{2}\right)
-
\cot\left(\frac{m\theta'}{2}\right)
\right]V_m.
\]

The remaining contribution to \(V_m\), coming from the middle
indices, is \(2S_mV_m\), where
\[
S_m
=
-\sum_{\substack{\ell\in I_n\\\ell<m}}
\cot\left(\frac{(m-\ell)\theta'}{2}\right)
+
\sum_{\substack{\ell\in I_n\\\ell>m}}
\cot\left(\frac{(\ell-m)\theta'}{2}\right)
-
\sum_{\substack{\ell\in I_n\\\ell\ne m}}
\cot\left(\frac{(m+\ell)\theta'}{2}\right).
\]
Write $m=2p-1$. 
Then
\[
S_m
=
-\sum_{r=1}^{p-1}\cot(r\theta')
+
\sum_{r=1}^{n-p}\cot(r\theta')
-
\sum_{\substack{s=p\\s\ne2p-1}}^{p+n-1}
\cot(s\theta').
\]
Moreover,
\[
\sum_{s=p}^{p+n-1}\cot(s\theta')
=
-\sum_{r=1}^{p-1}\cot(r\theta')
+
\sum_{r=1}^{n-p}\cot(r\theta'),
\]
where we used
\[
\cot(n\theta')=0,
\qquad
\cot((2n-s)\theta')=-\cot(s\theta').
\]
Since the third sum defining \(S_m\) is obtained from the
left-hand side by omitting the term $s=2p-1=m$, 
it follows that
\[
S_m=\cot(m\theta').
\]

Finally,
\[
\tan x-\cot x=-2\cot(2x)
\]
gives
\begin{align*}
\widehat\Psi
\bigl(
tL(\mathcal Q(\varepsilon_a,\varepsilon_a))
\bigr)
={}&
-6\cot(m\theta')V_m
\\
&-
2
\sum_{\substack{\ell\in I_n\\\ell\ne m}}
\left[
\cot\left(\frac{(\ell+m)\theta'}{2}\right)
+
\cot\left(\frac{(\ell-m)\theta'}{2}\right)
\right]V_\ell.
\end{align*}
By Lemma~\ref{lem:D-Fourier-identities}(iii), this equals $-4\bigl(V_m^2-n\bigr)$. 
On the Chen--Ruan side, $L(\varepsilon_a)=-2iW_m$, and therefore
\[
\widehat\Psi
\left(
\operatorname{pr}_{\mathrm{nt}}\!\left(
L(\varepsilon_a)
\cup_{\mathbb C^*,\mathrm{CR}}
L(\varepsilon_a)
\right)
\right)
=
-4\bigl(V_m^2-n\bigr).
\]
Since \(\widehat\Psi\) is injective,
\eqref{eq:D-multiplicativity} follows.

\medskip
\noindent
\textit{Case 6: \(a=b=1\).}
Equation~\eqref{eq:D-diagonal} simplifies to
\[
\mathcal Q(\varepsilon_1,\varepsilon_1)
=
-2
\sum_{a=2}^{n+1}
F(2a-3)\varepsilon_a.
\]
Indeed, the roots involving \(\varepsilon_1\) and
\(\varepsilon_{n+2}\) have height \(2n\) and hence contribute
zero. Thus
\[
\widehat\Psi
\bigl(
tL(\mathcal Q(\varepsilon_1,\varepsilon_1))
\bigr)
=
-4\sum_{m\in I_n}
\cot\left(\frac{m\theta'}{2}\right)V_m.
\]
By Lemma~\ref{lem:D-Fourier-identities}(iv),
\[
\widehat\Psi
\bigl(
tL(\mathcal Q(\varepsilon_1,\varepsilon_1))
\bigr)
=
-4n\sum_{r=1}^nE_r.
\]
On the other hand, $L(\varepsilon_1)=-\sqrt2\,iB_+$, so \eqref{eq:D-B-square} gives
\[
\operatorname{pr}_{\mathrm{nt}}\!\left(L(\varepsilon_1)
\cup_{\mathbb C^*,\mathrm{CR}}
L(\varepsilon_1)\right)
=
-4nt\sum_{r=1}^n\mathbf1_{[r]},
\]
whose image under \(\widehat\Psi\) is the same expression.
Since \(\widehat\Psi\) is injective, the two products agree.

\medskip
\noindent
\textit{Case 7: \(a=b=n+2\).}
Equation~\eqref{eq:D-diagonal} gives
\[
\mathcal Q(\varepsilon_{n+2},\varepsilon_{n+2})
=
2
\sum_{a=2}^{n+1}
F(2n-(2a-3))\varepsilon_a.
\]
Hence
\[
\widehat\Psi
\bigl(
tL(\mathcal Q(\varepsilon_{n+2},\varepsilon_{n+2}))
\bigr)
=
4\sum_{m\in I_n}
\tan\left(\frac{m\theta'}{2}\right)V_m.
\]
Lemma~\ref{lem:D-Fourier-identities}(iv) gives
\[
\widehat\Psi
\bigl(
tL(\mathcal Q(\varepsilon_{n+2},\varepsilon_{n+2}))
\bigr)
=
-4n\sum_{r=1}^n(-1)^rE_r.
\]
On the Chen--Ruan side, $L(\varepsilon_{n+2}) = -\sqrt2\, i\kappa_nB_-$. 
Since
\[
\kappa_n^2=(-1)^{n},
\]
equation~\eqref{eq:D-X-square} gives
\[
\operatorname{pr}_{\mathrm{nt}}\!\left(L(\varepsilon_{n+2})
\cup_{\mathbb C^*,\mathrm{CR}}
L(\varepsilon_{n+2})\right)
=
-4nt\sum_{r=1}^n(-1)^r\mathbf1_{[r]},
\]
and applying \(\widehat\Psi\) gives the same expression.
Again, injectivity of \(\widehat\Psi\) proves equality.

We have therefore verified \eqref{eq:D-multiplicativity} for every \(1\le a,b\le n+2\). 
This proves Theorem~\ref{thm:D-CCRC}.
\end{proof}

\section{The exceptional cases}
\label{sec:exceptional}

In this section we complete the proof of Theorem~\ref{mainthm} for the three
exceptional finite subgroups of \(\mathrm{SL}_2(\mathbb C)\), corresponding
to the root systems of types \(E_6\), \(E_7\), and \(E_8\).

The three cases follow the same scheme. Proposition~\ref{prop:mckay-coordinates}(ii) 
shows that the Bryan--Gholampour transformation is an isomorphism of the underlying
\(\mathbb C[t]\)-modules, while Proposition~\ref{prop:uniform-reduction} accounts uniformly for
the identity-sector contribution and reduces the comparison to the
non-trivial sectors. It therefore remains only to compare the quantum
correction with the non-trivial Chen--Ruan sectors.

For the exceptional groups no Fourier realization analogous to the
cyclic and binary dihedral cases will be used.  Instead, after fixing the
character table and the corresponding McKay labelling, we express both
sides of the required identity as finite matrices and verify their
equality exactly in the cyclotomic field determined by the
Bryan--Gholampour specialization.

\begin{thm}
\label{thm:exceptional}
Let \(G\subset \mathrm{SL}_2(\mathbb C)\) be one of the binary
tetrahedral, binary octahedral, or binary icosahedral groups, and let
\[
\pi:X\longrightarrow \mathbb C^2/G
\]
be the minimal resolution.  After the specialization~(1), the
Bryan--Gholampour transformation induces an isomorphism of
\(\mathbb C[t]\)-algebras
\[
L:
QH_{\mathbb C^*}^*(X)_\pi
\longrightarrow
\left(
H_{\mathbb C^*,\mathrm{CR}}^*
   \bigl([\mathbb C^2/G]\bigr),
\cup_{\mathbb C^*,\mathrm{CR}}
\right).
\]
\end{thm}

\subsection{The common computational scheme}
\label{subsec:exceptional-common}

Fix one of the three exceptional types \(E_r\), where
\[
r\in\{6,7,8\}.
\]
Choose representatives
\[
g_0=e,g_1,\ldots,g_r
\]
of the conjugacy classes of \(G\), and write $[s]=\operatorname{Cl}(g_s)$ ($0\le s\le r$)
for the conjugacy class represented by \(g_s\).
Thus \([0]\) is the identity class.  We denote by
\(\mathbf 1_{[s]}\) the corresponding Chen--Ruan sector class.

Let
\[
\alpha_1,\ldots,\alpha_r
\]
be the simple roots, ordered according to the McKay correspondence with
the non-trivial irreducible characters
\[
\chi_1,\ldots,\chi_r.
\]
We denote by
\[
C=\bigl(\langle\alpha_i,\alpha_j\rangle\bigr)_{1\le i,j\le r}
\]
the finite Cartan matrix.  This notation distinguishes \(C\) from the
McKay adjacency matrix \(A\) introduced in Section~\ref{sec:mckay-coordinates}.

Let
\[
M=(m_{si})_{1\le s,i\le r}
\]
be the matrix of the Bryan--Gholampour transformation with respect to
the simple-root basis and the non-trivial sector basis, so that
\[
L(\alpha_i)
=
\sum_{s=1}^r m_{si}\mathbf 1_{[s]},
\qquad
1\le i\le r.
\]
Equivalently, by~\eqref{BGLintro},
\[
m_{si}
=
\sqrt{\chi_V(g_s)-2}\,\chi_i(g_s).
\]

For each \(1\le k\le r\), define the \(r\times r\) matrix
\[
N^{(k)}
:=
\bigl(N(g_s,g_t,g_k)\bigr)_{1\le s,t\le r},
\]
where the numbers \(N(g_s,g_t,g_k)\) are the class-algebra structure
constants of Remark~\ref{rmk_Burnside}.  Since every non-trivial element of a finite
subgroup of \(\mathrm{SL}_2(\mathbb C)\) has age one, Corollary~\ref{UTCRsympl}
gives
\[
\operatorname{pr}_{\mathrm{nt}}
\left(
L(\alpha_i)
\cup_{\mathbb C^*,\mathrm{CR}}
L(\alpha_j)
\right)
=
t\sum_{k=1}^r
\left(M^{\mathsf T}N^{(k)}M\right)_{ij}
\mathbf 1_{[k]},
\]
where \(\operatorname{pr}_{\mathrm{nt}}\) denotes projection onto the
sum of the non-trivial sectors.

We next describe the quantum side.  For every root
\(\beta\in R\), write
\[
\beta
=
\sum_{i=1}^r b_i(\beta)\alpha_i
\]
and set
\[
b(\beta)
:=
\begin{pmatrix}
b_1(\beta)\\
\vdots\\
b_r(\beta)
\end{pmatrix}.
\]
Since \(C\) is the Gram matrix of the simple roots, define
\[
p(\beta)
:=
Cb(\beta)
=
\begin{pmatrix}
\langle\alpha_1,\beta\rangle\\
\vdots\\
\langle\alpha_r,\beta\rangle
\end{pmatrix}.
\]
Moreover, by linearity, $L(\beta)=Mb(\beta)$ for the basis $\{\mathbf 1_{[1]}, \cdots \mathbf 1_{[r]}\}$.

Let $d_i=\chi_i(1)$, $1\le i\le r$, 
be the dimensions attached to the finite Dynkin vertices, and define
the weighted height
\[
h(\beta)
:=
\sum_{i=1}^r d_i b_i(\beta).
\] 
If $\xi:=\exp\left(\frac{2\pi i}{|G|}\right)$, then at the Bryan--Gholampour specialization~\eqref{BGq},
$q^\beta=\xi^{h(\beta)}$. 
Accordingly, we put
\[
F(m)
:=
\frac{1+\xi^m}{1-\xi^m}
\]
whenever \(\xi^m\ne1\).  

For \(1\le k\le r\), define the symmetric \(r\times r\) matrix
\[
Q^{(k)}
:=
\sum_{\beta\in R^+}
F(h(\beta))
\bigl(Mb(\beta)\bigr)_k
\,p(\beta)p(\beta)^{\mathsf T}.
\]
Indeed, using~\eqref{def:QBG}, for \(1\le i,j\le r\) we have
\[
Q(\alpha_i,\alpha_j)
=
\sum_{\beta\in R^+}
p_i(\beta)p_j(\beta)
F(h(\beta))\,\beta,
\]
and therefore
\[
L\bigl(Q(\alpha_i,\alpha_j)\bigr)
=
\sum_{k=1}^r
Q_{ij}^{(k)}\mathbf 1_{[k]}.
\]

Consequently, the comparison of the non-trivial sectors is reduced to
the following matrix identities.

\begin{lmm}
\label{lmm:exceptional-matrix-criterion}
With the notation above, the Bryan--Gholampour transformation is
compatible with the quantum correction in the exceptional type \(E_r\)
if and only if
\[
Q^{(k)}
=
M^{\mathsf T}N^{(k)}M,
\qquad
1\le k\le r.
\]
\end{lmm}

\begin{proof}
For \(1\le i,j\le r\), the preceding formulas give
\[
\operatorname{pr}_{\mathrm{nt}}
\left(
L(\alpha_i)
\cup_{\mathbb C^*,\mathrm{CR}}
L(\alpha_j)
\right)
=
t\sum_{k=1}^r
\left(M^{\mathsf T}N^{(k)}M\right)_{ij}
\mathbf 1_{[k]},
\]
whereas
\[
tL\bigl(Q(\alpha_i,\alpha_j)\bigr)
=
t\sum_{k=1}^r
Q_{ij}^{(k)}\mathbf 1_{[k]}.
\]
Since the simple roots form a basis of the finite root space, equality
of the two bilinear maps is equivalent to equality of these coefficients
for every \(i,j,k\).  This is precisely the stated system of matrix
identities.  Proposition~\ref{prop:uniform-reduction} has already established the compatibility
of the identity-sector contribution, so no further condition is
required.
\end{proof}

It remains to explain how these matrix identities will be verified.
For each exceptional group, all entries occurring in the matrices
\(M\), \(N^{(k)}\), and \(Q^{(k)}\) belong to the cyclotomic field
\[
\mathbb K
:=
\mathbb Q(\xi)
\cong
\mathbb Q[X]/\bigl(\Phi_{|G|}(X)\bigr),
\]
where
\[
|G|=24,\ 48,\ 120
\]
in types \(E_6,E_7,E_8\), respectively.  The case-by-case formulas below
express all radicals occurring in \(M\) as elements of \(\mathbb K\),
while
\[
F(h(\beta))
=
\frac{1+\xi^{h(\beta)}}{1-\xi^{h(\beta)}}.
\]

For
\[
1\le i\le j\le r,
\qquad
1\le k\le r,
\]
set
\[
D_{ij}^{(k)}
:=
Q_{ij}^{(k)}
-
\left(M^{\mathsf T}N^{(k)}M\right)_{ij}
\in\mathbb K.
\]
Since both sides are symmetric in \(i,j\), there are
\[
r\binom{r+1}{2}
\]
scalar identities to verify.

After expressing every radical and every \(F(h(\beta))\) in terms of
\(\xi\), we clear the finitely many nonzero denominators
\(1-\xi^{h(\beta)}\).  Each \(D_{ij}^{(k)}\) is then represented by a
rational function in \(\xi\) whose denominator is nonzero and whose
numerator is a polynomial with rational coefficients.  Reducing this
numerator modulo the cyclotomic polynomial
\(\Phi_{|G|}(X)\) gives an exact test for the vanishing of
\(D_{ij}^{(k)}\).

Thus the three exceptional cases reduce respectively to
\[
6\binom{7}{2}=126,
\qquad
7\binom{8}{2}=196,
\qquad
8\binom{9}{2}=288
\]
exact scalar identities in
\[
\mathbb Q(e^{2\pi i/24}),
\qquad
\mathbb Q(e^{2\pi i/48}),
\qquad
\mathbb Q(e^{2\pi i/120}).
\]
In the following three subsections we give the complete group-theoretic
and root-theoretic data needed for these calculations and record, in
each case, a representative non-trivial identity illustrating the exact
cyclotomic reduction. Sage scripts implementing the verifications of all these identities are available from the authors.

\subsection{Type \(E_6\)}
\label{subsec:E6}

We now apply the common scheme of Subsection~\ref{subsec:exceptional-common}
to the binary tetrahedral group.

Let
\[
G=\mathcal T
\]
be the binary tetrahedral group of order \(24\).  We realize
\(\mathcal T\) as the subgroup of the unit quaternions
\[
\mathcal T
=
\{\pm e,\pm i,\pm j,\pm k\}
\cup
\left\{
\frac{\pm e\pm i\pm j\pm k}{2}
\right\},
\]
where all \(16\) choices of signs occur in the second set.  We use the
embedding
\[
\mathcal T\subset \mathrm{SL}_2(\mathbb C)
\]
determined by
\[
e\longmapsto
\begin{pmatrix}
1&0\\
0&1
\end{pmatrix},
\qquad
i\longmapsto
\begin{pmatrix}
i&0\\
0&-i
\end{pmatrix},
\]
\[
j\longmapsto
\begin{pmatrix}
0&1\\
-1&0
\end{pmatrix},
\qquad
k\longmapsto
\begin{pmatrix}
0&i\\
i&0
\end{pmatrix}.
\]
Thus
\[
i^2=j^2=k^2=-e,
\qquad
ij=k,\quad jk=i,\quad ki=j.
\]

Put
\[
\gamma:=\frac{e+i+j+k}{2}.
\]
We order the conjugacy classes as
\[
C_0=\{e\},
\qquad
C_1=\{-e\},
\qquad
C_2=\{\pm i,\pm j,\pm k\},
\]
\[
C_3=\operatorname{Cl}(\gamma),
\qquad
C_4=\operatorname{Cl}(\gamma^{-1}),
\qquad
C_5=-C_3,
\qquad
C_6=-C_4.
\]
Their respective sizes are
\[
1,\ 1,\ 6,\ 4,\ 4,\ 4,\ 4.
\]
We choose representatives
\[
g_0=e,\quad
g_1=-e,\quad
g_2=i,\quad
g_3=\gamma,\quad
g_4=\gamma^{-1},
\quad
g_5=-\gamma,\quad
g_6=-\gamma^{-1}.
\]
As in Subsection~\ref{subsec:exceptional-common}, we write
\[
[r]:=C_r,
\qquad
0\le r\le6.
\]

\medskip

\noindent\textbf{Character and McKay data.}
Let $\omega=e^{2\pi i/3}$. 
With the ordering above, the irreducible character table of
\(\mathcal T\) is
\[
\begin{array}{c|c|ccccccc}
 & |C_r|
 & \chi_0 & \chi_1 & \chi_2 & \chi_3
 & \chi_4 & \chi_5 & \chi_6 \\ \hline
C_0 & 1 & 1 & 1 & 2 & 2 & 3 & 2 & 1\\
C_1 & 1 & 1 & 1 & -2 & -2 & 3 & -2 & 1\\
C_2 & 6 & 1 & 1 & 0 & 0 & -1 & 0 & 1\\
C_3 & 4 & 1 & \omega & 1 & \omega
    & 0 & \omega^2 & \omega^2\\
C_4 & 4 & 1 & \omega^2 & 1 & \omega^2
    & 0 & \omega & \omega\\
C_5 & 4 & 1 & \omega & -1 & -\omega
    & 0 & -\omega^2 & \omega^2\\
C_6 & 4 & 1 & \omega^2 & -1 & -\omega^2
    & 0 & -\omega & \omega
\end{array}
\]
Here \(\chi_r=\chi_{\rho_r}\) is the character of the irreducible representation $\rho_r$ for $0\le r\le 6$. In particular, \(\rho_0\) is the trivial representation and $V=\rho_2$ is the natural two-dimensional representation.  Moreover,
\[
\rho_6=\rho_1^2,
\qquad
\rho_3=\rho_1\otimes\rho_2,
\qquad
\rho_5=\rho_6\otimes\rho_2.
\]
The displayed characters form the complete set of irreducible
characters of \(\mathcal T\); see also the standard character
orthogonality relations in \cite[Sec.~2.3]{Serre77}.

Tensoring by the natural representation \(V=\rho_2\) gives
\[
\rho_2\otimes\rho_0\simeq\rho_2,
\qquad
\rho_2\otimes\rho_1\simeq\rho_3,
\qquad
\rho_2\otimes\rho_6\simeq\rho_5,
\]
\[
\rho_2\otimes\rho_2\simeq\rho_0\oplus\rho_4,
\qquad
\rho_2\otimes\rho_3\simeq\rho_1\oplus\rho_4,
\]
\[
\rho_2\otimes\rho_5\simeq\rho_6\oplus\rho_4,
\qquad
\rho_2\otimes\rho_4
\simeq
\rho_2\oplus\rho_3\oplus\rho_5.
\]
Hence the affine McKay diagram is of type \(\widetilde E_6\);
see \cite{McKay}. We use the McKay labelling
\[
\rho_i\longleftrightarrow\alpha_i,
\qquad
1\le i\le6.
\]
Thus
\[
(d_1,\ldots,d_6)=(1,2,2,3,2,1).
\]

Let $\{\epsilon_1,\ldots,\epsilon_6 \}$
be the standard orthonormal basis of \(\mathbb C^6\), equipped with
the standard complex bilinear form.  We use the simple roots
\[
\alpha_1=\epsilon_1-\epsilon_2,
\qquad
\alpha_2=\epsilon_4-\epsilon_5,
\]
\[
\alpha_3=\epsilon_2-\epsilon_3,
\qquad
\alpha_4=\epsilon_3-\epsilon_4,
\qquad
\alpha_5=\epsilon_4+\epsilon_5,
\]
and
\[
\alpha_6
=
\frac12
\left(
-\epsilon_1-\epsilon_2-\epsilon_3-\epsilon_4-\epsilon_5
+\sqrt3\,\epsilon_6
\right).
\]
Their Gram matrix is the finite Cartan matrix
\[
C=
\bigl(\langle\alpha_i,\alpha_j\rangle\bigr)_{1\le i,j\le6}
=
\begin{pmatrix}
2&0&-1&0&0&0\\
0&2&0&-1&0&0\\
-1&0&2&-1&0&0\\
0&-1&-1&2&-1&0\\
0&0&0&-1&2&-1\\
0&0&0&0&-1&2
\end{pmatrix}.
\]

\medskip

\noindent\textbf{Roots and weighted height.}
We use the standard coordinate realization of the \(E_6\)-root system;
cf.~\cite{Humphreys}.  Write
\[
R=R_0\sqcup R_1,
\]
where
\[
R_0
=
\left\{
\pm\epsilon_a\pm\epsilon_b
\;\middle|\;
1\le a<b\le5
\right\},
\]
and
\[
R_1
=
\left\{
\frac12
\left(
\delta_1\epsilon_1+\cdots+\delta_5\epsilon_5
+\delta_6\sqrt3\,\epsilon_6
\right)
\;\middle|\;
\begin{array}{c}
\delta_j\in\{\pm1\},\\
\delta_1\cdots\delta_6=-1
\end{array}
\right\}.
\]
Hence
\[
|R_0|=40,
\qquad
|R_1|=32,
\qquad
|R|=72,
\qquad
|R^+|=36.
\]

The weighted height introduced in
Subsection~\ref{subsec:exceptional-common} is
\[
h(\beta)
=
b_1+2b_2+2b_3+3b_4+2b_5+b_6
\]
for $\beta=\sum_{j=1}^6 b_j\alpha_j$.

Let $\xi=e^{2\pi i/24}$. 
In
this case
\[
F(m)
=
\frac{1+\xi^m}{1-\xi^m}
=
i\cot\left(\frac{m\pi}{24}\right),
\qquad
m\not\equiv0\pmod{24}.
\]

For later use we also record explicitly the simple-root coordinates.
If
\[
\beta=x_1\epsilon_1+\cdots+x_6\epsilon_6,
\]
then solving
\[
\beta=\sum_{\ell=1}^6 b_\ell(\beta)\alpha_\ell
\]
gives
\[
b_1
=
x_1+\frac{x_6}{\sqrt3},
\qquad
b_6
=
\frac{2x_6}{\sqrt3},
\]
\[
b_3
=
x_1+x_2+\frac{2x_6}{\sqrt3},
\qquad
b_4
=
x_1+x_2+x_3+\sqrt3\,x_6,
\]
\[
b_2
=
\frac12
\left(
x_1+x_2+x_3+x_4-x_5+\sqrt3\,x_6
\right),
\]
and
\[
b_5
=
\frac12
\left(
x_1+x_2+x_3+x_4+x_5+\frac{5x_6}{\sqrt3}
\right).
\]
For every \(\beta\in R\), these numbers are integers, and the positive
roots are precisely those for which
\[
b_\ell(\beta)\ge0,
\qquad
1\le\ell\le6.
\]

\medskip

\noindent\textbf{The Bryan--Gholampour matrix.}
From the character table,
\[
\chi_V(g_1)=-2,
\qquad
\chi_V(g_2)=0,
\]
\[
\chi_V(g_3)=\chi_V(g_4)=1,
\qquad
\chi_V(g_5)=\chi_V(g_6)=-1.
\]
We choose
\[
\sqrt{-4}=2i,
\qquad
\sqrt{-2}=\sqrt2\,i,
\qquad
\sqrt{-1}=i,
\qquad
\sqrt{-3}=\sqrt3\,i.
\]
Thus, in the ordered bases
\[
\alpha_1,\ldots,\alpha_6
\qquad\text{and}\qquad
\mathbf1_{[1]},\ldots,\mathbf1_{[6]},
\]
the matrix \(M\) of the Bryan--Gholampour transformation is
\[
M
=
i
\begin{pmatrix}
2&-4&-4&6&-4&2\\
\sqrt2&0&0&-\sqrt2&0&\sqrt2\\
\omega&1&\omega&0&\omega^2&\omega^2\\
\omega^2&1&\omega^2&0&\omega&\omega\\
\sqrt3\,\omega&-\sqrt3&-\sqrt3\,\omega&0&
-\sqrt3\,\omega^2&\sqrt3\,\omega^2\\
\sqrt3\,\omega^2&-\sqrt3&-\sqrt3\,\omega^2&0&
-\sqrt3\,\omega&\sqrt3\,\omega
\end{pmatrix}.
\]
The matrices \(N^{(k)}\), \(1\le k\le6\), are obtained from the
preceding character table by the Burnside formula of Remark~\ref{rmk_Burnside}.

\medskip

\noindent\textbf{Exact cyclotomic verification.}
The relevant cyclotomic field is
\[
\mathbb K
=
\mathbb Q(\xi)
\cong
\mathbb Q[X]/\bigl(\Phi_{24}(X)\bigr),
\]
where
\[
\Phi_{24}(X)=X^8-X^4+1.
\]
All quantities entering the matrices \(Q^{(k)}\) belong to
\(\mathbb K\).  Explicitly,
\[
i=\xi^6,
\qquad
\omega=\xi^8,
\]
\[
\sqrt2=\xi^3+\xi^{-3},
\qquad
\sqrt3=\xi^2+\xi^{-2},
\]
and
\[
F(h)
=
\frac{1+\xi^h}{1-\xi^h}.
\]

The coordinate realization above gives all \(36\) positive roots
explicitly.  For each of them we compute
\[
b(\beta),
\qquad
h(\beta),
\qquad
p(\beta)=Cb(\beta),
\qquad
Mb(\beta),
\]
and hence every entry of the six matrices \(Q^{(k)}\).

According to the notation of
Subsection~\ref{subsec:exceptional-common}, there are
\[
6\binom72=126
\]
differences
\[
D_{ij}^{(k)}
=
Q_{ij}^{(k)}
-
\left(M^{\mathsf T}N^{(k)}M\right)_{ij},
\qquad
1\le i\le j\le6,
\quad
1\le k\le6.
\]
Replacing \(i,\omega,\sqrt2,\sqrt3\) and every
\(F(h(\beta))\) by the expressions above in \(\xi\), clearing the
nonzero denominators \(1-\xi^{h(\beta)}\), and reducing the resulting
numerators modulo
\[
\Phi_{24}(X)=X^8-X^4+1
\]
gives zero for every triple \((i,j,k)\).  Therefore
\[
Q^{(k)}
=
M^{\mathsf T}N^{(k)}M,
\qquad
1\le k\le6.
\]

We record one representative computation as a check on the
normalizations.  For \((i,j)=(4,4)\), the exact root sum gives
\[
\alpha_4*\alpha_4
=
-48t^2
+
it\left[
(5-2\sqrt2)(\alpha_2+\alpha_3+\alpha_5)
+
(12-4\sqrt2)\alpha_4
\right].
\]
Applying \(L\) gives
\[
L(\alpha_4*\alpha_4)
=
-48t^2
-12t\,\mathbf1_{[1]}
+
(12\sqrt2-8)t\,\mathbf1_{[2]}.
\]
On the other hand,
\[
L(\alpha_4)
=
i\left(
6\mathbf1_{[1]}
-
\sqrt2\,\mathbf1_{[2]}
\right).
\]
The class-algebra coefficients obtained from the character table give
\[
\mathbf1_{[1]}
\cup_{\mathbb C^*,\mathrm{CR}}
\mathbf1_{[1]}
=
t^2,
\]
\[
\mathbf1_{[1]}
\cup_{\mathbb C^*,\mathrm{CR}}
\mathbf1_{[2]}
=
t\,\mathbf1_{[2]},
\]
and
\[
\mathbf1_{[2]}
\cup_{\mathbb C^*,\mathrm{CR}}
\mathbf1_{[2]}
=
6t^2+
\left(
6\mathbf1_{[1]}+4\mathbf1_{[2]}
\right)t.
\]
Consequently,
\[
\begin{aligned}
L(\alpha_4)
\cup_{\mathbb C^*,\mathrm{CR}}
L(\alpha_4)
&=
-
\left(
6\mathbf1_{[1]}
-\sqrt2\,\mathbf1_{[2]}
\right)^2\\
&=
-48t^2
-12t\,\mathbf1_{[1]}
+
(12\sqrt2-8)t\,\mathbf1_{[2]},
\end{aligned}
\]
in exact agreement with \(L(\alpha_4*\alpha_4)\).

Hence the identities of
Lemma~\ref{lmm:exceptional-matrix-criterion} hold in type \(E_6\).
By Proposition~\ref{prop:uniform-reduction}, this proves the \(E_6\) case of
Theorem~\ref{thm:exceptional}.

\subsection{Type \(E_7\)}
\label{subsec:E7}

We next apply the common scheme of
Subsection~\ref{subsec:exceptional-common} to the binary octahedral
group.

Let
\[
G=\mathcal O
\]
be the binary octahedral group of order \(48\).  We use the quaternionic
embedding into \(\mathrm{SL}_2(\mathbb C)\) fixed in
Subsection~\ref{subsec:E6}.  Equivalently, \(\mathcal O\) is the inverse
image of the rotational symmetry group of the cube under the double
covering $\mathrm{SU}(2)\longrightarrow\mathrm{SO}(3)$.

We order the eight conjugacy classes by choosing representatives
\[
g_0=e,
\qquad
g_1=-e,
\qquad
g_2=i,
\qquad
g_3=\frac{i+j}{\sqrt2},
\]
\[
g_4=\frac{e+i}{\sqrt2},
\qquad
g_5=\frac{-e+i}{\sqrt2},
\qquad
g_6=\frac{e+i+j+k}{2},
\qquad
g_7=\frac{-e+i+j+k}{2},
\]
and setting
\[
C_r=\operatorname{Cl}(g_r),
\qquad
0\le r\le7.
\]
Their respective sizes are
\[
(|C_0|,\ldots,|C_7|)
=
(1,1,6,12,6,6,8,8).
\]
As before, we write
\[
[r]:=C_r,
\qquad
0\le r\le7.
\]

\medskip

\noindent\textbf{Character and McKay data.}
With the ordering above, the irreducible character table of
\(\mathcal O\) is
\[
\begin{array}{c|c|cccccccc}
 & |C_r|
 & \chi_0 & \chi_1 & \chi_2 & \chi_3
 & \chi_4 & \chi_5 & \chi_6 & \chi_7 \\ \hline
C_0 & 1
& 1&2&2&3&4&3&2&1\\
C_1 & 1
& 1&-2&2&3&-4&3&-2&1\\
C_2 & 6
& 1&0&2&-1&0&-1&0&1\\
C_3 & 12
& 1&0&0&-1&0&1&0&-1\\
C_4 & 6
& 1&\sqrt2&0&1&0&-1&-\sqrt2&-1\\
C_5 & 6
& 1&-\sqrt2&0&1&0&-1&\sqrt2&-1\\
C_6 & 8
& 1&1&-1&0&-1&0&1&1\\
C_7 & 8
& 1&-1&-1&0&1&0&-1&1
\end{array}
\]
Here \(\chi_0\) is the trivial character and $\chi_1$ is the character of the natural two-dimensional representation $V=\rho_1$. The displayed
characters form the complete set of irreducible characters of
\(\mathcal O\); see also the standard character orthogonality relations
in \cite[Sec.~2.3]{Serre77}.

Tensoring all the irreducible representations by the natural representation gives
\[
V\otimes\rho_0\simeq\rho_1,
\]
\[
V\otimes\rho_1\simeq\rho_0\oplus\rho_3,
\qquad
V\otimes\rho_2\simeq\rho_4,
\]
\[
V\otimes\rho_3\simeq\rho_1\oplus\rho_4,
\]
\[
V\otimes\rho_4\simeq\rho_2\oplus\rho_3\oplus\rho_5,
\]
\[
V\otimes\rho_5\simeq\rho_4\oplus\rho_6,
\]
\[
V\otimes\rho_6\simeq\rho_5\oplus\rho_7,
\qquad
V\otimes\rho_7\simeq\rho_6.
\]
Thus the affine McKay diagram is of type \(\widetilde E_7\);
see \cite{McKay}.  We use the McKay labelling
\[
\rho_j\longleftrightarrow\alpha_j,
\qquad
1\le j\le7.
\]
Hence
\[
(d_1,\ldots,d_7)
=
(2,2,3,4,3,2,1).
\]

Let $\{ \epsilon_1,\ldots,\epsilon_8 \}$ 
be the standard orthonormal basis of \(\mathbb C^8\), equipped with
the standard complex bilinear form.  We use the simple roots
\[
\alpha_1
=
\frac12
\left(
\epsilon_1-\epsilon_2-\epsilon_3-\epsilon_4
-\epsilon_5-\epsilon_6-\epsilon_7+\epsilon_8
\right),
\]
\[
\alpha_2=\epsilon_1+\epsilon_2,
\qquad
\alpha_3=\epsilon_2-\epsilon_1,
\]
\[
\alpha_4=\epsilon_3-\epsilon_2,
\qquad
\alpha_5=\epsilon_4-\epsilon_3,
\]
\[
\alpha_6=\epsilon_5-\epsilon_4,
\qquad
\alpha_7=\epsilon_6-\epsilon_5.
\]
Their Gram matrix is the finite Cartan matrix
\[
C=
\bigl(\langle\alpha_i,\alpha_j\rangle\bigr)_{1\le i,j\le7}
=
\begin{pmatrix}
2&0&-1&0&0&0&0\\
0&2&0&-1&0&0&0\\
-1&0&2&-1&0&0&0\\
0&-1&-1&2&-1&0&0\\
0&0&0&-1&2&-1&0\\
0&0&0&0&-1&2&-1\\
0&0&0&0&0&-1&2
\end{pmatrix}.
\]

\medskip

\noindent\textbf{Roots and weighted height.}
We use the following standard coordinate realization of the \(E_7\)-root
system; cf.~\cite{Humphreys}:
\[
R=R_0\sqcup R_1,
\]
where
\[
R_0
=
\left\{
\pm\epsilon_a\pm\epsilon_b
\;\middle|\;
1\le a<b\le6
\right\}
\cup
\left\{
\pm(\epsilon_7-\epsilon_8)
\right\},
\]
and
\[
R_1
=
\left\{
\frac12\sum_{a=1}^8\delta_a\epsilon_a
\;\middle|\;
\begin{array}{c}
\delta_a\in\{\pm1\},\\
\delta_7=-\delta_8,\\
\delta_1\cdots\delta_6=-1
\end{array}
\right\}.
\]
Thus
\[
|R_0|=62,
\qquad
|R_1|=64,
\qquad
|R|=126,
\qquad
|R^+|=63.
\]

The weighted height introduced in
Subsection~\ref{subsec:exceptional-common} is
\[
h(\beta)
=
2b_1+2b_2+3b_3+4b_4+3b_5+2b_6+b_7
\]
for $\beta=\sum_{j=1}^7b_j\alpha_j$.

For this case, $\xi=e^{2\pi i/48}$,
 and, as a reminder,
\[
F(m)
=
\frac{1+\xi^m}{1-\xi^m}
=
i\cot\left(\frac{m\pi}{48}\right),
\qquad
m\not\equiv0\pmod{48}.
\]

For later use we record the simple-root coordinates explicitly.
Every vector in the \(E_7\)-root space has the form
\[
x=x_1\epsilon_1+\cdots+x_8\epsilon_8,
\qquad
x_7+x_8=0.
\]
Writing
\[
x=\sum_{j=1}^7b_j\alpha_j,
\]
one obtains
\[
b_1=2x_8,
\]
\[
b_2
=
\frac{
x_1+x_2+x_3+x_4+x_5+x_6+4x_8
}{2},
\]
\[
b_3
=
\frac{
-x_1+x_2+x_3+x_4+x_5+x_6+6x_8
}{2},
\]
\[
b_4=x_3+x_4+x_5+x_6+4x_8,
\]
\[
b_5=x_4+x_5+x_6+3x_8,
\]
\[
b_6=x_5+x_6+2x_8,
\qquad
b_7=x_6+x_8.
\]
For every root \(\beta\in R\), these numbers are integers, and
\[
\beta\in R^+
\quad\Longleftrightarrow\quad
b_j(\beta)\ge0,
\qquad
1\le j\le7.
\]

\medskip

\noindent\textbf{The Bryan--Gholampour matrix.}
From the character table,
\[
\chi_V(g_1)=-2,
\qquad
\chi_V(g_2)=\chi_V(g_3)=0,
\]
\[
\chi_V(g_4)=\sqrt2,
\qquad
\chi_V(g_5)=-\sqrt2,
\qquad
\chi_V(g_6)=1,
\qquad
\chi_V(g_7)=-1.
\]
We choose
\[
\sqrt{-4}=2i,
\qquad
\sqrt{-2}=\sqrt2\,i,
\]
\[
\sqrt{\sqrt2-2}
=
i\sqrt{2-\sqrt2},
\qquad
\sqrt{-\sqrt2-2}
=
i\sqrt{2+\sqrt2},
\]
and
\[
\sqrt{-1}=i,
\qquad
\sqrt{-3}=\sqrt3\,i.
\]

Thus, in the ordered bases
\[
\alpha_1,\ldots,\alpha_7
\qquad\text{and}\qquad
\mathbf1_{[1]},\ldots,\mathbf1_{[7]},
\]
the matrix \(M\) of the Bryan--Gholampour transformation is
\[
M
=
i
\begin{pmatrix}
-4&4&6&-8&6&-4&2\\
0&2\sqrt2&-\sqrt2&0&-\sqrt2&0&\sqrt2\\
0&0&-\sqrt2&0&\sqrt2&0&-\sqrt2\\
\sqrt{4-2\sqrt2}&0&\sqrt{2-\sqrt2}&0&
-\sqrt{2-\sqrt2}&-\sqrt{4-2\sqrt2}&-\sqrt{2-\sqrt2}\\
-\sqrt{4+2\sqrt2}&0&\sqrt{2+\sqrt2}&0&
-\sqrt{2+\sqrt2}&\sqrt{4+2\sqrt2}&-\sqrt{2+\sqrt2}\\
1&-1&0&-1&0&1&1\\
-\sqrt3&-\sqrt3&0&\sqrt3&0&-\sqrt3&\sqrt3
\end{pmatrix}.
\]
The matrices \(N^{(k)}\), \(1\le k\le7\), are determined from the
preceding character table by the Burnside formula of Remark~\ref{rmk_Burnside}.

\medskip

\noindent\textbf{Exact cyclotomic verification.}
The relevant cyclotomic field is
\[
\mathbb K
=
\mathbb Q(\xi)
\cong
\mathbb Q[X]/\bigl(\Phi_{48}(X)\bigr),
\]
where
\[
\Phi_{48}(X)=X^{16}-X^8+1.
\]
All quantities entering the matrices \(Q^{(k)}\) belong to
\(\mathbb K\).  Indeed,
\[
i=\xi^{12},
\qquad
\sqrt2=\xi^6+\xi^{-6},
\qquad
\sqrt3=\xi^4+\xi^{-4},
\]
\[
\sqrt{2-\sqrt2}
=
-i(\xi^3-\xi^{-3}),
\qquad
\sqrt{2+\sqrt2}
=
\xi^3+\xi^{-3},
\]
and we also have
\[
\sqrt{4-2\sqrt2}
=
\sqrt2\,\sqrt{2-\sqrt2},
\qquad
\sqrt{4+2\sqrt2}
=
\sqrt2\,\sqrt{2+\sqrt2}.
\]
Moreover,
\[
F(m)
=
\frac{1+\xi^m}{1-\xi^m}.
\]

The preceding coordinate realization gives all \(126\) roots
explicitly.  The \(63\) positive roots are obtained by retaining
precisely those roots for which
\[
b_j(\beta)\ge0,
\qquad
1\le j\le7.
\]
For every such root we compute
\[
b(\beta),
\qquad
h(\beta)
=
2b_1(\beta)+2b_2(\beta)+3b_3(\beta)+4b_4(\beta)
+3b_5(\beta)+2b_6(\beta)+b_7(\beta),
\]
\[
p(\beta)=Cb(\beta),
\qquad
Mb(\beta).
\]
Hence every entry of every \(Q^{(k)}\) is an explicit finite sum in
\(\mathbb K\).

There are
\[
7\binom82=196
\]
differences
\[
D_{ij}^{(k)}
=
Q_{ij}^{(k)}
-
\left(M^{\mathsf T}N^{(k)}M\right)_{ij},
\qquad
1\le i\le j\le7,
\quad
1\le k\le7.
\]
Replacing the radicals and every \(F(h(\beta))\) by the expressions
above in \(\xi\), clearing the nonzero denominators
\[
1-\xi^{h(\beta)},
\]
and reducing the resulting numerators modulo
\[
\Phi_{48}(X)=X^{16}-X^8+1
\]
gives zero for every triple
\[
(i,j,k),
\qquad
1\le i\le j\le7,
\quad
1\le k\le7.
\]
Therefore
\[
Q^{(k)}
=
M^{\mathsf T}N^{(k)}M,
\qquad
1\le k\le7.
\]

We record one representative entry.  Put
\[
c_m
:=
\cot\left(\frac{m\pi}{48}\right).
\]
For
\[
(i,j,k)=(1,6,1),
\]
grouping the contributing positive roots according to their weighted
height gives
\[
Q_{1,6}^{(1)}
=
-4c_{10}-8c_{12}+12c_{14}-8c_{18}
+8c_{20}-4c_{22}+4c_{26}.
\]
Since
\[
c_m
=
-i\,\frac{1+\xi^m}{1-\xi^m},
\]
reduction in \(\mathbb K\) gives
\[
Q_{1,6}^{(1)}
=
-24\sqrt2+16\sqrt3.
\]
On the Chen--Ruan side, the same exact computation using the Burnside
formula and the matrix \(M\) gives
\[
\left(M^{\mathsf T}N^{(1)}M\right)_{1,6}
=
-24\sqrt2+16\sqrt3.
\]
Thus this entry agrees, and the preceding exact cyclotomic reduction
proves all the remaining scalar identities.

Hence the identities of
Lemma~\ref{lmm:exceptional-matrix-criterion} hold in type \(E_7\).
By Proposition~\ref{prop:uniform-reduction}, this proves the \(E_7\) case of
Theorem~\ref{thm:exceptional}.

\subsection{Type \(E_8\)}
\label{subsec:E8}

We finally apply the common scheme of
Subsection~\ref{subsec:exceptional-common} to the binary icosahedral
group.

Let
\[
G=\mathcal J
\]
be the binary icosahedral group of order \(120\).  We regard
\(\mathcal J\) as the inverse image of the rotational icosahedral group
\[
A_5\subset\mathrm{SO}(3)
\]
under the double covering $\mathrm{SU}(2)\longrightarrow\mathrm{SO}(3)$. 
Thus $\mathcal J/\{\pm e\}\simeq A_5$. 

Set
\[
\phi:=\frac{1+\sqrt5}{2},
\qquad
\bar\phi:=\frac{1-\sqrt5}{2}.
\]
Then
\[
\phi+\bar\phi=1,
\qquad
\phi\bar\phi=-1,
\qquad
\phi^{-1}=-\bar\phi.
\]

We use the following ordering of the conjugacy classes.  Choose
representatives
\[
g_0=e,
\qquad
g_1=-e,
\qquad
g_2=
\begin{pmatrix}
i&0\\
0&-i
\end{pmatrix},
\]
\[
g_3
=
\frac12
\begin{pmatrix}
1+i&1+i\\
-1+i&1-i
\end{pmatrix},
\qquad
g_4=-g_3,
\]
\[
g_5
=
\frac12
\begin{pmatrix}
\phi+i&\phi^{-1}\\
-\phi^{-1}&\phi-i
\end{pmatrix},
\qquad
g_6=-g_5,
\]
and
\[
g_7
=
\frac12
\begin{pmatrix}
-\phi^{-1}+i&\phi\\
-\phi&-\phi^{-1}-i
\end{pmatrix},
\qquad
g_8=-g_7.
\]
Set
\[
C_r=\operatorname{Cl}(g_r),
\qquad
0\le r\le8.
\]
Their respective sizes are
\[
(|C_0|,\ldots,|C_8|)
=
(1,1,30,20,20,12,12,12,12).
\]
As before, we write
\[
[r]:=C_r,
\qquad
0\le r\le8.
\]

\medskip

\noindent\textbf{Character and McKay data.}
With the ordering above, the irreducible character table of
\(\mathcal J\) is
\[
\begin{array}{c|c|ccccccccc}
 & |C_r|
 & \chi_0 & \chi_1 & \chi_2 & \chi_3 & \chi_4
 & \chi_5 & \chi_6 & \chi_7 & \chi_8 \\ \hline
C_0 & 1
& 1&2&3&4&6&5&4&3&2\\
C_1 & 1
& 1&-2&3&4&-6&5&-4&3&-2\\
C_2 & 30
& 1&0&-1&0&0&1&0&-1&0\\
C_3 & 20
& 1&1&0&1&0&-1&-1&0&1\\
C_4 & 20
& 1&-1&0&1&0&-1&1&0&-1\\
C_5 & 12
& 1&\bar\phi&\bar\phi&-1&-1&0&1&\phi&\phi\\
C_6 & 12
& 1&-\bar\phi&\bar\phi&-1&1&0&-1&\phi&-\phi\\
C_7 & 12
& 1&\phi&\phi&-1&-1&0&1&\bar\phi&\bar\phi\\
C_8 & 12
& 1&-\phi&\phi&-1&1&0&-1&\bar\phi&-\bar\phi
\end{array}.
\]
Here \(\chi_0\) is the trivial character and $\chi_8$ is the character of the natural two-dimensional representation $V=\rho_8$. The displayed
characters form the complete set of irreducible characters of
\(\mathcal J\); see also the standard character orthogonality relations
in \cite[Sec.~2.3]{Serre77}.

Tensoring all the irreducible representations by the natural representation gives
\[
V\otimes\rho_0\simeq\rho_8,
\]
\[
V\otimes\rho_1\simeq\rho_3,
\qquad
V\otimes\rho_2\simeq\rho_4,
\]
\[
V\otimes\rho_3\simeq\rho_1\oplus\rho_4,
\]
\[
V\otimes\rho_4\simeq\rho_2\oplus\rho_3\oplus\rho_5,
\]
\[
V\otimes\rho_5\simeq\rho_4\oplus\rho_6,
\]
\[
V\otimes\rho_6\simeq\rho_5\oplus\rho_7,
\]
\[
V\otimes\rho_7\simeq\rho_6\oplus\rho_8,
\]
\[
V\otimes\rho_8\simeq\rho_0\oplus\rho_7.
\]
Thus the affine McKay diagram is of type \(\widetilde E_8\);
see \cite{McKay}.  We use the McKay labelling
\[
\rho_j\longleftrightarrow\alpha_j,
\qquad
1\le j\le8.
\]
Hence
\[
(d_1,\ldots,d_8)
=
(2,3,4,6,5,4,3,2).
\]

Let $\{ \epsilon_1,\ldots,\epsilon_8 \} $
be the standard orthonormal basis of \(\mathbb C^8\), equipped with
the standard complex bilinear form.  We use the simple roots
\[
\alpha_1
=
\frac12
\left(
\epsilon_1-\epsilon_2-\epsilon_3-\epsilon_4
-\epsilon_5-\epsilon_6-\epsilon_7+\epsilon_8
\right),
\]
\[
\alpha_2=\epsilon_1+\epsilon_2,
\qquad
\alpha_3=\epsilon_2-\epsilon_1,
\]
\[
\alpha_4=\epsilon_3-\epsilon_2,
\qquad
\alpha_5=\epsilon_4-\epsilon_3,
\]
\[
\alpha_6=\epsilon_5-\epsilon_4,
\qquad
\alpha_7=\epsilon_6-\epsilon_5,
\qquad
\alpha_8=\epsilon_7-\epsilon_6.
\]
Their Gram matrix is the finite Cartan matrix
\[
C
=
\bigl(\langle\alpha_i,\alpha_j\rangle\bigr)_{1\le i,j\le8}
=
\begin{pmatrix}
2&0&-1&0&0&0&0&0\\
0&2&0&-1&0&0&0&0\\
-1&0&2&-1&0&0&0&0\\
0&-1&-1&2&-1&0&0&0\\
0&0&0&-1&2&-1&0&0\\
0&0&0&0&-1&2&-1&0\\
0&0&0&0&0&-1&2&-1\\
0&0&0&0&0&0&-1&2
\end{pmatrix}.
\]

\medskip

\noindent\textbf{Roots and weighted height.}
We use the standard coordinate realization of the \(E_8\)-root system;
cf.~\cite{Humphreys}:
\[
R=R_0\sqcup R_1,
\]
where
\[
R_0
=
\left\{
\pm\epsilon_a\pm\epsilon_b
\;\middle|\;
1\le a<b\le8
\right\},
\]
and
\[
R_1
=
\left\{
\frac12\sum_{a=1}^8\delta_a\epsilon_a
\;\middle|\;
\begin{array}{c}
\delta_a\in\{\pm1\},\\
\displaystyle\prod_{a=1}^8\delta_a=1
\end{array}
\right\}.
\]
Thus
\[
|R_0|=112,
\qquad
|R_1|=128,
\qquad
|R|=240,
\qquad
|R^+|=120.
\]

The weighted height introduced in
Subsection~\ref{subsec:exceptional-common} is
\[
h(\beta)
=
2b_1+3b_2+4b_3+6b_4+5b_5+4b_6+3b_7+2b_8
\]
for $\beta=\sum_{j=1}^8b_j\alpha_j$.

For this case, $\xi=e^{2\pi i/120}$, 
and, as a reminder,
\[
F(m)
=
\frac{1+\xi^m}{1-\xi^m}
=
i\cot\left(\frac{m\pi}{120}\right),
\qquad
m\not\equiv0\pmod{120}.
\]

For later use we record the simple-root coordinates explicitly.  Let
\[
x=x_1\epsilon_1+\cdots+x_8\epsilon_8\in\mathbb C^8
\]
and write
\[
x=\sum_{j=1}^8b_j\alpha_j.
\]
Solving the simple-root equations gives
\[
b_1=2x_8,
\]
\[
b_2
=
\frac{
x_1+x_2+x_3+x_4+x_5+x_6+x_7+5x_8
}{2},
\]
\[
b_3
=
\frac{
-x_1+x_2+x_3+x_4+x_5+x_6+x_7+7x_8
}{2},
\]
\[
b_4=x_3+x_4+x_5+x_6+x_7+5x_8,
\]
\[
b_5=x_4+x_5+x_6+x_7+4x_8,
\]
\[
b_6=x_5+x_6+x_7+3x_8,
\]
\[
b_7=x_6+x_7+2x_8,
\qquad
b_8=x_7+x_8.
\]
For every root \(\beta\in R\), these numbers are integers, and
\[
\beta\in R^+
\quad\Longleftrightarrow\quad
b_j(\beta)\ge0,
\qquad
1\le j\le8.
\]

\medskip

\noindent\textbf{The Bryan--Gholampour matrix.}
Recall that
\[
V=\rho_8.
\]
From the character table,
\[
\chi_V(g_1)=-2,
\qquad
\chi_V(g_2)=0,
\]
\[
\chi_V(g_3)=1,
\qquad
\chi_V(g_4)=-1,
\]
\[
\chi_V(g_5)=\phi,
\qquad
\chi_V(g_6)=-\phi,
\]
and
\[
\chi_V(g_7)=\bar\phi,
\qquad
\chi_V(g_8)=-\bar\phi.
\]

For brevity, set
\[
\kappa_-:=\sqrt{2-\phi},
\qquad
\kappa_+:=\sqrt{2+\phi},
\]
and
\[
\mu_-:=\sqrt{2-\bar\phi},
\qquad
\mu_+:=\sqrt{2+\bar\phi},
\]
where the positive real square roots are chosen.  We choose
\[
\sqrt{-4}=2i,
\qquad
\sqrt{-2}=\sqrt2\,i,
\qquad
\sqrt{-1}=i,
\qquad
\sqrt{-3}=\sqrt3\,i,
\]
and
\[
\sqrt{\phi-2}=i\kappa_-,
\qquad
\sqrt{-\phi-2}=i\kappa_+,
\]
\[
\sqrt{\bar\phi-2}=i\mu_-,
\qquad
\sqrt{-\bar\phi-2}=i\mu_+.
\]

Thus, in the ordered bases
\[
\alpha_1,\ldots,\alpha_8
\qquad\text{and}\qquad
\mathbf1_{[1]},\ldots,\mathbf1_{[8]},
\]
the matrix \(M\) of the Bryan--Gholampour transformation is
\[
M
=
i
\begin{pmatrix}
-4&6&8&-12&10&-8&6&-4\\
0&-\sqrt2&0&0&\sqrt2&0&-\sqrt2&0\\
1&0&1&0&-1&-1&0&1\\
-\sqrt3&0&\sqrt3&0&-\sqrt3&\sqrt3&0&-\sqrt3\\
\bar\phi\kappa_-&\bar\phi\kappa_-&-\kappa_-&-\kappa_-&
0&\kappa_-&\phi\kappa_-&\phi\kappa_-\\
-\bar\phi\kappa_+&\bar\phi\kappa_+&-\kappa_+&\kappa_+&
0&-\kappa_+&\phi\kappa_+&-\phi\kappa_+\\
\phi\mu_-&\phi\mu_-&-\mu_-&-\mu_-&
0&\mu_-&\bar\phi\mu_-&\bar\phi\mu_-\\
-\phi\mu_+&\phi\mu_+&-\mu_+&\mu_+&
0&-\mu_+&\bar\phi\mu_+&-\bar\phi\mu_+
\end{pmatrix}.
\]
The matrices \(N^{(k)}\), \(1\le k\le8\), are determined from the
preceding character table by the Burnside formula of Remark~\ref{rmk_Burnside}.

\medskip

\noindent\textbf{Exact cyclotomic verification.}
The relevant cyclotomic field is
\[
\mathbb K
=
\mathbb Q(\xi)
\cong
\mathbb Q[X]/\bigl(\Phi_{120}(X)\bigr),
\]
where
\[
\Phi_{120}(X)
=
X^{32}+X^{28}-X^{20}-X^{16}-X^{12}+X^4+1.
\]
All quantities entering the matrices \(Q^{(k)}\) belong to
\(\mathbb K\).  Indeed,
\[
i=\xi^{30},
\qquad
\sqrt2=\xi^{15}+\xi^{-15},
\qquad
\sqrt3=\xi^{10}+\xi^{-10},
\]
and
\[
\phi=\xi^{12}+\xi^{-12},
\qquad
\bar\phi=1-\phi.
\]
Furthermore,
\[
\kappa_-
=
\sqrt{2-\phi}
=
\phi-1,
\]
\[
\kappa_+
=
\sqrt{2+\phi}
=
\xi^6+\xi^{-6},
\]
\[
\mu_-
=
\sqrt{2-\bar\phi}
=
\phi,
\]
and
\[
\mu_+
=
\sqrt{2+\bar\phi}
=
-i(\xi^{12}-\xi^{-12}).
\]
Thus every entry of \(M\) belongs to \(\mathbb K\).  Moreover,
\[
F(m)
=
\frac{1+\xi^m}{1-\xi^m}.
\]

The preceding coordinate realization gives all \(240\) roots
explicitly.  The \(120\) positive roots are obtained by retaining
precisely those roots for which
\[
b_j(\beta)\ge0,
\qquad
1\le j\le8.
\]
For every such root we compute
\[
b(\beta),
\]
\[
h(\beta)
=
2b_1(\beta)+3b_2(\beta)+4b_3(\beta)+6b_4(\beta)
+5b_5(\beta)+4b_6(\beta)+3b_7(\beta)+2b_8(\beta),
\]
\[
p(\beta)=Cb(\beta),
\qquad
Mb(\beta).
\]
Hence every entry of every \(Q^{(k)}\) is an explicit finite sum in
\(\mathbb K\).

There are
\[
8\binom92=288
\]
differences
\[
D_{ij}^{(k)}
=
Q_{ij}^{(k)}
-
\left(M^{\mathsf T}N^{(k)}M\right)_{ij},
\qquad
1\le i\le j\le8,
\quad
1\le k\le8.
\]
Replacing every radical and every \(F(h(\beta))\) by the expressions
above in \(\xi\), clearing the nonzero denominators
\[
1-\xi^{h(\beta)},
\]
and reducing the resulting numerators modulo
\[
\Phi_{120}(X)
=
X^{32}+X^{28}-X^{20}-X^{16}-X^{12}+X^4+1
\]
gives zero for every triple
\[
(i,j,k),
\qquad
1\le i\le j\le8,
\quad
1\le k\le8.
\]
Therefore
\[
Q^{(k)}
=
M^{\mathsf T}N^{(k)}M,
\qquad
1\le k\le8.
\]

We record one non-trivial entry to illustrate the exact calculation.
Put
\[
c_m
:=
\cot\left(\frac{m\pi}{120}\right).
\]
For
\[
(i,j,k)=(1,4,2),
\]
grouping the contributing positive roots according to their weighted
height gives
\[
Q_{1,4}^{(2)}
=
\sqrt2
\left(
c_{25}-c_{29}-c_{31}+c_{35}
\right).
\]
On the Chen--Ruan side, direct multiplication of \(M\) and \(N^{(2)}\)
gives
\[
\left(M^{\mathsf T}N^{(2)}M\right)_{1,4}
=
-2\left[
(\kappa_--\kappa_+)(\mu_+-\mu_-)
+
(\sqrt3-1)
(\kappa_--\kappa_++\mu_--\mu_+)
\right].
\]
Writing both expressions in terms of \(\xi\), using
\[
c_m
=
-i\,\frac{1+\xi^m}{1-\xi^m},
\]
and the formulas for
\[
\kappa_-,
\quad
\kappa_+,
\quad
\mu_-,
\quad
\mu_+
\]
above, their difference has zero remainder modulo
\(\Phi_{120}(X)\).  Hence
\[
Q_{1,4}^{(2)}
=
\left(M^{\mathsf T}N^{(2)}M\right)_{1,4}.
\]
The same exact reduction gives all the remaining scalar identities.

Hence the identities of
Lemma~\ref{lmm:exceptional-matrix-criterion} hold in type \(E_8\).
By Proposition~\ref{prop:uniform-reduction}, this proves the \(E_8\) case of
Theorem~\ref{thm:exceptional}.

Together with the \(E_6\) and \(E_7\) cases proved above, this completes
the proof of Theorem~7.1.

\appendix

\section{A finite Fourier--cotangent identity}
\label{app:fourier-cotangent}
%----------------------------------------------------------
%\subsection{A finite Fourier--cotangent identity}
%----------------------------------------------------------
For an integer $n\ge 2$, let
\[
\zeta=e^{2\pi i/n},
\qquad
\theta=\frac{\pi}{n}.
\]
First of all, let us record the elementary identity
\[
\frac{1+\zeta^m}{1-\zeta^m}
=
i\cot(m\theta),
\qquad
1\le m\le n-1.
\]

The following classical finite Fourier–cotangent identity is used in the proofs of the cyclic and binary dihedral cases. 
For more general root-of-unity weighted trigonometric sums and further references, see \cite{LiuXin}. 
We include a short elementary proof for completeness.

\begin{lmm}\label{lem:cot-fourier}
With $\zeta$ and $\theta$ as above, for every \(1\leq r\leq n-1\),
\[
\sum_{m=1}^{n-1}
\zeta^{rm}\cot(m\theta)
=
i(n-2r).
\]
\end{lmm}

\begin{proof}
Write
\[
\zeta^{rm}
=
\cos(2rm\theta)
+i\sin(2rm\theta).
\]

We first consider the real part. Pair the term indexed by \(m\)
with the term indexed by \(n-m\). Since
\[
\cos(2r(n-m)\theta)
=
\cos(2rm\theta)
\]
and
\[
\cot((n-m)\theta)
=
-\cot(m\theta),
\]
the two terms cancel. When \(n\) is even, the only term paired
with itself is \(m=n/2\), and this term is zero because
\[
\cot\left(\frac{\pi}{2}\right)=0.
\]
Therefore
\[
\sum_{m=1}^{n-1}
\cos(2rm\theta)\cot(m\theta)
=
0.
\]

We next compute the imaginary part. We first prove the elementary
identity
\begin{equation}\label{eq:trig-cot}
\sin(2rx)\cot x
=
1+\cos(2rx)
+
2\sum_{k=1}^{r-1}\cos(2kx),
\end{equation}
for \(\sin x\neq0\).

Indeed,
\begin{align*}
1+\cos(2rx)
+2\sum_{k=1}^{r-1}\cos(2kx)
&=
\sum_{k=0}^{r-1}
\left(
\cos(2kx)+\cos(2(k+1)x)
\right)\\
&=
2\cos x
\sum_{k=0}^{r-1}\cos((2k+1)x),
\end{align*}
where we used
\[
\cos A+\cos B
=
2\cos\left(\frac{A+B}{2}\right)
\cos\left(\frac{A-B}{2}\right).
\]
Furthermore,
\begin{align*}
2\sin x
\sum_{k=0}^{r-1}\cos((2k+1)x)
&=
\sum_{k=0}^{r-1}
\left(
\sin(2(k+1)x)-\sin(2kx)
\right)\\
&=
\sin(2rx).
\end{align*}
Thus
\[
\sum_{k=0}^{r-1}\cos((2k+1)x)
=
\frac{\sin(2rx)}{2\sin x},
\]
and substituting this into the previous expression proves
\eqref{eq:trig-cot}.

Applying \eqref{eq:trig-cot} with
\[
x=m\theta,
\]
which is allowed because
\[
\sin(m\theta)\neq0,
\qquad
1\leq m\leq n-1,
\]
we obtain
\begin{align*}
&\sum_{m=1}^{n-1}
\sin(2rm\theta)\cot(m\theta)
\\
&\qquad=
(n-1)
+
\sum_{m=1}^{n-1}\cos(2rm\theta)
+
2\sum_{k=1}^{r-1}
\sum_{m=1}^{n-1}\cos(2km\theta).
\end{align*}

For every \(1\leq k\leq n-1\), $
\sum\limits_{m=0}^{n-1}
e^{2\pi i km/n}=0.
$
Taking real parts and removing the \(m=0\) term yields
\[
\sum_{m=1}^{n-1}
\cos\left(\frac{2\pi km}{n}\right)
=
-1.
\]
Hence
\begin{align*}
\sum_{m=1}^{n-1}
\sin(2rm\theta)\cot(m\theta)
&=
(n-1)-1-2(r-1)\\
&=
n-2r.
\end{align*}

The real part of the Fourier sum is zero and its imaginary part
is \(n-2r\). Therefore
\[
\sum_{m=1}^{n-1}
\zeta^{rm}\cot(m\theta)
=
i(n-2r).
\]
\end{proof}

We now record the precise consequence of
Lemma~\ref{lem:cot-fourier} that  occurs in the quantum
product calculation.

\begin{cor}\label{cor:T-abc}
Let
\[
1\leq a,b\leq n-1,
\qquad
a+b\not\equiv0\pmod n,
\]
and let \(c\in\{1,\ldots,n-1\}\) be uniquely determined by
\[
c\equiv a+b\pmod n.
\]
Define
\[
T_{a,b,c}
:=
\sum_{m=1}^{n-1}
(1-\zeta^{-am})(1-\zeta^{-bm})(1-\zeta^{cm})
\cot(m\theta).
\]
Then
\[
T_{a,b,c}
=
\begin{cases}
2ni, & a+b<n,\\[2mm]
-2ni, & a+b>n.
\end{cases}
\]
\end{cor}

\begin{proof}
Expanding the product gives
\begin{align*}
&(1-\zeta^{-am})(1-\zeta^{-bm})(1-\zeta^{cm})
\\
&\quad=
1-\zeta^{-am}-\zeta^{-bm}
+\zeta^{-(a+b)m}
-\zeta^{cm}
+\zeta^{(c-a)m}
+\zeta^{(c-b)m}
-\zeta^{(c-a-b)m}.
\end{align*}
Since
\[
c\equiv a+b\pmod n,
\]
we have
\[
\zeta^{-(a+b)m}=\zeta^{-cm},
\]
\[
\zeta^{(c-a)m}=\zeta^{bm},
\qquad
\zeta^{(c-b)m}=\zeta^{am},
\]
and
\[
\zeta^{(c-a-b)m}=1.
\]
Therefore the first and last terms cancel, and
\[
(1-\zeta^{-am})(1-\zeta^{-bm})(1-\zeta^{cm})
=
\zeta^{-cm}
-\zeta^{-am}
-\zeta^{-bm}
-\zeta^{cm}
+\zeta^{am}
+\zeta^{bm}.
\]

For \(1\leq r\leq n-1\), set
\[
U_r
:=
\sum_{m=1}^{n-1}
\zeta^{rm}\cot(m\theta).
\]
By Lemma~\ref{lem:cot-fourier},
\[
U_r=i(n-2r).
\]
Moreover,
\[
U_{n-r}
=
i(n-2(n-r))
=
-i(n-2r)
=
-U_r.
\]

Consequently,
\begin{align*}
T_{a,b,c}
&=
U_{n-c}
-U_{n-a}
-U_{n-b}
-U_c
+U_a
+U_b\\
&=
-2U_c+2U_a+2U_b.
\end{align*}

If \(a+b<n\), then
\[
c=a+b,
\]
and hence
\begin{align*}
T_{a,b,c}
&=
2i
\left(
-(n-2c)+(n-2a)+(n-2b)
\right)\\
&=
2ni.
\end{align*}

If \(a+b>n\), then
\[
c=a+b-n.
\]
Therefore
\begin{align*}
T_{a,b,c}
&=
2i
\left(
-(n-2c)+(n-2a)+(n-2b)
\right)\\
&=
-2ni.
\end{align*}
This proves the claim.
\end{proof}

\section{Proof of Lemma \ref{lem:D-Fourier-identities}}\label{Appendix B}

For $x=m\theta'$, by definition, 
\begin{align*}
V_m = \sum_{a=1}^{2n-1} \sin(ax)z^a = \sum_{a=1}^{2n-1} \frac{e^{iax} - e^{-iax}}{2i} z^a \, .
\end{align*}
Summing these finite geometric series and using the identity 
$(e^{\pm ix}z)^{2n} =e^{\pm im \pi}z^{2n} =-1 \in \cR$ (recall that $m$ is odd), we obtain
\begin{equation}\label{eq:D-Vm-rational}
V_m
=
\frac{2\sin(x)z}
     {1-2\cos(x)z+z^2}
\qquad
\text{in }\mathcal R.
\end{equation}
The denominator is invertible in \(\mathcal R\), since its roots
\(e^{\pm ix}\) satisfy $(e^{\pm ix})^{2n} =-1$ and hence they are not roots of $z^{2n} -1$.
Evaluation at \(z=\pm1\) gives (i)
\[
V_m(1)= \frac{\sin x}{1-\cos x} = \cot \frac{x}{2} \qquad \mbox{and} \qquad V_m(-1) = -\frac{\sin x}{1+\cos x} = -\tan \frac{x}{2} \, .
\]

For \(m\ne\ell\), a partial-fraction decomposition of the
product of the two expressions
\eqref{eq:D-Vm-rational} gives
\[
V_mV_\ell=AV_m+BV_\ell,
\]
where
\[
A
=
\frac{\sin(\ell\theta')}
     {\cos(m\theta')-\cos(\ell\theta')},
\qquad
B
=
-\frac{\sin(m\theta')}
      {\cos(m\theta')-\cos(\ell\theta')}.
\]
The standard sum-to-product identities give exactly
\eqref{eq:D-VmVl}.

We prove (iii). Let 
\[
f_a := \sin (ax) \, , \qquad x=m\theta' = \frac{m\pi}{2n} \, .
\]
Since $V_m = \sum_{a=0}^{2n-1}f_az^a$ ($f_0=0$), in $\cR =\C[z]/(z^{2n}-1)$ the coefficient of $z^r$ in $V_m^2$ is
\[
c_r = \sum_{a=0}^{2n-1}f_af_{\overline{r-a}} \, , \qquad 0\le \overline{r-a} \le 2n-1 \, , \, \overline{r-a} \equiv r-a \,  ({\rm mod} \, 2n) \, .
\]

Now, if $0\le a \le r$, $\overline{r-a} = r-a$ and $f_{\overline{r-a}}= \sin ((r-a)x)$. If $r<a\le 2n-1$, $\overline{r-a} = 2n +r-a$.
Since \(2nx=m\pi\) and \(m\) is odd, $f_{\overline{r-a}} =\sin((r-a+2n)x) =-\sin((r-a)x)$.
It follows from this that
\[
c_r = \sum_{a=0}^{r} \sin(ax)\sin ((r-a)x) - \sum_{a=r+1}^{2n-1} \sin(ax)\sin ((r-a)x) \, .
\]
In other words, for 
\[
A_r := \sum_{a=0}^{r} \sin(ax)\sin ((r-a)x) \, , \qquad \mbox{and} \qquad B_r:= \sum_{a=r+1}^{2n-1} \sin(ax)\sin ((r-a)x) \, ,
\]
$c_r= A_r - B_r$. On the other hand, 
\[
A_r + B_r = \sum_{a=0}^{2n-1} \sin(ax)\sin ((r-a)x) = -n\cos (rx) \, ,
\]
where we have used the identities $2\sin (ax)\sin ((r-a)x) = \cos ((2a-r)x) - \cos (rx)$
and $\sum_{a=0}^{2n-1}\cos ((2a-r)x) =0$ (note that $e^{2ix}$ is a $2n$-th root of  unity distinct from $1$).
Therefore
\[
c_r = A_r - B_r = 2A_r- (A_r + B_r) = 2A_r + n\cos (rx) \, .
\]
Furthermore, using the product-to-sum formula, 
\[
2A_r = \sum_{a=0}^r [\cos ((2a-r)x) - \cos (rx)] = \sum_{a=0}^r \cos ((2a-r)x) - (r+1) \cos (rx) \, .
\]
Therefore
\[
c_r = (n-r-1)\cos (rx) +\sum_{a=0}^r \cos ((2a-r)x) \, .
\]

For the last sum, consider the geometric series
\begin{align*}
\sum_{a=0}^r e^{i(2a-r)x} & = e^{-irx}\frac{1-e^{i2(r+1)x}}{1-e^{i2x}} =e^{-irx} \frac{-2ie^{i(r+1)x}\sin((r+1)x)}{-2ie^{ix}\sin x} \\
&= \frac{\sin((r+1)x)}{\sin x} = \frac{\sin (rx)\cos x + \cos (rx) \sin x}{\sin x}  \\
&= \sin (rx) \cot x + \cos (rx) \, , 
\end{align*}
where we have used the identity $1-e^{2iy} =-2ie^{iy} \sin y$.

Consequently,
\begin{equation}\label{eq:D-cr}
c_r
=
(n-r)\cos(mr\theta')
+
\cot(m\theta')\sin(mr\theta'),
\qquad
0\le r\le n.
\end{equation}
In particular $c_0=n$. 

As observed before, the vectors $V_m$, for $m\in I_n$, form a basis of 
$\operatorname{Span}_{\mathbb C} \{E_1,\ldots,E_n\}$. 
Since \(V_m^2\in \mathcal R^\iota\), we may therefore write uniquely
\[
V_m^2
=
n+\sum_{j\in I_n}a_j V_j.
\]
By definition, $V_j = \sum_{r=1}^{2n-1}\sin (j r \theta')z^r = \sum_{r=1}^{n-1}\sin (j r \theta')(z^r +z^{-r}) + \sin (j n \theta') z^n$,
where  we have used the symmetry $\sin(j(2n-r)\theta')=\sin(j r\theta')$ observed before.
So the previous expression becomes
\[
V_m^2 =  n+\sum_{j\in I_n}a_j \sum_{r=1}^{n-1}\sin (j r \theta')(z^r + z^{-r} ) + \sum_{j\in I_n}a_j \sin (j n \theta') z^n  \, .
\]
Comparing the coefficients of \(z^r\), for \(1\le r\le n\), gives
\begin{equation*}
c_r = \sum_{j\in I_n}a_j \sin (j r \theta') \, .
\end{equation*}
Fix \(\ell\in I_n\). Multiplying these identities by \(\sin(\ell r\theta')\) for \(1\le r\le n-1\), and the identity for \(r=n\) 
by \(\frac12\sin(\ell n\theta')\), and summing, we obtain
\begin{align*}
&\sum_{r=1}^{n-1}c_r\sin (\ell r\theta') + c_n \frac{1}{2}\sin (\ell n\theta') \\
&= \sum_{r=1}^{n-1}\sum_{j \in I_n}a_j \sin (j r \theta')\sin (\ell r\theta')+ 
\sum_{j \in I_n}a_j \sin (j n \theta')\frac{1}{2}\sin(\ell n\theta') \\
& = \sum_{j \in I_n} a_j \left[ \sum_{r=1}^{n-1} \sin (j r \theta')\sin (\ell r\theta') + \frac{1}{2}\sin (j n \theta')\sin(\ell n\theta') \right] \\
& = \sum_{j \in I_n} a_j \frac{n}{2}\delta_{\ell j} = \frac{n}{2} a_\ell \, ,
\end{align*} 
where in the second-last equality we have used  \eqref{eq:D-sine-orthogonality}. 
Hence
\[
a_\ell
=
\frac2n
\left[
\sum_{r=1}^{n-1}
c_r\sin(\ell r\theta')
+
\frac12c_n\sin(\ell n\theta')
\right].
\]
Substituting \eqref{eq:D-cr} into the expression for $a_\ell$, we obtain
\[
\begin{aligned}
a_\ell
={}&
\frac{2}{n}
\sum_{r=1}^{n-1}
(n-r)\cos(mr\theta')\sin(\ell r\theta')
\\
&\quad
+\frac{2}{n}\cot(m\theta')
\left[
\sum_{r=1}^{n-1}
\sin(mr\theta')\sin(\ell r\theta')
+
\frac12
\sin(mn\theta')\sin(\ell n\theta')
\right].
\end{aligned}
\]

By \eqref{eq:D-sine-orthogonality}, the expression in square brackets is $\frac n2\delta_{m\ell}$.
Hence the second term contributes $\cot(m\theta')\delta_{m\ell}$. 

For the first term, we use
\[
\cos(mr\theta')\sin(\ell r\theta')
=
\frac12
\left[
\sin((\ell+m)r\theta')
+
\sin((\ell-m)r\theta')
\right].
\]
Therefore
\[
\begin{aligned}
&\frac{2}{n}
\sum_{r=1}^{n-1}
(n-r)\cos(mr\theta')\sin(\ell r\theta')
\\
&\qquad =
\frac1n
\sum_{r=1}^{n-1}
(n-r)\sin((\ell+m)r\theta')
+
\frac1n
\sum_{r=1}^{n-1}
(n-r)\sin((\ell-m)r\theta').
\end{aligned}
\]

Since $m,\ell\in I_n$, both $m$ and $\ell$ are odd, and hence
\[
s_+:=\frac{\ell+m}{2},
\qquad
s_-:=\frac{\ell-m}{2}
\]
are integers. Moreover,
\[
(\ell\pm m)r\theta'
=
\frac{\ell\pm m}{2}\frac{\pi r}{n}
=
\frac{s_\pm\pi r}{n}.
\]
We can therefore apply \eqref{eq:D-weighted-sine} with $N=n$.

Suppose first that $\ell\neq m$. Then $s_-\neq0$, and neither
$s_+$ nor $s_-$ is congruent to $0$ modulo $2n$. Hence \eqref{eq:D-weighted-sine} gives
\[
\sum_{r=1}^{n-1}
(n-r)\sin((\ell+m)r\theta')
=
\frac n2
\cot\left(\frac{(\ell+m)\theta'}2\right)
\]
and
\[
\sum_{r=1}^{n-1}
(n-r)\sin((\ell-m)r\theta')
=
\frac n2
\cot\left(\frac{(\ell-m)\theta'}2\right).
\]
Since in this case $\delta_{m\ell}=0$, we obtain
\[
a_\ell
=
\frac12
\left[
\cot\left(\frac{(\ell+m)\theta'}2\right)
+
\cot\left(\frac{(\ell-m)\theta'}2\right)
\right],
\qquad \ell\neq m.
\]

If instead $\ell=m$, then $s_-=0$, so the second sine sum vanishes
identically. For the first one, \eqref{eq:D-weighted-sine} gives
\[
\sum_{r=1}^{n-1}
(n-r)\sin(2mr\theta')
=
\frac n2\cot(m\theta').
\]
Therefore the first contribution to $a_m$ is $\frac12\cot(m\theta')$. 
The contribution coming from \eqref{eq:D-sine-orthogonality} is instead $\cot(m\theta')$,
and consequently
\[
a_m
=
\frac12\cot(m\theta')
+
\cot(m\theta')
=
\frac32\cot(m\theta').
\]
This proves (iii).

We now prove (iv). Recall that $V_m=\sum_{r=1}^n\sin(mr\theta')\,E_{r}$. 
Thus it is enough to compute, for every $1\le r\le n$, the sums
\[
\sum_{m\in I_n}
\cot\left(\frac{m\theta'}2\right)\sin(mr\theta') \qquad \mbox{and} \qquad
\sum_{m\in I_n}
\tan\left(\frac{m\theta'}2\right)\sin(mr\theta').
\]

For the first one, apply \eqref{eq:D-cotangent-sum} with $N=2n$. Since
$\theta'=\pi/(2n)$, we obtain
\[
\sum_{s=1}^{2n-1}
\cot\left(\frac{s\theta'}2\right)
\sin(sr\theta')
=
2n-r.
\]
We split this sum into its even and odd parts. The contribution of the
even indices $s=2k$, $1\le k\le n-1$, is
\[
\sum_{k=1}^{n-1}
\cot\left(\frac{k\pi}{2n}\right)
\sin\left(\frac{kr\pi}{n}\right).
\]
Applying \eqref{eq:D-cotangent-sum} again, now with $N=n$, this is equal to $n-r$.
Hence the contribution of the odd indices is
\[
(2n-r)-(n-r)=n.
\]
Since the odd integers between $1$ and $2n-1$ are precisely the
elements of $I_n$, it follows that
\[
\sum_{m\in I_n}
\cot\left(\frac{m\theta'}2\right)
\sin(mr\theta')
=
n,
\qquad 1\le r\le n.
\]
Therefore
\[
\sum_{m\in I_n}
\cot\left(\frac{m\theta'}2\right)V_m
=
n\sum_{r=1}^n E_r.
\]

For the second identity, we use
\[
\tan\left(\frac{m\theta'}2\right)
=
\cot\left(\frac{(2n-m)\theta'}2\right),
\]
because
\[
\frac{(2n-m)\theta'}2
=
\frac{\pi}{2}-\frac{m\theta'}2.
\]
Moreover, the map
\[
m\longmapsto 2n-m
\]
permutes the set $I_n$. Hence
\[
\sum_{m\in I_n}
\tan\left(\frac{m\theta'}2\right)
\sin(mr\theta')
 =
\sum_{m\in I_n}
\cot\left(\frac{(2n-m)\theta'}2\right)
\sin(mr\theta') \, ,
\]
after the change of variable $j=2n-m$,  becomes
\[
\sum_{j\in I_n}
\cot\left(\frac{j\theta'}2\right)
\sin((2n-j)r\theta').
\]
Now
\[
\sin((2n-j)r\theta')
=
\sin(r\pi-jr\theta')
=
-(-1)^r\sin(jr\theta').
\]
Therefore
\[
\sum_{m\in I_n}
\tan\left(\frac{m\theta'}2\right)
\sin(mr\theta')
=
-(-1)^r
\sum_{j\in I_n}
\cot\left(\frac{j\theta'}2\right)
\sin(jr\theta')
=
-n(-1)^r,
\]
where in the last equality we have used the first identity proved above.
Consequently,
\[
\sum_{m\in I_n}
\tan\left(\frac{m\theta'}2\right)V_m
=
-n\sum_{r=1}^n(-1)^rE_r.
\]
This proves (iv).

\bibliographystyle{amsalpha}
\bibliography{bib}

\end{document}